\documentclass[a4paper,reqno, 11pt]{amsart}

\usepackage{amssymb,amsmath}
\usepackage{amsopn}
\usepackage{mathrsfs}
\usepackage{color}
\usepackage{mathtools}
\usepackage{wasysym}
\usepackage{vmargin}
\usepackage{pdfsync}
\usepackage{vmargin}
\usepackage{color}
\usepackage{amscd}
\usepackage{verbatim}
\usepackage{bbm}
\usepackage[capitalise]{cleveref}

\usepackage[usenames, dvipsnames]{xcolor}

\usepackage{tikz}
\usepackage{xspace}

\theoremstyle{plain} 
\newtheorem{prop}{Proposition}[section] 
\newtheorem{thm}[prop]{Theorem}

\newtheorem{lem}[prop]{Lemma}
\newtheorem{cor}[prop]{Corollary}
\theoremstyle{definition}
\newtheorem{defi}[prop]{Definition}
\newtheorem{rem}[prop]{Remark}

\newtheorem{c-ex}[prop]{Contre-exemple}

\newcommand{\ioe}{\leqslant}
\newcommand{\soe}{\geqslant}
\newcommand{\wt}{\widetilde}

\newcommand{\1}{\mathrm{1~\hspace{-1.4ex}l}}

\def\llambda{|b|^{\frac{4}{41}}}

\DeclareMathOperator{\ad}{ad}

\DeclareMathOperator{\Ran}{Ran}

\DeclareMathOperator{\Span}{Span}
\DeclareMathOperator{\Id}{Id}
\DeclareMathOperator{\sign}{sign}
\DeclareMathOperator{\dom}{Dom}
\DeclareMathOperator{\Quad}{quad}
\DeclareMathOperator{\Cub}{cub}
\DeclareMathOperator{\reg}{reg}
\DeclareMathOperator{\lin}{lin}
\DeclareMathOperator{\J}{\scriptscriptstyle J}
\DeclareMathOperator{\supp}{supp}
\DeclareMathOperator{\Rref}{ref}

\def\R{\mathbb{R}}
\def\C{\mathbb{C}}
\def\N{\mathbb{N}}
\def\P{\mathbb{P}}
\def\Z{\mathbb{Z}}

\def\eps{\varepsilon}

\def \F{\mathcal{F}}
\def \S{\mathcal{S}}
\def \H{\mathcal{H}}
\def \G{\mathcal{F}}

\def \O{\mathcal{O}}
\def \M{\mathcal{M}}
\def \U{\mathcal{U}}
\def \E{\mathcal{E}}
\def \V{\mathcal{V}}
\def \tild{\widetilde}
\def \Cspace{\mathcal{C}}

\title[STLC of the Schrödinger equation despite a drift, thanks to a cubic term]{Small-time local controllability of the bilinear Schrödinger equation, despite a quadratic obstruction, thanks to a cubic term}
\date{}
\author[Mégane Bournissou]{Mégane Bournissou$^*$}
\thanks{$^*$Univ Rennes, CNRS, IRMAR - UMR 6625, F-35000 Rennes, France.}
\thanks{This work benefits from the support of ANR project TRECOS,
grant ANR-20-CE40-0009.}
\email{megane.bournissou@ens-rennes.fr}
\begin{document}

\begin{abstract}
We consider a 1D linear Schrödinger equation, on a bounded interval, with Dirichlet boundary conditions and bilinear control. We study its controllability around the ground state when the linearized system is not controllable. More precisely, we study to what extent the nonlinear terms of the expansion can recover the directions lost at the first order. 

In the works \cite{BM14, B21bis}, for any positive integer $n$, assumptions have been formulated under which the quadratic term induces a drift in the nonlinear dynamics, quantified by the $H^{-n}$-norm of the control. This drift is an obstruction to the small-time local controllability (STLC) under a smallness assumption on the controls in regular spaces.

In this paper, we prove that for controls small in less regular spaces, the cubic term can recover the controllability lost at the linear level, despite the quadratic drift. 
The proof is inspired by Sussman's method to prove the sufficiency of the $\S(\theta)$ condition for STLC of ODEs. However, it uses a different global strategy relying on a new concept of tangent vector, better adapted to the infinite-dimensional setting of PDEs. 
%
%
%
Given a target, we first realize the expected motion along the lost direction by using control variations for which the cubic term dominates the quadratic one. Then, we correct the other components exactly, by using the STLC in projection result of \cite{B21}, with simultaneous estimates of weak norms of the control. These estimates ensure that the new error along the lost direction is negligible, and we conclude with the Brouwer fixed point theorem. 
\end{abstract}

\maketitle


\section{Introduction}

\subsection{Description of the control system}
Let $T>0$. In this paper, we consider the 1D linear Schrödinger equation given by, 
\begin{equation}  
\label{Schrodinger} \left\{
    \begin{array}{ll}
        i \partial_t \psi(t,x) = - \partial^2_x \psi(t,x) -u(t)\mu(x)\psi(t,x),  \quad &(t,x) \in (0,T) \times (0,1),\\
        \psi(t,0) = \psi(t,1)=0, \quad &t \in (0,T).
    \end{array}
\right.  \end{equation}
This equation is used in quantum physics to describe a quantum particle stuck in an infinite potential well and subjected to a uniform electric field whose amplitude is given by $u(t)$. The function $\mu : (0,1) \rightarrow \mathbb{R}$ depicts the dipolar moment of the particle. This equation is a bilinear control system where the state is the wave function $\psi$ such that for all time $\| \psi(t) \|_{L^2(0,1)} = 1$ and $u : (0,T) \rightarrow \mathbb{R}$ denotes a scalar control. 

\subsection{Functional settings}
Unless otherwise specified, in space, we will work with complex valued functions. The Lebesgue space $L^2(0,1)$ is equipped with the hermitian scalar product given by
$$\langle f,g\rangle := \int_0^1 f(x) \overline{g(x)}dx, \quad \forall f, g \in L^2(0,1).$$ Let $\mathcal{S}$ be the unit-sphere of $L^2(0,1)$.
The operator $A$ is defined by
\begin{equation*}
\dom(A):=H^2(0,1) \cap H^1_0(0,1) 
\quad
\text{ and }
\quad
A\varphi:=-\frac{d^2 \varphi}{dx^2}.
\end{equation*}
Its eigenvalues and eigenvectors are respectively given by
\begin{equation*}
\forall j \in \N^*, \quad \lambda_j:= (j \pi)^2 \quad \text{ and } \quad \varphi_j:=\sqrt{2} \sin(j \pi \cdot).
\end{equation*}
The family of eigenvectors $(\varphi_j)_{j \in \N^*}$ is an orthonormal basis of $L^2(0,1)$. We denote by,  
$$\forall j \in \N^*, \quad \psi_j(t,x):= \varphi_j(x) e^{-i \lambda_j t}, \quad \forall (t,x) \in \R \times [0,1],$$ 
the solutions of the Schrödinger equation \eqref{Schrodinger} with $u \equiv 0$ and initial data $\varphi_j$ at time $t=0$. When $j=1$, $\psi_1$ is called the ground state. We also introduce the normed spaces linked to the operator $A$, given by, for all $s \soe 0$, 
\begin{multline*}
H^s_{(0)}(0,1):=\dom(A^{\frac{s}{2}}) 
\\
\text{ endowed with the norm }
\quad 
\left\|
\varphi
\right\|_{H^s_{(0)}(0,1)} 
:=
\| ( \langle \varphi, \varphi_j \rangle)\|_{h^s(\N^*)}
:=
\left( 
\sum \limits_{j=1}^{+\infty} 
\left| 
j^s \langle  \varphi, \varphi_j\rangle  
\right|^2 
\right)^{\frac{1}{2}}.
\end{multline*}
Let $T>0$. For $u \in L^1(0,T)$, the family $(u_n)_{n \in \N}$ of the iterated primitives of $u$ is defined by induction as, 
\begin{equation*}
u_0:=u \quad \text{ and } \quad \forall n \in \mathbb{N}, \ u_{n+1}(t) := \int_0^t u_n(\tau) d\tau, \quad t \in [0,T].
\end{equation*}
Sometimes, to uniformize the notations of the primitives and derivatives of $u$, we will write $u^{(n)}$ when $n$ is negative to denote $u_{|n|}$, the $|n|$-th primitive of $u$. 

\noindent We will also consider, for any integer $k \in \N$, $H^k \left( (0,T), \R \right)$, the usual integer-order real Sobolev space, equipped with the usual $H^k(0,T)$-norm 
and $H^k_0(0,T)$ the adherence of $C_c^{\infty}(0,T)$, the set of functions with compact support inside $(0,T)$, for the topology $\| \cdot \|_{H^k(0,T)}$. For any integer $k \in \N^*$, a negative Sobolev norm is defined by,
\begin{equation}
\label{def_norm_faible}
\| u \|_{H^{-k}(0,T)} :=|u_1(T)|+\| u_k \|_{L^2(0,T)},  \quad u \in L^2(0,T). 
\end{equation}
These norms don't coincide with the usual $H^{-k}$-norms, but this definition is taken as these quantities arise naturally in both the nonlinear and linearized dynamics.

\subsection{Assumptions on the dipolar moment $\mu$}
Let us make precise the assumptions on the dipolar moment $\mu$ we shall consider in the following.

\smallskip \noindent
(H$_{\reg}$)
The function $\mu$ is in $H^{11}( (0,1), \R)$ with the following boundary conditions 
\begin{equation}
\label{mu_bc}
\mu'(0)
=
\mu'(1)
=
\mu^{(3)}(0)
=
\mu^{(3)}(1)
=
0
.
\end{equation}

\smallskip \noindent
(H$_{\lin}$) 
There exists an integer $K \in \N^* - \{1 \}$ such that 
\begin{equation}
\label{lin_nul}
\langle \mu \varphi_1, \varphi_K \rangle =0,
\end{equation}
\begin{equation}
\label{H_lin_2}
\text{and there exists } 
c>0 
\text{ such that for all } 
j \in \N^*- \{K\},
\quad
\left|
\langle \mu \varphi_1, \varphi_j \rangle 
\right| 
\soe 
\frac{c}{j^7}.
\end{equation}
Define, for $p=1, 2$ and $3$, the following quadratic (with respect to $\mu$) coefficients, 
\begin{equation}
\label{def_Apk}
A^p_K
:=
(-1)^{p-1}
\sum 
\limits_{j=1}^{+\infty} 
\left(\lambda_j -  \frac{\lambda_1+\lambda_K}{2} \right) 
( \lambda_K-\lambda_j)^{p-1} 
(\lambda_j -\lambda_1)^{p-1}
\langle 
\mu 
\varphi_1, 
\varphi_j
\rangle  
\langle 
\mu 
\varphi_K, 
\varphi_j
\rangle.
\end{equation}

\smallskip \noindent
(H$_{\Quad}$) 
The first two quadratic coefficients vanish and the third one does not vanish
\begin{align}
\label{quad_nul_1}
A^1_K
&=
0
,
\\
\label{quad_nul_2}
A^2_K
&=
0
,
\\
\label{quad_non_nul}
A^3_K
&\neq 0
.
\end{align}

\smallskip \noindent
(H$_{\Cub}$) 
The following cubic (with respect to $\mu$) coefficient does not vanish
\begin{equation}
\label{cub_non_nul}
C_K := 
\sum \limits_{j=1}^{+\infty} 
\left(
\lambda_1- \lambda_j
\right)
\langle
\mu 
\varphi_1
,
\varphi_j 
\rangle
\langle 
\mu'^2
\varphi_K
,
\varphi_j 
\rangle
-
\sum \limits_{j=1}^{+\infty} 
\left(
\lambda_j- \lambda_K
\right)
\langle
\mu 
\varphi_j
,
\varphi_K 
\rangle
\langle 
\mu'^2
\varphi_1
,
\varphi_j 
\rangle
\neq 0.
\end{equation}
%
\begin{rem}
\label{decay_coeff}
Notice that, under (H$_{\reg}$), integrations by parts and Riemann-Lebesgue lemma lead to 
\begin{equation}
\label{coeff_IPP}
\forall q \in \N^*, 
\quad 
\langle 
\mu 
\varphi_q
, 
\varphi_j 
\rangle 
= 
\frac{
12 q
}
{
\pi^{6} 
j^{7}
} 
\left( 
(-1)^{j+q} \mu^{(5)}(1)
-\mu^{(5)}(0) 
\right) 
+ 
\underset{j \rightarrow +\infty}{o} 
\left( 
\frac{1}{j^{7}} 
\right). 
\end{equation}
Thus, all the series considered in \eqref{def_Apk} and \eqref{cub_non_nul} converge absolutely. 
\end{rem}

\begin{rem}
\label{rem:lie_brackets}
For smooth vector fields $X, Y$ in $C^{\infty}(\R^d, \R^d)$, the Lie bracket $[X,Y]$ is defined as the following smooth vector field: 
$
[
X,
Y]
(x)
:=
X'(x) Y(x)
- Y'(x) X(x).
$ 
Notice that the sign convention chosen is not usual. 
We also define by induction
\begin{equation*}
\ad_{X}^0(Y) := Y 
\quad
\text{ and }
\quad 
\forall k \in \N, 
\quad
\ad_{X}^{k+1}(Y) := [X, \ad_{X}^{k}(Y)].  
\end{equation*}
At least formally, the assumption (H$_{\Quad}$) can be written in terms of Lie brackets as 
\begin{equation*}
\forall p=1, 2, 
\
\langle  [\ad_{A}^{p-1}(\mu), \ad_{A}^{p}(\mu) ] \varphi_1,  \varphi_K\rangle  =0
\ 
\text{ and }
\ 
\langle  [\ad_{A}^{2}(\mu), \ad_{A}^{3}(\mu) ] \varphi_1,  \varphi_K\rangle \neq 0
.
\end{equation*}
Notice that the last Lie bracket is exactly the one along which the quadratic order adds a drift, denying $W^{3, \infty}$-STLC (see \cref{def_STLC}) for finite-dimensional systems $\dot{x}=f_0(x)+uf_1(x)$ in \cite[Theorem 3]{BM18}. Similarly, (H$_{\Cub}$) can be written as
\begin{equation*}
\langle 
[ \ad_A(\mu),
[ \ad_A(\mu), 
\mu]
]
\varphi_1, 
\varphi_K
\rangle
\neq 0.
\end{equation*}
All these computations can be made rigorous when $\mu$ satisfies (H$_{\reg}$) for instance. In that case, the iterated Lie brackets are well-defined (for all $k \in \N^*$, to compute $\ad^k_A(\mu)\varphi_1$, one needs to check that $\ad^{k-1}_A(\mu)\varphi_1$ is in $\dom A$) and denote commutators of operators. 
\end{rem}
%
%
%
%
At a heuristic level, assumptions (H$_{\reg}$), (H$_{\lin}$), (H$_{\Quad}$) and (H$_{\Cub}$) entail that, in the asymptotic of small controls (in a norm to be specified), the leading terms of the solution $\psi$ of the Schrödinger equation  \eqref{Schrodinger} along the lost direction are given by
\begin{equation}
\label{eq:heuristic}
\langle
\psi(T), 
\varphi_K
e^{-i \lambda_1 T}
\rangle 
\approx
- iA^3_K \int_0^T u_3(t)^2 dt 
+
iC_K
\int_0^T u_1(t)^2 u_2(t) dt. 
\end{equation} 
The existence of a function $\mu$ satisfying (H$_{\reg}$), (H$_{\lin}$), (H$_{\Quad}$) and (H$_{\Cub}$) is proved in \cref{existence_mu}.

\subsection{Main result}
First, we state the notion of small-time local controllability (STLC) used in this paper, stressing the smallness assumption imposed on the control, as it plays a key role in the validity of controllability results.
\begin{defi}
\label{def_STLC}
Let $(E_T, \|\cdot\|_{E_T})$ be a family of normed vector spaces of real functions defined on $[0,T]$ for $T> 0$.
The system \eqref{Schrodinger} is said to be E-STLC around the ground state if there exists $s \in \mathbb{N}$ such that for every $T>0$, for every $\eta>0$, there exists $\delta> 0$ such that for every $\psi_f \in \mathcal{S} \cap H^s_{(0)}(0,1)$ with $\|\psi_f - \psi_1(T)\|_{H^s_{(0)}(0,1)} < \delta$,  there exists $u \in L^2((0,T),\R) \cap E_T$ with $\|u\|_{E_T} < \eta$ such that the solution $\psi$ of \eqref{Schrodinger} associated to the initial condition $\varphi_1$ at time $t=0$ and the control $u$ satisfies $\psi(T)=\psi_f$.  
\end{defi}
When the linearized system around an equilibrium is controllable, using a fixed-point theorem, one can hope to prove STLC for the nonlinear system as explained in \cite[Chap.\ 3.1]{bookC07} in finite dimension. 
When it is not the case, one needs to go further into the expansion in the spirit of \cite[Chap.\ 8]{bookC07}. 
For the Schrödinger equation, a few STLC results are already known. 

\
\paragraph{\textbf{Linear behavior.}} Since \cite{BL10}, it is known that when the coefficients $(\langle \mu \varphi_1, \varphi_j \rangle)_{j \in \N^*}$ satisfy
\begin{equation}
\label{hyp_mu_lin}
\text{there exists a constant } c>0
\text{ such that } 
\quad
\forall j \in \N^*, 
\quad 
\left|
\langle \mu \varphi_1, \varphi_j \rangle
\right|
\soe 
\frac{c}{j^3},
\end{equation}
then the Schrödinger equation is $H^k_0$-STLC around the ground state with targets in $H^{2k+3}_{(0)}$. This result has been extended in \cite{B21} where the author proved that when \eqref{hyp_mu_lin} holds only on a subset of $\N^*$,  STLC holds in projection with a unique control map for a finite range of regularity on the control. Moreover, it also has been proved in \cite{BM14} that, under the weaker assumption,  
\begin{equation*}
\mu'(1) \pm \mu'(0) \neq 0,
\end{equation*}
a finite number of coefficients $\langle \mu \varphi_1, \varphi_K \rangle$ vanish, but the local controllability with controls in $L^2$ holds in large time. 

\
\paragraph{\textbf{Quadratic behavior.}} In \cite{B21bis}, the author proved that when, for some $n \geq 2$ (resp.\ $n=1$)
\begin{equation*}
\langle \mu \varphi_1, \varphi_K \rangle =0,
\quad 
A^1_K=\cdots = A^{n-1}_K=0
\quad
\text{ and }
\quad
A^n_K \neq 0,
\end{equation*}
(with enough regularity on $\mu$ so that the associated series converge), the Schrödinger equation is not $H^{2n-3}$-STLC (resp.\ $W^{-1,\infty}$-STLC), due to a drift quantified by the $H^{-n}$-norm of the control. This follows the work of \cite{BM14} where the authors already denied $L^{\infty}$-STLC in the case $n=1$. Let us stress that such a result entails that, under  ($H_{\lin})$, \eqref{lin_nul} and ($H_{\Quad})$, the Schrödinger equation \eqref{Schrodinger} is not $H^3$-STLC. 

The goal of this paper is the proof of the following statement by taking advantage of the cubic term of the expansion. 
\begin{thm}
\label{the_theorem}
Let $\mu$ satisfying (H$_{\reg}$), (H$_{\lin}$), (H$_{\Quad}$) and (H$_{\Cub}$). Then, the Schrödinger equation \eqref{Schrodinger} is $H^2_0$-STLC around the ground state with targets in $H^{11}_{(0)}(0,1)$.
\end{thm}

Using `higher-order control variations' to prove STLC is classical for ODEs: it has been used for example to prove the sufficiency of Sussman $\S( \theta)$ condition \cite{S87} (see also \cite[Theorem 3.29]{bookC07}) or by Kawski in \cite[Theorem 3.7]{K90}. 
However, up to our knowledge, it is the first time that this strategy is used in infinite dimension.
The proof of \cref{the_theorem} can be brought down to the following ideas. 
\begin{itemize}
\item First, to recover STLC, it is enough to have STLC in projection on the reachable space of the linearized system and to move into the directions lost at the linear level using `higher-order control variations'. 
\item To move into the directions lost at the linear level, the strategy is the following.
\begin{itemize}
\item[$\triangleright$] First, one computes a well-quantified expansion of the solution along the lost directions to identify the leading terms. 
\item[$\triangleright$] Then, one proves that the cubic term can absorb the quadratic term along the lost directions using oscillating controls small in a good asymptotic. 
\item[$\triangleright$] Finally, one corrects the linear components using the STLC result in projection on the reachable space of the linearized system. Sharp estimates on the $H^{-k}$-norms of the control (see \eqref{def_norm_faible}) are needed to prove that such a correction didn't destroy the work done previously along the lost directions. Thus, the work done in \cite{B21}  is a key tool for this paper.  
\end{itemize}
\end{itemize}

The paper is organized as follows. 
First, a systematic approach to recover STLC when the linearized system misses a finite number of directions is described in \cref{sec:black_box} and presented on a few toy-models in finite dimension in \cref{sec:toy_models}. 
Before applying this method to the Schrödinger equation, in \cref{sec:WP_Schro}, we recall its well-posedness and the controllability result in projection of \cite{B21}. 
Then, the power series expansion of the Schrödinger equation is computed in \cref{sec:expansion}. 
Finally, \cref{sec:STLC_result} is dedicated to the proof of \cref{the_theorem}. 

\subsection{State of the art}

\
\paragraph{\textbf{Controllability results on the Schrödinger equation.}}

\medskip 
\subparagraph{\textit{Local exact controllability results.}}
Bilinear control systems have been considered as non controllable for a long time because of a negative result   \cite{BMS82} from Ball, Marsden and Slemrod, later adapted for the Schrödinger equation by Turinici in \cite{T00}. 
The result in \cite{BMS82} was later completed by Boussaïd, Caponigro and Chambrion in \cite{BCC20}.


However, these statements don't give obstructions for the Schrödinger equation to be controllable in different functional spaces. Later, exact local controllability results for 1D models have been proved by Beauchard in \cite{B05, B08}, later improved by Beauchard and Laurent in \cite{BL10} and generalized later in \cite{B21}. In \cite{BL10}, the authors also proved the local controllability of a nonlinear Schrödinger equation with Neumann boundary conditions. The case of Dirichlet boundary conditions has been treated later by Duca and Nersesyan in \cite{DN22}. 

Morancey and Nersesyan also proved the controllability of one Schrödinger equation with a polarizability term \cite{MN14} and of a finite number of equations with one control \cite{M14, MN15}. In dimension $N \leq 3$, Puel \cite{P16} also proved the local exact controllability of a Schrödinger equation, in a bounded regular domain in a neighborhood of an eigenfunction corresponding to a simple eigenvalue, for controls $u=u(t,x)$. Using the link between quantum and classical dynamics, in \cite{BBS21}, the authors also gave some necessary and sufficient conditions of the local controllability of the Schrödinger equation. 

\medskip \noindent \textit{Global results.} 
Using Galerkin approximations, global approximate controllability results have been proved by Boscain, Boussaïd, Caponigro, Chambrion, Mason and Sigalotti in \cite{BCCS12, BCS14, BCC13, BCC20, CMSB09}. The genericity of these sufficient conditions is stated in \cite{CS16}. This strategy has also been used to prove exact controllability in projection on the first eigenstates in \cite{CS18}. Adiabatic arguments \cite{BCMS12, BGRS15, DJT20, DJT22} or Lyapunov techniques \cite{M09, N10} can also be used.  Also, in \cite{DN21}, the authors proved the approximate controllability of a nonlinear Schrödinger equation with bilinear controls.
%

Nersesyan and Nersisyan also proved the global exact controllability in infinite time of one Schrödinger equation in one \cite{NN12} and any dimension \cite{NN12bis}. Later, Duca provided explicit times such that the global exact controllability holds in \cite{D19}. Global exact controllability in projection of infinite bilinear Schrödinger equations has also been proved in \cite{D20}.

\
\paragraph{\textbf{Local controllability result by the power series expansion.}}
When the linearized system is not controllable, the strategy, presented in \cite[Chap. 8]{C07} for finite-dimensional control system, of performing a power series expansion of the solution, can be used to prove both negative or positive controllability results. 

\medskip 
\subparagraph{\textit{Negative results.}}
 In \cite{BM18}, the authors proved that, in finite dimension, for scalar-input differential systems, when the linear test fails, the second-order term adds a drift quantified by the $H^{-n}$-norm of the control, along an explicit Lie bracket, denying $W^{2n-3, \infty}$-STLC for the nonlinear system. 
Such a phenomenon was already observed in infinite dimension, for a Schrödinger equation, by Coron in \cite{C06}, by Beauchard and Morancey in \cite{BM14} and later in \cite{B21bis}. 
In these works, using the second-order term, and more precisely proving a coercivity inequality involving an integer negative Sobolev norm of the control, the authors gave impossible motions in small time. 
For a Burgers equation, STLC is still denied in \cite{M15} proving a coercivity inequality but involving a fractional Sobolev norm of the control instead. 
In \cite{BM20}, obstructions caused by both quadratic integer and fractional drifts are proven on a scalar-input  parabolic equation. 
A similar result has also been proved on a KdV system, for boundary controls in \cite{CKN20} by Coron, Koenig and Nguyen. The authors also showed in \cite{CNK22} that the STLC of a water tank modeled by 1D Saint-Venant equations doesn't hold when the time is not large enough, proving a coercivity property for the quadratic term of the system.

\medskip
\subparagraph{\textit{Positive results.}}
The power series expansion method has been used in infinite dimension to recover STLC for the first time in \cite{CC04} for a KdV control system. Beforehand, in \cite{R97}, Rosier studied the controllability of the Kdv equation posed on a finite interval $(0,L)$ with Dirichlet boundary conditions and the control acting on the Neumann data at the right end-point of the interval. The author proved that when $L$ belongs to a set of critical values, the linearized system around the origin is not controllable due to the existence of a finite-dimension subspace of unreachable states.
For some critical values such that the space of unreachable states is of dimension 1, Coron and Crépeau in \cite{CC04} recovered local controllability in small time using a power series expansion of the solution of order 3 for which there is no quadratic term. 
In \cite{C07}, the same approach is used by Cerpa to treat another critical length. However, in this case, the unreachable set of directions at the linear level is of dimension 2 and a second order expansion is sufficient to recover local controllability, but the result holds only in large time. 
Later, Cerpa and Crépeau proved in \cite{CC09} the local controllability in large time for any critical lengths with the same strategy.

Such a method has also already been used for the Schrödinger equation. In \cite{BC06} and \cite{B08}, the controllability of a quantum particle in a 1D infinite square potential well with variable length is studied. In both cases, the proof relies on a compactness argument that needs local controllability results around many periodic trajectories. Those local results are proved by the linear test or using second order terms for some trajectories for which the linearized system looses one direction. 
In \cite{BM14}, local controllability with controls in $L^2$ in large time has also been proved using an expansion of order 2. 

Moreover, for the first time in \cite{BM20}, on a scalar-input parabolic equation, the power series method has been used to recover at the quadratic order an infinite number of direction lost at the linear level.

Finally, let us quickly mention that stabilization results have also been proved using the power series expansion method as in \cite{CE19, CR17, CRX17}.

\section{Method of `control variations'}
\label{sec:black_box}
Under (H$_{\lin}$), the linearized system around the ground state of the Schrödinger equation \eqref{Schrodinger} is not controllable: it misses one complex direction $\varphi_K$. This situation is called `controllability up to (real) codimension two'. 
The goal of this section is to propose a systematic approach to deal with these situations, which is different from the one for ODEs and better adapted to PDEs. For finite dimensional systems, the classical approach used by Kawski \cite[Theorem 2.4]{K90} consists in,
\begin{itemize}
\item proving that any direction lost at the linear level is a `tangent vector' thanks to higher order control variations,
\item and deducing the STLC of the nonlinear system thanks to a time-iterative process that uses arbitrary small-time intervals.
\end{itemize}
For PDEs, using arbitrary small-time intervals is not comfortable, because of the control-cost explosion when the time goes to zero. Therefore, in our new approach, Kawski's time-iteration process is replaced by a Brouwer fixed point argument. This is why our new notion of `tangent vector' contains a continuity property.

\subsection{Main result}
To encompass finite and infinite dimensional systems, STLC is discussed in terms of the surjectivity of the end-point map. Let us state first the functional setting. 
Let $X$ be a Banach space over $\R$. 
Let $(E_T, \|\cdot\|_{E_T})$ be a family of normed vector spaces of functions defined on $[0,T]$ for $T> 0$. 
Assume that for all $T_1, T_2>0$, for all $u \in E_{T_1}$ and $v \in E_{T_2}$, the concatenation of the two functions $u \# v$ defined by
\begin{equation}
\label{concatenation}
u \# v 
:= 
u 
\1_{(0,T_1)}
+
v( \cdot - T_1)
\1_{(T_1,T_1+T_2)}
\end{equation} 
is in $E_{T_1+T_2}$ with moreover the following estimate:
\begin{equation}
\label{continuity_ET}
\| 
u 
\#
v
\|_{E_{T_1+T_2}}
\ioe 
\| u\|_{E_{T_1}}
+
\| v\|_{E_{T_2}}.
\end{equation}
For example, for any positive integer $k$, the Sobolev space $H^k_0(0,T)$ satisfies this property whereas the Sobolev space $H^k(0,T)$ doesn't. 
Finally, let $(\F_T)_{T > 0}$ be a family of functions from $X \times E_T$ to $X$ for $T>0$. Later, in all our applications, $\F_T$ will denote the end-point of the control system. 

%
\smallskip \noindent 
First, let us make precise the new definition of `tangent vector' used in this paper. 
\begin{defi}
\label{def_new_tv} 
A vector $\xi \in X$ is called a small-time $E$-continuously approximately reachable vector if there exists a continuous map $\Xi : [0, +\infty)\rightarrow X$ with $\Xi(0)=\xi$ such that for all $T > 0$, there exists $C, \rho, s >0$ and a continuous map $b \in (-\rho, \rho) \mapsto u_b \in E_T$ such that, 
\begin{equation}
\label{new_tv}
\forall b \in (-\rho, \rho), 
\quad 
\left\| 
\F_{T}
\left(
0,
u_b
\right) 
- 
b 
\Xi(T)
\right\|_{X}
\ioe 
C
|
b
|^{1+s}
\quad
\text{ with }
\quad
\| 
u_b
\|_{E_T}
\ioe 
C
|
b
|^{s}.
\end{equation}
The family $(u_b)_{b \in \R}$ (resp.\ the map $\Xi$) is called the control variations (resp.\ the vector variations) associated with $\xi$.
\end{defi}

\begin{rem}
Let us stress that, for finite-dimensional system, in \cite{K90}, Kawski (see also the work of Frankowska \cite{F87, F89}) introduced rather the following definition: a vector $\xi$ is said to be an $m$-th order tangent vector if there exists a family of controls $(u_T)_{T>0}$ such that 
\begin{equation}
\label{methodo:tv_bis}
\F_T( 0, u_T)= T^m \xi + o(T^m) 
\quad
\text{ when } 
T \rightarrow 0. 
\end{equation}
Our \cref{def_new_tv} is different: the final time and the amplitude of the target are unrelated. This allows constants in \eqref{new_tv} badly quantified with respect to the final time $T$, which is not possible in \eqref{methodo:tv_bis}. Hence, being a small-time $E$-continuously approximately reachable vector is a weaker property than being a tangent vector. Also, for this reason, the dependency of the constants with respect to the final time will not be tracked in this paper. 
For the Schrödinger equation \eqref{Schrodinger}, the lost directions at the linear level are approximately reachable vectors, but it seems more complicated to prove that they are tangent vectors. 
\end{rem}

Our systematic approach is formalized in the following statement. 
\begin{thm}
\label{black_box}
Assume the following hypotheses hold.
\begin{itemize}
\item[$(A_1)$] For all $T>0$, $\F_T: X \times E_T \rightarrow X$ is of class $C^2$ on a neighborhood of $(0,0)$ with $\F_T(0,0)=0$. 
\item [$(A_2)$] For all $x \in X$, $T \in \R_+ \mapsto d \F_T(0,0).(x,0) \in X$ can be continuously extended at zero with $d \F_0(0,0).(x,0)=x$. 
\item [$(A_3)$] For all $T_1, T_2 > 0$, for all $x \in X$, for all $u \in E_{T_1}$ and $v \in E_{T_2}$, 
\begin{equation}
\label{translation_F}
\F_{T_1+T_2}
\left(
x,
u \# v
\right)
=
\F_{T_2}
\left(
\F_{T_1}
\left(
x,
u
\right)
,
v
\right)
.
\end{equation}
\item [$(A_4)$]  The space $\H:=\Ran d\F_T(0,0).(0, \cdot)$ doesn't depend on time, is closed and of finite codimension $n$.

\item  [$(A_5)$] There exists $\M$ a supplementary of $\H$ that admits a basis $(\xi_i)_{i=1, \ldots,n }$  of small-time $E$-continuously approximately reachable vectors.
\end{itemize}
Then, for all $T>0$, $\F_T$ is locally onto from zero: for all $\eta>0$, there exists $\delta>0$ such that for all $x_f \in  X$ with $\| x_f\|_X < \delta$, there exists $u \in E_T$ with $\| u \|_{E_T} < \eta$ such that 
\begin{equation*}
\F_T(0,u)=x_f.
\end{equation*}
\end{thm}

\begin{rem}
\label{rem:time_revers}
If in addition of $(A_1)-(A_5)$, we assume that 
\begin{itemize}
\item [$(A_6)$] for all $T > 0$ and $u \in E_{T}$, $u(T- \cdot)$ is in $E_{T}$ with
\begin{equation*}
\F_T( \F_T(0,u), u(T- \cdot)) = 0,
\end{equation*}
\end{itemize}
then, for all $T>0$, $\F_T$ is locally onto: for all $\eta>0$, there exists $\delta>0$ such that for all $(x_0,x_f) \in  X^2$ with $\| x_0\|_X+\| x_f\|_X < \delta$, there exists $u \in E_T$ with $\| u \|_{E_T} < \eta$ such that 
\begin{equation*}
\F_T(x_0,u)=x_f.
\end{equation*}
Indeed, let $(x_0,x_f) \in  X^2$ with $\| x_0\|_X+\| x_f\|_X < \delta$. By \cref{black_box}, there exists $u, v \in E_T$ such that $\F_T(0,u)=x_0$ and $\F_T(0,v)=x_f$. 
Then, using successively $(A_3)$ and $(A_6)$, one has 
\begin{equation*}
\F_{2T}( x_0, u(T-\cdot) \# v)
=
\F_T( \F_T(x_0, u(T-\cdot)), v)
=
\F_T( 0, v)
=
x_f. 
\end{equation*}
\end{rem}

\begin{rem}
The $C^2$-regularity of $\F_T$ in $(A_1)$ is for convenience. In the proof of \cref{black_box}, one needs the following estimates: for all $T>0$ and $R>0$, there exists $C>0$ such that for all $(x,u) \in X \times E_T$ with $\| x\|_X+\| u \|_{E_T} <R$, 
\begin{align}
\label{Gronwall_F}
\left\|
\F_{T}(x,u)
-
\F_{T}(0,u)
-
\F_{T}(x,0)
\right\|_{X}
&\ioe
C
\| x\|_X 
\| u \|_{E_T},
\\
\label{dev_diff_F}
\left\|
\F_{T}(x,u)
-
d \F_T(0,0).(x,u)
\right\|_X
&\ioe
C
\left(
\| x\|_X^2
+
\| u\|_{E_T}^2
\right)
.
\end{align}
Both estimates follow from Taylor formulas when $\F_T$ is of class $C^2$.
\end{rem}

\begin{rem}
When $\F_T$ denotes the end-point map of a control system, 
\begin{itemize}
\item $(A_1)$ is linked to the well-posedness of the system: for controls in $E_T$ and initial data in $X$, the end-point of the solution must take values in $X$;
\item $(A_2)$ asks that the solutions of the linearized system are continuous with respect to time; 
\item $(A_3)$ is related to the semigroup property of the equation;
\item $(A_4)$ means that the linearized system is `controllable up to finite codimension';
\item  $(A_5)$ means that the directions lost at the linear level can be recovered using `higher order control variations';
\item and $(A_6)$ is linked to the time reversibility of the equation. 
\end{itemize}
\end{rem}

\subsection{Proof of \cref{black_box}}
The first tool is the local surjectivity of the nonlinear map $\F_T$ up to finite codimension. 
\begin{prop}
\label{lem:linear_test}
Assume $(A_1)$ and $(A_3)$. 
Let $T>0$ and $\mathcal{N}$ a supplementary of $\H$. Denote by $\P$ the projection on $\H$ parallely to $\mathcal{N}$.
Then, $\F_T$ is locally onto in projection on $\H$:  there exists $\delta_0, C>0$ and a $C^1$-map $\Gamma_{T} : B_{X}(0, \delta_0) \times \left( B_{X}(0, \delta_0) \cap \H \right) \rightarrow E_T$ with $\Gamma_{T}( 0,0)=0$ such that for all $(x_0, x_f) \in B_{X}(0, \delta_0) \times \left( B_{X}(0, \delta_0) \cap \H \right)$, 
\begin{equation}
\label{target_lin}
\P
[
\F_{T}
\left(
x_0,
\Gamma_{T}(x_0,x_f) 
\right)
] 
= x_f,
\end{equation}
with the size estimate 
\begin{equation}
\label{estim_contr_lin}
\| 
\Gamma_{T}(x_0, x_f) 
\|_{E_T}
\ioe
C
\left(
\| x_0 \|_{X}
+
\| x_f \|_{X}
\right).
\end{equation}
\end{prop}
\noindent 
The proof follows from applying the inverse mapping theorem to the $C^1$-map 
\begin{equation}
\label{end_point_proj}
\begin{array}{lccl}
 & X \times E_T & \to & X \times \H \\
 & (x, \ u) & \mapsto & \left( x, \P[ \F_{T}(x, u) ] \right).  \\
\end{array}
\end{equation}

Then, 
we prove that every direction spanned by approximately reachable vectors can be recovered using higher order control variations. 
\begin{prop}
\label{prop:motion_M}
Under the assumptions of \cref{black_box}, there exists $T^*>0$ such that for all $T \in (0, T^*)$ and $\eta>0$, there exists 
$\M_T$ a supplementary of $\H$, $C, s, \rho>0$ and a continuous map 
$
z \in \M_T \cap B_X(0, \rho) 
\mapsto
u_z \in E_T
$
%
%
such that for all $z \in \M_T \cap B_X(0, \rho)$, 
\begin{equation}
\label{motion_M}
\left\|
\F_{T}
\left(
0,u_z
\right) 
-  
 z 
\right\|_{X}
\ioe 
C 
\| z \|_X^{1+s}
\quad
\text{ with }
\quad 
\|u_z\|_{E_T} \ioe \eta.
\end{equation}
\end{prop}

\begin{proof}
Let $T >0$ and $\eta >0$. Let $0=T_0 < \cdots < T_n=T$ be a subdivision of $[0, T]$.
By $(A_5)$, there exists $C, \rho, s>0$ and for all $i=1, \ldots,n$, two continuous maps $\Xi_i : [0,+\infty) \rightarrow X$ with $\Xi_i(0)=\xi_i$ and $b \in (-\rho, \rho) \mapsto u_b^i \in E_{T_{i}-T_{i-1}}$ such that for all $b \in (-\rho, \rho)$, 
\begin{equation}
\label{methodo:tv_eq1}
\left\| 
\F_{T_i -T_{i-1}}
\left(
0,
u^{i}_b
\right)
- 
b \Xi_i(T_i -T_{i-1})  
\right\|_{X}
\ioe 
C
|
b
|^{1+s}
\
\text{ with }
\
\| 
u^{i}_b
\|_{E_{T_i -T_{i-1}}}
\ioe 
C
|
b
|^{s}.
\end{equation}
For all $c=(c_1, \ldots, c_n) \in \R^n$ with $\|c\| < \rho$ and $k \in \{1, \ldots, n\}$, we define
\begin{equation*}
u^{\#^k}_c
:=
u^{1}_{
c_1
}
\#
u^{2}_{
c_2
}
\#
\ldots
\#
u^{k}_{
c_k
}
\in 
E_{T_k}
.
\end{equation*}
We prove by induction on $k \in \{1, \ldots, n\}$  the existence of $C>0$ such that for all $\|c\|_X < \rho$, 
\begin{equation}
\label{rec:hyp_k}
\left\|
\G_{T_k} 
(
0,
u^{\#^k}_c
)
-  
\sum_{i=1}^k 
c_i
d \G_{T_k-T_i}(0,0).
(
\Xi_i(T_i -T_{i-1})  
, 0)
\right\|_{X}
\ioe 
C 
\| c \|^{1+s}.
\end{equation}
The initialization with $k=1$ follows from the definition of the family $(u^1_b)_{b \in \R}$ in \eqref{methodo:tv_eq1} and that $d \G_0(0,0).( \cdot, 0)=\Id_X$ by $(A_2)$. Let us prove the heredity: assume that \eqref{rec:hyp_k} holds for some $k \in \{1, \ldots, n\}$. 
First, by $(A_3)$, one has, 
\begin{equation*}
\G_{T_{k+1}}
(
0,
u^{\#^{k+1}}_c
)
=
\G_{T_{k+1}-T_k}
\left(
\G_{T_k}
(
0,
u^{\#^{k}}_c
)
,
u^{k+1}_{
c_{k+1}
}
\right).
\end{equation*} 
Thus, together with the inequality \eqref{Gronwall_F}, one gets, 
\begin{multline}
\label{rec1}
\left\|
\G_{T_{k+1}}
(
0,
u^{\#^{k+1}}_c
)
-
\G_{T_{k+1}-T_k}
(
0
,
u^{k+1}_{
c_{k+1}
}
)
-
\G_{T_{k+1}-T_k}
(
\G_{T_k}
(
0,
u^{\#^{k}}_c
)
,
0
)
\right\|_X
\\
\ioe 
C
\| 
\G_{T_k}
(
0,
u^{\#^{k}}_c
)
\|_X
\| 
u^{k+1}_{
c_{k+1}
}
\|_{E_{T_{k+1}-T_k}}
\ioe
C
\| c\|
| c_{k+1}|^s,
\end{multline}
thanks to \eqref{rec:hyp_k} and the size estimate in \eqref{methodo:tv_eq1}. 
Moreover, by \cref{def_new_tv}, 
\begin{equation}
\label{rec_2}
\| 
\G_{T_{k+1} - T_{k}}
(
0,
u^{k+1}_{
c_{k+1}
}
)
- 
c_{k+1} 
\Xi_{k+1}(T_{k+1} -T_{k})  
\|_{X}
\ioe 
C
|
c_{k+1}
|^{1+s}.
\end{equation}
Besides, using the Taylor expansion \eqref{dev_diff_F}, one gets, 
\begin{multline}
\label{rec3}
\left\|
\G_{T_{k+1}-T_k}
(
\G_{T_k}
(
0,
u^{\#^{k}}_c
)
,
0
)
-
d \G_{T_{k+1}-T_k}(0,0).
(
\G_{T_k}
(
0,
u^{\#^{k}}_c
)
, 
0
)
\right\|_X
\\
\ioe 
C
\|
\G_{T_k}
(
0,
u^{\#^{k}}_c
)
\|_X^2
\ioe
C \| c\|^2.
\end{multline}
Moreover, using the induction hypothesis \eqref{rec:hyp_k}, one has
\begin{multline}
\label{rec4}
\Big\|
d \G_{T_{k+1}-T_k}(0,0). 
(
\G_{T_k}
(
0,
u^{\#^{k}}_c
)
,
0
)
\\-
\sum \limits_{i=1}^k
c_i 
d \G_{T_{k+1}-T_k}(0,0). 
\left(
d \G_{T_k-T_i}(0,0).
( 
\Xi_i(T_i -T_{i-1})  
, 0)
, 
0
\right)
\Big\|_X
\ioe
C \|c\|^{1+s}.
\end{multline}
Besides, differentiating \eqref{translation_F}, one gets that for all $T_1, T_2 > 0$ and $x \in X$, 
\begin{equation*}
d\F_{T_1}(0,0).
(
d\F_{T_2}(0,0).(x,0)
,
0
)
=
d \F_{T_1+T_2}(0,0).(x,0)
.
\end{equation*}
Thus, \eqref{rec3} and \eqref{rec4} lead to 
\begin{equation}
\label{rec5}
\left\|
\F_{T_{k+1}-T_k}
(
\G_{T_k}
(
0,
u^{\#^{k}}_c
)
,
0
)
-
\sum \limits_{i=1}^k
c_i 
d \G_{T_{k+1}-T_i}(0,0). 
\left(
\Xi_i(T_i -T_{i-1})  
, 
0
\right)
\right\|_X
\ioe 
C
\| c\|^{1+s}
.
\end{equation}
Then, estimates \eqref{rec1}, \eqref{rec_2} and \eqref{rec5} lead to \eqref{rec:hyp_k} for $k+1$ and this concludes the induction.

\smallskip \noindent \emph{Conclusion.} 
The following map is continuous and doesn't vanish at zero,
\begin{equation*}
(t_1, \ldots, t_n, \hat{t}_1, \ldots, \hat{t}_n)
\mapsto
\det\left(
\P_{\M}[d\F_{t_1}(0,0).(\Xi_1(\hat{t}_1), 0)]
, 
\ldots, 
\P_{\M}[d\F_{t_n}(0,0).(\Xi_n(\hat{t}_n), 0)]
\right),
\end{equation*}
where $\P_{\M}:=\Id - \P$ denotes the projection on $\M$ defined in $(A5)$ parallely to $\H$. 
Thus, there exists $T^*>0$ such that for all $T \in [0, T^*)$, $(\P_{\M}[d\F_{T-T_i}(0,0).(\Xi_i(T_i-T_{i-1}), 0)])_{i=1, \ldots, n}$ is a basis of $\M$. As $\M$ is a supplementary of $\H$, one deduces that $\M_T$ defined as
\begin{equation*}
\M_T
:=
\Span( 
d \F_{T-T_i}(0,0). 
\left(
\Xi_i(T_i -T_{i-1})
,
0
\right)
,
\
i=1, \ldots, n
),
\end{equation*}
is also a supplementary of $\H$. Thus, by  \eqref{rec:hyp_k}, the proof is concluded with 
\begin{equation*}
u_z:=u^{\#^n}_{c_1, \ldots, c_n}
\quad 
\text{ for all }
\quad
z
=
\sum \limits_{i=1}^n 
c_i
d \F_{T-T_i}(0,0). 
\left(
\Xi_i(T_i -T_{i-1})
,
0
\right)
\in 
\M_T 
.
\end{equation*}
Notice that the continuity of the map $z \mapsto u_z$ stems from the ones of $b \mapsto u^i_b$. 
\end{proof}

Finally, one can prove \cref{black_box}: the `higher order control variations' constructed in \cref{prop:motion_M} and the local surjectivity of $\F_T$ up to finite codimension given in \cref{lem:linear_test} are enough to gain back the controllability lost at the linear level. 

\begin{rem}
\label{rem:temps_petit}
It is enough to prove the conclusion of \cref{black_box} for sufficiently small final times $T \in (0, T^*)$. Indeed, for $T>T^*$, taking $T_i>0$ such that $T-T_i < T^*$, one has
\begin{equation*}
\forall u \in E_{T-T_i},
\quad
\F_T( 0, 0_{[0,T_i]} \# u)
=\F_{T-T_i}( \F_{T_i}(0,0), u)
=\F_{T-T_i}(0,u),
\end{equation*}
using $(A_1)$ and $(A_3)$. Thus, the result in small time entails the result in large time. 
\end{rem}

\begin{proof}[Proof of \cref{black_box}.]
Let $T>0$ the final time, $T_1 \in (0,T)$ an intermediate time and $\eta >0$ the accuracy on the control. Define $\delta:= \min( \delta_0, \frac{\eta}{4C})$ where $\delta_0$ (resp.\ $C$) is defined by \cref{lem:linear_test} (resp.\ \eqref{estim_contr_lin}). Let $x_f$ in $X$ with $\| x_f \|_X < \delta$.

\medskip \noindent \emph{Step 1: Steering $0$ almost to $x_f$.} 
Let $\M_{T_1}$ the supplementary of $\H$ given in \cref{prop:motion_M}. 
The goal of this step is to construct a $n$-parameters family $(v_z)_{z \in \M_{T_1}}$ such that, for every $z \in \M_{T_1}$ small enough, one has
\begin{align}
\label{goal_H}
&\P \G_T( 0, v_z) = \P x_f 
,
\\
\label{goal_M}
&\left\| 
\P_{T_1} \G_T( 0, v_z)- z  
\right\|_X
\ioe 
C
\|z\|_X^{1+\gamma}
+
C 
\| \P x_f\|_X^2
,
\quad 
\text{ with } \gamma >0,
\\
\label{goal_vz}
&
\| v_z\|_{E_T} \ioe \eta
,
\end{align}
where $\P_{T_1}:=\Id- \P$ denotes the projection on $\M_{T_1}$ parallely to $\H$. 
By \cref{prop:motion_M}, there exists $C, \rho, s>0$ and a continuous map $\tild{z} \mapsto u_{\tild{z}}$ from $\M_{T_1} \cap B_X(0, \rho)$ to $E_{T_1}$ such that, 
\begin{equation}
\label{contr_step1}
\forall \tild{z} \in \M_{T_1} \cap B_X(0, \rho), 
\quad
\left\|
\G_{T_1}
\left(
0,u_{\tild{z}}
\right) 
-  
\tild{z}
\right\|_{X}
\ioe 
C 
\| \tild{z} \|_X^{1+s}
\quad
\text{ with }
\quad 
\|u_{\tild{z}}\|_{E_{T_1}} \ioe \frac{\eta}{2}.
\end{equation}
Denote by $(e_i^{T_1})_{i=1, \ldots,n}$ a basis of $\M_{T_1}$. 
Then, the following map is continuous and non-vanishing at zero,
\begin{equation*}
t
\mapsto 
\det
\left( 
\P_{T_1}[
d \F_{t}(0,0).( e_1^{T_1}, 0)
]
,
\ldots
, 
\P_{T_1}[
d \F_{t}(0,0).( e_n^{T_1}, 0)
]
\right).
\end{equation*}
Thus, for $T$ small enough, $\P_{T_1}[d \G_{T-T_1}(0,0).( \cdot, 0)]$ is invertible from $\M_{T_1}$ to $\M_{T_1}$ with a continuous inverse by the open mapping principle. Hence, there exists a linear continuous map $h$ from $\M_{T_1}$ to $\M_{T_1}$ such that, 
\begin{equation}
\label{def_h}
\forall z \in \M_{T_1},
\quad 
\P_{T_1}[
d \G_{T-T_1}(0,0).(h(z), 0)
]
=z.
\end{equation}
Finally, for all $z \in \M_{T_1} \cap B_X(0, \rho)$, we define
\begin{equation}
\label{def_vz}
v_z 
:= 
u_{h(z)}
\#
\Gamma_{T-T_1}
\left(
\G_{T_1}
\left(
0,u_{h(z)}
\right) 
,
\P x_f
\right),
\end{equation}
where $\Gamma_{T-T_1} : B_X(0, \delta_0) \times \left( B_X(0, \delta_0) \cap \H \right)$ is constructed in \cref{lem:linear_test} with the supplementary $\M_{T_1}$ and the family$(u_{\tild{z}})_{\tild{z} \in \M_{T_1}}$ is constructed in \eqref{contr_step1}. 
As 
$
\G_{T_1}
\left(
0,u_{h(z)}
\right) 
\rightarrow 0
$
when $z$ goes to 0, for $\rho$ small enough, 
$
\left\|
\G_{T_1}
\left(
0,u_{h(z)}
\right) 
\right\|_X
< \delta_0
$
.

\smallskip \noindent \emph{Size estimate.}
By \eqref{continuity_ET}, for all $z \in \M_{T_1} \cap B_X(0, \rho)$, $v_z$ is in $E_{T}$ with 
\begin{align*}
\| v_z \|_{E_T} 
&\ioe
\| u_{h(z)} \|_{E_{T_1}}
+
\left\|
\Gamma_{T-T_1}
\left(
\G_{T_1}
\left(
0,u_{h(z)}
\right) 
,
\P x_f
\right)
\right\|_{E_{T-T_1}} 
\\
&\ioe
\frac{\eta}{2}
+
C
\left(
\| 
\G_{T_1}
\left(
0,u_{h(z)}
\right) 
\|_X
+
\|
\P
x_f
\|_X
\right)
\ioe 
\frac{\eta}{2}
+ 
2 C\delta
\ioe
\eta,
\end{align*}
using the size estimate \eqref{contr_step1} on $u_{h(z)}$ and the estimate \eqref{estim_contr_lin} on $\Gamma_{T-T_1}$. This proves \eqref{goal_vz}.

\smallskip \noindent \emph{Target almost reached.}
Moreover, using $(A_3)$, one has, 
\begin{equation}
\label{translation}
\G_T (0, v_z)
=
\G_{T-T_1}
\left(
\G_{T_1}(0, u_{h(z)})
,
\Gamma_{T-T_1}
\left(
\G_{T_1}
\left(
0,u_{h(z)}
\right) 
,
\P x_f
\right)
\right).
\end{equation}
Therefore, by definition \eqref{target_lin} of $\Gamma_{T-T_1}$, \eqref{goal_H} is already satisfied. To prove \eqref{goal_M}, one can use \eqref{translation} together with the inequality \eqref{Gronwall_F} to get 
\begin{multline}
\label{big_estim}
\left\|
\P_{T_1} \G_T(0, v_z)
-
z
\right\|_X
\ioe 
\left\|
\P_{T_1}
\left[
\G_{T-T_1}
\left(
0
,
\Gamma_{T-T_1}
\left(
\G_{T_1}
\left(
0,u_{h(z)}
\right) 
,
\P x_f
\right)
\right)
\right]
\right\|_X
\\
+
\left\|
\P_{T_1}
\left[
\G_{T-T_1}
\left(
\G_{T_1}(0, u_{h(z)}),
0
\right)
-z
\right]
\right\|_X
\\
+ 
C
\|
\Gamma_{T-T_1}
\left(
\G_{T_1}
\left(
0,u_{h(z)}
\right),
\P x_f
\right) 
\|_{E_{T-T_1}}
\|
\G_{T_1}
\left(
0,u_{h(z)}
\right) 
\|_X.
\end{multline}
Yet, using the estimates \eqref{estim_contr_lin} on $\Gamma_{T-T_1}$, \eqref{contr_step1} on $\G_{T_1}(0, u_{h(z)})$ and the continuity of $h$, the last term of the right-hand side of \eqref{big_estim} is estimated by 
$
C
\|
z
\|^2
+
C
\| \P x_f\|^2
.$
Using the Taylor expansion \eqref{dev_diff_F}, the second term of the right-hand side of \eqref{big_estim} is estimated by 
\begin{equation} 
\label{big_estim_1} 
\|
\P_{T_1}
[
d \G_{T-T_1}(0,0). 
\left(
\G_{T_1}(0, u_{h(z)}),
0
\right)
-
z
]
\|_X
+
\| 
\G_{T_1}(0, u_{h(z)})
\|_X^2
\ioe 
C 
\| z\|^{1+\min(1,s)}
,
\end{equation}
using estimate \eqref{contr_step1}, the construction \eqref{def_h} and the continuity of $h$. 
Moreover, by definition of $\H$ in $(A_4)$,  
$
\P_{T_1}
\left[
d \G_{T-T_1}(0,0). 
(
0, 
\Gamma_{T-T_1}
(\G_{T_1}(0, u_z) 
, 
\P x_f
)
)
\right]
=0
$.
Thus, using again the Taylor expansion \eqref{dev_diff_F}, the first term of the right-hand side of \eqref{big_estim}
is estimated by 
\begin{equation}
\label{big_estim_2}
C
\left\|
\Gamma_{T-T_1}
(\G_{T_1}(0, u_z) 
, 
\P x_f
)
\right\|_X^2
\ioe 
C
\left(
\| z\|^2
+
\| \P x_f\|^2
\right),
\end{equation}
using estimate \eqref{estim_contr_lin} on $\Gamma_{T-T_1}$ and \eqref{contr_step1}. Therefore, \eqref{big_estim}, \eqref{big_estim_1} and \eqref{big_estim_2} lead to \eqref{goal_M}.

\medskip \noindent \emph{Step 2: Steering $0$ to $x_f$.} Thanks to \eqref{goal_H} and \eqref{goal_vz}, to conclude the proof, it remains to prove the existence of $z \in \M_{T_1} \cap B_X(0, \rho)$ such that $\P_{T_1} \G_{T}(0, v_z)= \P_{T_1} x_f$. To that end, we apply the Brouwer fixed-point theorem to the function
\begin{displaymath}
G_{x_f}:
\left|
  \begin{array}{rcl}
    \M_{T_1} \cap B_X(0, \rho) & \longrightarrow & \M_{T_1} \\
    z & \longmapsto & z -\P_{T_1}[\G_T(0, v_z)] +\P_{T_1}[x_f]. \\
  \end{array}
\right.
\end{displaymath}
First, notice that by continuity of $\G_{T}$, of $\Gamma_{T-T_1}$, of $h$ and of $\tild{z} \mapsto u_{\tild{z}}$, the map $z \mapsto v_z$ defined in \eqref{def_vz} is continuous from $\M_{T_1}$ to $E_T$. Thus, $G_{x_f}$ is continuous.
It remains to prove that it stabilizes a ball. Let $\rho' \in (0, \rho)$ such that $C \rho'^{ \gamma} < \frac{1}{2}$ and reduce $\delta$ such that $C \delta^2 + \delta < \frac{\rho'}{2}$ where $C$ is given in \eqref{goal_M}. Then, using estimate \eqref{goal_M}, one has for all $z \in \M_{T_1} \cap B_X(0, \rho')$, 
\begin{equation*}
\| G_{x_f}(z) \|_X
\ioe 
C
\| z \|_X^{1+\gamma}
+
C\| \P x_f\|_X^2
+
\| \P x_f\|_X
\ioe 
C \rho'^{1+\gamma} 
+
C \delta^2 + \delta
\ioe \rho'.
\end{equation*}
Thus, one can apply Brouwer fixed-point theorem to $G_{x_f}$ to conclude the proof. 
\end{proof}

\section{Toy-models in finite dimension}
\label{sec:toy_models}
The goal of this section is to illustrate the method presented in \cref{sec:black_box} on examples in finite dimension. 
For the sake of simplicity, we only explain how to prove that the directions lost at the linear level are small-time $E$-continuously approximately reachable vectors.
 The verification of the other assumptions of \cref{black_box} is left to the reader.
 In this section, for all $n \in \N^*$, $(e_i)_{i=1, \ldots, n}$ denotes the canonical basis of $\R^n$.

\subsection{A first toy-model: not all lost directions can be recovered}
Consider the following control-affine polynomial system
 \begin{equation}
 \label{toy_model_1}
  \left\{
    \begin{array}{ll}
        \dot{x}_1= u, \\
        \dot{x}_2=  x_1^2 +x_1^3.
    \end{array}
\right.
\end{equation}
The reachable space of the linearized system around $(0,0)$ is given by $\H=\Span(e_1)$ and its supplementary by $\M=\Span(e_2)$. However, $e_2$ is not a small-time $W^{-1, \infty}$-continuously approximately reachable vector because the following quantity
\begin{equation*}
x_2(T; \ u, \ 0)
= \int_0^T u_1(t)^2 dt + \int_0^T u_1(t)^3 dt
\geq (1- T \| u \|_{W^{-1, \infty}}) \int_0^T u_1^2(t) dt
\end{equation*}
is positive for $T$ and $\| u \|_{W^{-1, \infty}}$ small enough. This system illustrates Sussman's necessary condition \cite{S83} on 
$
\left[
[f_0, f_1]
, 
f_1
\right](0)
$
for $L^{\infty}$-STLC.

\subsection{Sussman example: a quadratic/cubic competition}
The following classical example illustrates that a cubic term can be used to dominate a quadratic drift and restore STLC,
  \begin{equation}
  \label{toy_model_2}
  \left\{
    \begin{array}{ll}
        \dot{x}_1= u, \\
        \dot{x}_2=  x_1,\\
        \dot{x}_3= x_2^2+x_1^3.
    \end{array}
\right.
\end{equation}
For this system, $\H=\Span(e_1, e_2)$ and $\M=\Span(e_3)$. 

First,  \eqref{toy_model_2} is not $W^{1, \infty}$-STLC. Indeed, considering a trajectory such that $x_1(T)=x_2(T)=0$, two integrations by parts give, 
\begin{equation*}
\int_0^T u_1(t)^3 dt 
=
- 2\int_0^T u_2(t) u_1(t) u(t) dt
=
\int_0^T u_2(t)^2 u'(t) dt. 
\end{equation*}
Thus, provided that $\| u \|_{W^{1, \infty}(0,T)} \leq \frac{1}{2}$,
\begin{equation*}
x_3(T; \ u, \ 0) 
= \int_0^T u_2(t)^2 ( 1+ u'(t)) dt
\geq 
\frac{1}{2}
\int_0^T u_2(t)^2 dt.
\end{equation*}
Hence, it is impossible to reach states of the form $(0,0, -\delta)$ with $\delta >0$. 

However, $e_3$ is a small-time $L^{\infty}$-continuously approximately reachable vector. Indeed, in this asymptotic, one can use the cubic term to absorb the quadratic term along the lost direction using oscillating controls defined for $b \in \R^*$ by 
\begin{equation*}
\forall t \in [0,T], 
\
u_b(t)
:=
\sign(b)
|b|^{\frac{1}{11}}
\phi''
\left(
\frac{
t
}
{
|b|^{\frac{2}{11}}
}
\right)
\quad
\text { with }
\
\phi \in C_c^{\infty}(0,1)
\text{ s.\ t.\ }
\int_0^1 \phi'(\theta)^3 d\theta
=
1.
\end{equation*}
Indeed, performing the change of variables $t= |b|^{\frac{2}{11}} \theta$, one gets 
\begin{align*}
x_3(T; \ u_b, \ 0)
&=
\int_0^{|b|^{\frac{2}{11}}}
\left(
\sign(b)
|b|^{\frac{5}{11}}
\phi
\left(
\frac{
t
}
{
|b|^{\frac{2}{11}}
}
\right)
\right)^2
dt
+
\int_0^{|b|^{\frac{2}{11}}}
\left(
\sign(b)
|b|^{\frac{3}{11}}
\phi'
\left(
\frac{
t
}
{
|b|^{\frac{2}{11}}
}
\right)
\right)^3
dt
\\
&= 
|b|^{\frac{12}{11}}
\int_0^1 \phi(\theta)^2 d\theta
+
\sign(b)
|b|
\int_0^1 \phi'(\theta)^3 d\theta
=
b + \O( |b|^{\frac{12}{11}}).
\end{align*}
Moreover, along the `linear components', as $u_b$ is supported on $(0, |b|^{\frac{2}{11}}) \subset (0,T)$ for $b$ small enough, one directly has  
\begin{equation*}
\left(
x_1(T; \ u_b, \ 0),
\
x_2(T; \ u_b, \ 0)
\right)
=
\left(
u_1(T), 
u_2(T)
\right)
=
(0,0)
.
\end{equation*}
Besides, one has the following estimates on the controls, 
\begin{equation*}
\forall k \in \N, 
\quad
\| 
u_b^{(k)}
\|_{L^{\infty}(0,T)}
\ioe 
\| \phi^{(2+k)} \|_{L^{\infty}(0,1)} 
|b|^{\frac{1-2k}{11}}. 
\end{equation*}
Hence, this family of controls is arbitrary small in $L^{\infty}(0,T)$ (but not in $W^{1, \infty}(0,T)$). Moreover, this estimate (with $k=0$) also gives that the map $b \mapsto u_b$ from $\R^*$ to $L^{\infty}(0,T)$ can be extended continuously at zero with $u_0=0$. 
Therefore, $e_3$ is a small-time $L^{\infty}$-continuously approximately reachable vector and by \cref{black_box}, \eqref{toy_model_2} is $L^{\infty}$-STLC. Notice that this was already known thanks to the Susmann $\S(\theta)$ condition \cite{S83}.

\subsection{A polynomial toy-model for the Schrödinger PDE}
The next polynomial control system is designed to be a toy-model for the Schrödinger equation \eqref{Schrodinger} as explained later in \cref{rem:toy_model_poly},
  \begin{equation}
  \label{toy_model_3}
  \left\{
    \begin{array}{ll}
        \dot{x}_1= u, \\
        \dot{x}_2=  x_1,\\
        \dot{x}_3= x_2, \\
        \dot{x}_4 = x_3^2 +  x_1^2 x_2 , \\
        \dot{x}_5 = x_4. 
    \end{array}
\right.
\end{equation}
For this example, $\H= \Span( e_1, e_2, e_3)$ and $\M=\Span(e_4, e_5)$. Moreover, solving explicitly \eqref{toy_model_3}, the fourth and fifth components are given by, 
\begin{align}
\label{tm3_com4}
x_4(T; \ u , \ 0)
&=
\int_0^T u_3(t)^2 dt 
+
\int_0^T u_1(t)^2 u_2(t) dt, 
\\
\label{tm3_com5}
x_5(T; \ u , \ 0)
&=
\int_0^T (T-t) u_3(t)^2 dt
+
\int_0^T (T-t) u_1(t)^2 u_2(t) dt 
.
\end{align}
First, using Cauchy-Schwarz and Gagliardo-Nirenberg inequalities \cite{N59}, one gets the existence of $C>0$ such that for all $u \in H^3(0,T)$
\begin{equation*}
\left| 
\int_0^T u_1(t)^2 u_2(t) dt
\right|
\ioe 
C
\| u_1\|_{L^2(0,T)}^3
\ioe 
C \left( 
\|u^{(3)}\|_{L^2(0,T)} + T^{-3} \|u\|_{L^2(0,T)} \right)
\|u_3\|^2_{L^2(0,T)}. 
\end{equation*}
Thus, the quadratic term prevails on the cubic term in \eqref{tm3_com4} and \eqref{tm3_com5} when controls are small in $H^3$. This allows to deny $H^3$-STLC for \eqref{toy_model_3}. Nonetheless, let us prove that \eqref{toy_model_3} is $H^2_0$-STLC. 

\smallskip \noindent \emph{Step 1: $e_4$ is a small-time $H^2_0$-continuously approximately reachable vector with vector variations $\Xi_4(T)=e_4+Te_5$.}  Heuristically, for a final time $T$ fixed, looking at \eqref{tm3_com4} and \eqref{tm3_com5}, the cubic terms of $x_4$ and $x_5$ have the same size. Thus, it seems better to use the vector variations $\Xi_4(T)=e_4+Te_5$ instead of $\Xi_4(T)=e_4$. 

As before, in the asymptotic of controls small in $H^2_0$, one can use the cubic term to absorb the quadratic term along the lost direction using oscillating controls of the form, for all $b \in \R^*$, 
\begin{equation}
\label{control_variations}
u_b(t)
=
\sign(b)
|b|^{\frac{7}{41}}
\phi^{(3)}
\left(
\frac{
t
}
{
|b|^{\frac{4}{41}}
}
\right)
\quad
\text { with }
\
\phi \in C_c^{\infty}(0,1)
\text{ s.\ t.\ }
\int_0^1 \phi''(\theta)^2 \phi'(\theta) d\theta
=
1.
\end{equation}
Indeed, substituting these controls into \eqref{tm3_com4} and \eqref{tm3_com5} and performing the change of variables $t= |b|^{\frac{4}{41}} \theta$, one gets 
\begin{align*}
x_4(T; \ u_b, \ 0)
&=
|b|^{
\frac{42}{41}
}
\int_0^1 \phi(\theta)^2 d \theta
+
b
,
\\
x_5(T; \ u_b, \ 0)
&=
|b|^{
\frac{42}{41}
}
\int_0^1 
(T- 
|b|^{
\frac{4}{41}
}
\theta 
) \phi(\theta)^2 d \theta
+
T b 
-
\sign(b)
|b|^{
\frac{45}{41}
}
\int_0^1
\theta
\phi''(\theta)^2
\phi'(\theta)
d\theta
.
\end{align*}
Moreover, as $u_b$ is supported on $(0, |b|^{\frac{4}{41}}) \subset (0,T)$ for $b$ small enough, one directly has,
\begin{equation*}
(
x_1(T; \ u_b, \ 0)
,
\
x_2(T; \ u_b, \ 0)
,
\
x_3(T; \ u_b, \ 0)
)
=
(
u_1(T),
u_2(T), 
u_3(T)
)
=
(0, 0,0).
\end{equation*}
Besides, for all $b \in \R^*$, 
\begin{align}
\label{size_ub_H2}
\| 
u_b'' 
\|_{L^2(0,T)}
&\ioe 
\| \phi^{(5)} \|_{L^2(0,1)}
|b|^{\frac{1}{41}}.
\end{align}
Thus, the family $(u_b)_{b \in \R}$ is arbitrary small in $H^2_0(0,T)$ and the map $b \mapsto u_b$ from $\R^*$ to $H^2_0(0,T)$ can be continuously extended at zero with $u_0=0$. This concludes Step 1.

\smallskip \noindent \emph{Step 2: Constructing the second approximately reachable vector from the first one.}
Hermes and Kawski proved in \cite[Theorem 6]{HK87} that for affine-control systems of the form $\dot{x}=f_0(x)+uf_1(x)$, if for some Lie bracket $V$ of $f_0$ and $f_1$, $V(0)$ is a tangent vector in the sense of \eqref{methodo:tv_bis}, then $[f_0, V](0)$ is also a tangent vector. 

Using the same construction, we prove that $e_5$ is also a small-time $H^2_0$-continuously approximately reachable vector with vector variations $\Xi_5(T)=e_5$. Denote by $(u_b)_{b \in \R}$ the control variations associated with $e_4$, constructed at step 1. We are going to prove that, for all $(b, c) \in \R^2$ small enough, 
\begin{equation}
\label{tm3_goal}
x(3T; \ u_b \# 0_{[0,T]} \# u_c, \ 0) 
=
(b+c)(e_4 +Te_5)
+
2Tb e_5
+
\O
( \| (b,c)\|^{1+\frac{1}{41}}).
\end{equation}
Thus, taking for all $\alpha \in \R$, $b = \frac{\alpha}{2T}$ and $c=-b$, this proves the existence of a family of controls $(v_{\alpha})_{\alpha \in \R}$ such that, when $\alpha$ goes to zero, 
\begin{equation*}
x(3T; \ v_{\alpha}, \ 0)
=
\alpha e_5
+
\O
(| \alpha|^{1+\frac{1}{41}})
\quad 
\text{ with }
\quad 
\| v_{\alpha}\|_{H^2_0(0,T)} 
\ioe 
C
| \alpha|^{\frac{1}{41}},
\end{equation*}
using \eqref{size_ub_H2}. And, this will conclude Step 2. To prove \eqref{tm3_goal}, notice first that by definition of $(u_b)_{b \in \R}$, one has
\begin{equation}
\label{first_step}
x(T; \ u_b, \ 0)
=
b(e_4+Te_5)
+
\O
(
|b|^{1+\frac{1}{41}}
).
\end{equation}
Then, using the semi-group property of \eqref{toy_model_3}, one has
\begin{equation}
\label{second_step}
x(2T; \ u_b \# 0_{[0,T]}, \ 0)
=
x(T; \ 0_{[0,T]}, \ x(T; \ u_b, \ 0)). 
\end{equation}
Moreover, computing explicitly the solution, one gets a constant $C>0$ such that, 
\begin{equation}
\label{evolution_libre}
\forall p \in \R^5,
\quad
\left\| 
x(T; \ 0_{[0,T]}, \ p)
-
p
-
T p_4 e_5 
\right\|
\ioe 
C
\| p \|^2.
\end{equation}
Thus, \eqref{first_step}, \eqref{second_step} and \eqref{evolution_libre} lead to
\begin{equation}
\label{step2T}
x(2T; \ u_b \# 0_{[0,T]}, \ 0) 
= 
b(e_4+T e_5) 
+ 
Tb e_5 
+ 
\O( |b|^{1+\frac{1}{41}}).
\end{equation}
Then, once again, using the semi-group property, 
\begin{equation}
\label{third_step}
x(3T; \ u_b \# 0_{[0,T]} \# u_c, \ 0) 
=
x(T; \ u_c, \ x(2T; \ u_b \# 0_{[0,T]}, \ 0) ). 
\end{equation}
Moreover, using Grownall Lemma, one gets a constant $C>0$ such that for all $\| u \|_{H^2_0(0,T)} < 1$ and $\| p\| < 1$,
\begin{equation}
\label{Gronwall_dim_finie}
\left\|
x(T; \ u, \ p)
- 
x(T; \ 0_{[0,T]}, \ p)
-
x(T; \ u, \ 0)
\right\|
\ioe 
C \| u \|_{H^2_0(0,T)} \| p\|.
\end{equation}
Thus, \eqref{third_step} and \eqref{Gronwall_dim_finie} lead to 
\begin{multline}
\label{step4}
\|
x(3T; \ u_b \# 0_{[0,T]} \# u_c, \ 0) 
-
x(T; \ u_c, \ 0)
-
x(T; \ 0_{[0,T]}, \ x(2T; \ u_b \# 0_{[0,T]}, \ 0))
\|
\\
\ioe
C
\| u_c\|_{H^2_0}
\|x(2T; \ u_b \# 0_{[0,T]}, \ 0) \|
\ioe
C |c|^{\frac{1}{41}} |b|,
\end{multline}
using the size estimates  \eqref{size_ub_H2} on $(u_c)_{c \in \R}$ and \eqref{step2T} on $x(2T; \ u_b \# 0_{[0,T]}, \ 0)$. Besides, using the estimates \eqref{evolution_libre} and \eqref{step2T},
\begin{equation}
\label{step5}
x(T; \ 0_{[0,T]}, \ x(2T; \ u_b \# 0, \ 0)) 
=
b(e_4+Te_5)+ 2Tbe_5 + \O( | b|^{1+\frac{1}{41}}).
\end{equation}
Using the definition of the family $(u_c)_{c \in \R}$, \eqref{step4} and \eqref{step5} lead to \eqref{tm3_goal}. 

\begin{rem}
\label{rem:toy_model_poly}
The system \eqref{toy_model_3} can be seen as a polynomial toy-model for the Schrödinger PDE \eqref{Schrodinger} in the following way. For both control systems, the linearized system is controllable `up to codimension 2', the leading quadratic (resp.\ cubic) term of the solution along the first lost direction is given by  $\int_0^T u_3(t)^2 dt$ (resp.\ by $\int_0^T u_1(t)^2 u_2(t)dt$) and the second lost direction is more or less the `integration' of the first one. 
\end{rem}

\subsection{A bilinear toy-model for Schrödinger}
Let $p \in \N^*$ and $H_0, H_1 \in \mathcal{M}_p(\mathbb{R})$ symmetric matrices. Consider Schrödinger control systems of the form
\begin{equation}
\label{ODE}
i X'(t) = H_0 X(t) -u(t) H_1X(t),
\end{equation}
where the state is $X(t) \in \mathbb{C}^p$ and the control is $u(t) \in \R$. We write $(\varphi_1, \ldots, \varphi_p)$ for an orthonormal basis of eigenvectors of $H_0$ and $(\lambda_1, \ldots, \lambda_p)$ for its eigenvalues. We also denote by $X_{j}(t):= \varphi_j e^{-i \lambda_j t}$ for all $j \in \{1, \ldots, p\}$. In this section, the commutator of  $H_0$ and $H_1$ is denoted by $[H_0, H_1]:=H_0 H_1 - H_1 H_0$ and $\langle \cdot, \cdot \rangle$ denotes the classical hermitian scalar product on $\C^p$. 

\begin{rem}
\label{adaptation}
For Schrödinger ODEs \eqref{ODE}, we work around the trajectory $(X_1, u \equiv 0)$. The work in \cref{sec:black_box} can still be used by performing the change of function $X^*(t):=X(t)e^{i \lambda_1t} -\varphi_1$ to work around $(0,0)$. 
Thus, in this setting, a vector $\xi \in \R^p$ is called a small-time $E$-continuously approximately reachable vector if there exists a continuous map $\Xi : [0, +\infty)\rightarrow \R^p$ with $\Xi(0)=\xi$ such that for all $T > 0$, there exists $C, \rho, s >0$ and a continuous map $b \in (-\rho, \rho) \mapsto u_b \in E_T$ such that for all $b \in (-\rho, \rho),$
\begin{equation*}
\left\| 
X(T; \ u_b, \ \varphi_1)
- 
X_1(T)
-
b 
\Xi(T)
\right\|
\ioe 
C
|
b
|^{1+s}
\quad
\text{ with }
\quad
\| 
u_b
\|_{E_T}
\ioe 
C
|
b
|^{s}.
\end{equation*}
The topology on the state is not specified as all norms are equivalent in finite dimension. 
\end{rem}

\subsubsection{The linear test} 
The linearized system of \eqref{ODE} around the trajectory $(X_1, u \equiv 0)$ is given by
\begin{equation}
\label{ODE_lin}
i  X'_L = H_0 X_L - u(t) H_1 X_1. 
\end{equation}
By the Duhamel formula, the solution of \eqref{ODE_lin} with $X_L(0)=0$ can be written as 
\begin{equation}
\label{dim_finie_lin}
X_L(T) 
=
i 
\sum \limits_{j=1}^p
\left(
\langle H_1 \varphi_1, \varphi_j \rangle 
\int_0^T 
u(t) 
e^{i (\lambda_j - \lambda_1)t}
dt
\right)
X_j(T)
.
\end{equation}
Thus, the reachable space of the linearized system \eqref{ODE_lin} is given by 
\begin{equation*}
\H
:=
\Span_{\C}
\left(
X_j(T)
\quad
\text{ for } 
j \in \{1, \ldots, p \}
\text{ such that }
\quad
\langle H_1 \varphi_1, \varphi_j \rangle 
\neq 0
\right),
\end{equation*}
as the equality $X_L(T)=X_f$ is brought down to solving a finite polynomial moment problem when the coefficients 
$
\langle H_1 \varphi_1, \varphi_j \rangle 
$
don't vanish. 
To simplify, we assume that, 

\noindent
($\hat{H}_{\lin}$) 
there exists an integer $K \in \{2, \ldots, p\}$ such that 
\begin{equation}
\label{lin_nul_dim_finie}
\langle H_1 \varphi_1, \varphi_K \rangle =0
\quad 
\text{ and }
\quad
\forall j \in \{1, \ldots, p \}-\{K\}, 
\quad
\langle H_1 \varphi_1, \varphi_j \rangle \neq 0. 
\end{equation}
As the solution of the Schrödinger ODE \eqref{ODE} is complex-valued, it means that $\H$ is of codimension 2 and its supplementary is given by $\M=\Span_{\R}(\varphi_K, i \varphi_K)$. 

\subsubsection{Quadratic and cubic behaviors}
\label{sec:expansion_dim_finie}
To prove that $i \varphi_K$ and $\varphi_K$ are approximately reachable vectors, we need to study the behavior of the solution of \eqref{ODE} along the lost directions. Unlike for the previous polynomial toy-models,
here the computations of the first terms of the expansion of \eqref{ODE} are quite heavy. It can be lightened by introducing the new state 
\begin{equation}
\label{def_syst_aux_fin}
\tild{X}(t) := e^{-i H_1 u_1(t)} X(t),
\end{equation}
which solves the following ODE, called the auxiliary system,
\begin{equation}
\label{ODEaux}
\tild{X}'(t) 
= -i e^{-iH_1 u_1(t)} H_0 e^{i H_1 u_1(t)} \tild{X}(t) 
= -i \sum \limits_{k=0}^{+\infty} \frac{ \left(-i u_1(t)\right)^k}{k!} \ad_{H_1}^k(H_0) \tild{X}(t).
\end{equation}
Thus, working with $\tild{X}$, it is easier to quantify the expansion of the solution with respect to the primitives of the control and not with respect to the control $u$. 
This idea was introduced in \cite{C06} and later used in \cite{BM14, B21bis} for the Schrödinger equation. It was also used in finite dimension in \cite{BM18} to study the quadratic behavior of differential systems or in \cite{BLM21} to give refined error estimates for various expansions of scalar-input affine control systems.

\noindent 
By the Duhamel formula, the solution of the auxiliary system \eqref{ODEaux} with $\tild{X}(0)=\varphi_1$ satisfies
\begin{equation}
\label{expr:aux_fin}
\tild{X}(t)
=
X_{1}(t) 
-i
 \int_0^t 
e^{-i H_0 (t-\tau)} 
\sum \limits_{k=1}^{+\infty} \frac{ (-i u_1(\tau))^k}{k!} \ad_{H_1}^k(H_0)
\tild{X}(\tau) d\tau.
\end{equation}
Then, the linear term $\tild{X}_L$, the quadratic term $\tild{X}_Q$ and the cubic term $\tild{X}_C$ of the expansion of $\tild{X}$ around the trajectory $(X_{1}, u \equiv 0)$ are given by,
\begin{align}
\label{expr:X_1,L}
\tild{X}_L(t) &=  -\int_0^t e^{-i H_0 (t-\tau)}  u_1(\tau) \ad^1_{H_1}(H_0) X_{1}(\tau) d\tau, \\
\label{expr:X_1,Q}
\tild{X}_Q(t) &=  \int_0^t e^{-i H_0 (t-\tau)} \left( -u_1(\tau) \ad^1_{H_1}(H_0) \tild{X}_L(\tau) +  \frac{i u_1(\tau)^2}{2} \ad_{H_1}^2(H_0) X_{1}(\tau) \right) d\tau,
\\
\notag
\tild{X}_C(t) &= 
\int_0^t e^{-i H_0 (t-\tau)} \Big( -u_1(\tau) \ad^1_{H_1}(H_0) \tild{X}_Q(\tau) \\
\label{expr:X_1,C}
&
+ 
\frac{i u_1(\tau)^2}{2} \ad_{H_1}^2(H_0) \tild{X}_{L}(\tau) 
+
\frac{u_1(\tau)^3}{6} \ad_{H_1}^3(H_0) X_{1}(\tau)
\Big) d\tau
.
\end{align}
\emph{First-order term.}
Using \eqref{expr:X_1,L}, the linear term along the lost direction is given by
\begin{equation}
\label{bil_etape1}
\langle 
\tild{X}_L(T),
\varphi_K
e^{-i\lambda_1 T}
\rangle
=
(\lambda_K-\lambda_1)
\langle 
H_1 \varphi_1,
\varphi_K
\rangle 
\int_0^T 
u_1(t) 
e^{
i
(\lambda_K- \lambda_1)
(t-T)
}
dt
=0,
\end{equation}
under ($\hat{H}_{\lin}$). The $K$-th direction is also lost at the first order for the auxiliary system.

\smallskip \noindent \emph{Second-order term.}
Substituting the explicit form \eqref{expr:X_1,L} of $\tild{X}_L$ into \eqref{expr:X_1,Q}, the quadratic term along the lost direction is given by, 
\begin{equation}
\label{quad_sans_IPP}
\langle 
\tild{X}_Q(T), 
\varphi_K 
e^{-i \lambda_1 T}
\rangle 
=
-i 
\hat{A}^1_K 
\int_0^T 
u_1(t)^2 
e^{
i
(\lambda_K- \lambda_1)
(t-T)
}
dt
+
\int_0^T 
u_1(t) 
\int_0^t
u_1(\tau)
\hat{k}(t, \tau)
d\tau dt,
\end{equation}
where
\begin{align}
\notag
\hat{A}_K^1
&:= 
- \frac{1}{2}
\langle 
\ad^2_{H_1}(H_0) 
\varphi_1, 
\varphi_K
\rangle
=
\sum 
\limits_{j=1}^{p} 
\left(\lambda_j -  \frac{\lambda_1+\lambda_K}{2} \right) 
\langle 
H_1
\varphi_1, 
\varphi_j
\rangle  
\langle 
H_1
\varphi_K, 
\varphi_j
\rangle,
\\
\label{df_kernel_quad}
\hat{k}(t, \tau)
&:=
\sum \limits_{j=1}^p 
(\lambda_1- \lambda_j)
(\lambda_j - \lambda_K)
\langle H_1 \varphi_1, \varphi_j \rangle
\langle H_1 \varphi_K, \varphi_j \rangle
 e^{i 
 \left(
\lambda_j (\tau-t)
+
\lambda_K(t-T)
+
\lambda_1(T- \tau)
 \right)
 }.
\end{align}
To identify the leading quadratic term, one can compute integrations by parts to get, for all $n \in \N^*$, the existence of a quadratic form $Q_n$ on $\C^{2n}$ such that 
\begin{multline}
\label{quad_IPP}
\langle \tild{X}_Q(T), \varphi_K e^{-i \lambda_1 T}\rangle  
=
-i 
\sum \limits_{m=1}^n 
\hat{A}_K^m
\int_0^T u_m(t)^2 
e^{
i (\lambda_K- \lambda_1)(t-T)
}
dt
\\
+
\int_0^T u_n(t) \int_0^t u_n(\tau) \partial_1^{n-1} \partial^{n-1}_2 \hat{k} (t, \tau) d\tau dt
+
Q_n
\left(
 u_2(T), \ldots, u_n(T), \alpha_2^n, \ldots, \alpha_{n}^n
 \right),
\end{multline}
where, for all $m=2, \ldots, n$, 
\begin{align*}
\alpha_m^n
&:= 
\int_0^T u_n(\tau) \partial_1^m \partial_{2}^{n-1} \hat{k}(T, \tau) d\tau,
\\
\hat{A}_K^m
&:= 
\sum 
\limits_{j=1}^{p} 
\left(\lambda_j -  \frac{\lambda_1+\lambda_K}{2} \right) 
( \lambda_K-\lambda_j)^{m-1} 
(\lambda_j -\lambda_1)^{m-1}
\langle 
H_1
\varphi_1, 
\varphi_j
\rangle  
\langle 
H_1
\varphi_K, 
\varphi_j
\rangle.
\end{align*}
To see more details about this kind of computations, the reader can for example refer to \cite[Section 3.3]{BM20}  or \cite[Section 2.2\ and 5]{B21bis}. To identify the leading quadratic term, one must know which coefficient $\hat{A}^m_K$ is the first to not vanish. From now on, we assume that 

\noindent
($\hat{H}_{\Quad}$) 
$\hat{A}_K^1=\hat{A}_K^2=0$ and $\hat{A}_K^3 \neq 0$. 

\smallskip \noindent
This choice is explained later in \cref{rem:choix_derive}. Then, using \eqref{quad_IPP} for $n=3$ and Cauchy-Schwarz inequality, under ($\hat{H}_{\Quad}$), one gets that  
\begin{equation}
\label{bil_etape2}
\left| 
\langle \tild{X}_Q(T), \varphi_K e^{-i \lambda_1 T}\rangle 
\right|
=
\O
\left(
\| u_3\|^2_{L^2(0,T)}
+
| u_2(T)|^2
+
| u_3(T)|^2
\right).
\end{equation}
Thus, provided that the boundary terms can be neglected, the leading quadratic term of the expansion along the lost direction is $\int_0^T u_3(t)^2 dt.$ 

\smallskip \noindent \emph{Third-order term.}
Substituting the explicit forms \eqref{expr:X_1,L} and \eqref{expr:X_1,Q} of $\tild{X}_L$ and $\tild{X}_Q$ into \eqref{expr:X_1,C}, the cubic term along the lost direction is given by, 
\begin{multline*}
\langle \tild{X}_C(T), \varphi_K e^{-i \lambda_1 T}\rangle 
=
\frac{1}{6}
\langle 
\ad^3_{H_1}(H_0) 
\varphi_1, 
\varphi_K
\rangle
\int_0^T 
u_1(t)^3
e^{
i (\lambda_K- \lambda_1)(t-T)
}
dt
\\
+
\int_0^T 
u_1(t)^2
\int_0^t
u_1( \tau)
\hat{h}_1(t, \tau)
d\tau dt
+ 
\int_0^T 
u_1(t)
\int_0^t
u_1( \tau)^2
\hat{h}_2(t, \tau)
d\tau dt
\\
+
\int_0^T 
u_1(t)
\int_0^t 
u_1( \tau)
\int_0^{\tau}
u_1(s)
\hat{h}_3(t, \tau,s)
ds d\tau dt,
\end{multline*}
where the cubic kernels are given by
\begin{align}
\label{df_kernel_cub_1}
&\hat{h}_1(t, \tau)
:=
\frac{i}
{2}
\sum \limits_{j=1}^p 
(\lambda_j - \lambda_1)
\langle H_1 \varphi_1, \varphi_j \rangle
\langle \ad^2_{H_1}(H_0) \varphi_K, \varphi_j \rangle
e^{
i
\left(
\lambda_K(t-T)
+
\lambda_j(\tau-t)
+
\lambda_1(T-\tau)
\right)
}
,
\\
\label{df_kernel_cub_2}
&\hat{h}_2(t, \tau)
:=
\frac{i}
{2}
\sum \limits_{j=1}^p 
(\lambda_K - \lambda_j)
\langle H_1 \varphi_K, \varphi_j \rangle
\langle \ad^2_{H_1}(H_0) \varphi_1, \varphi_j \rangle
e^{
i
\left(
\lambda_K(t-T)
+
\lambda_j(\tau-t)
+
\lambda_1(T-\tau)
\right)
}
,
\\
\notag 
&\hat{h}_3(t, \tau,s)
:=
\sum \limits_{j=1}^p 
\sum \limits_{n=1}^p 
(\lambda_K - \lambda_j)
(\lambda_1- \lambda_n)
(\lambda_n- \lambda_j)
\langle H_1 \varphi_1, \varphi_n \rangle
\langle H_1 \varphi_n, \varphi_j \rangle
\langle H_1 \varphi_K, \varphi_j \rangle
\\
\label{df_kernel_cub_3}
&\times
e^{
i
\left(
\lambda_K(t-T)
+
\lambda_j(\tau-t)
+
\lambda_1(T-s)
+
\lambda_n(s- \tau)
\right)
}
.
\end{align}
For the Schrödinger PDE \eqref{Schrodinger}, computing formally the Lie bracket, one gets that $\ad^3_{A}(\mu)\varphi=0$ for all $\varphi$. Thus, to have a toy-model fitting the PDE, from now on we assume that 
\begin{equation}
\label{hyp_cubic_en_plus}
\langle 
\ad^3_{H_1}(H_0) 
\varphi_1, 
\varphi_K
\rangle
=0
.
\end{equation} 
Then, the cubic term along the lost direction behaves as 
\begin{multline}
\label{bil_etape3}
\langle \tild{X}_C(T), \varphi_K e^{-i \lambda_1 T}\rangle 
=
\int_0^T 
u_1(t)^2
\int_0^t
u_1( \tau)
\hat{h}_1(t, \tau)
d\tau dt
\\
+ 
\int_0^T 
u_1(t)
\int_0^t
u_1( \tau)^2
\hat{h}_2(t, \tau)
d\tau dt
+
\O 
(
\| u_1\|^3_{L^1(0,T)}
).
\end{multline} 
The last cubic term is neglected in front of the other two in an asymptotic of small-time, and thus, is seen as a small pollution. 

\smallskip \noindent \emph{Error estimate on the expansion.}
With a similar proof as in \cite[Prop.\ 2.5]{B21bis}, one can compute the following error estimate, 
\begin{equation}
\label{bil_etape4}
\| 
(
\tild{X}
-
X_1
- 
\tild{X}_L
-
\tild{X}_Q
-
\tild{X}_C
)(T)
\|
=
\O( \| u_1\|^4_{L^4(0,T)}).
\end{equation}

\begin{rem}[About assumption $(\hat{H}_{\Quad}$)]
\label{rem:choix_derive}
When \eqref{hyp_cubic_en_plus} holds, the equality \eqref{bil_etape3} gives that in an asymptotic of small-time, the leading cubic term is $\int_0^T u_1(t)^2 u_2(t)dt$. 
\begin{itemize}
\item If, instead of  $(\hat{H}_{\Quad}$), we assume that $\hat{A}^1_K \neq 0$, by \eqref{quad_sans_IPP}, the leading quadratic term is $\int_0^T u_1(t)^2 dt$. Thus, for controls small in $W^{-1, \infty}$, the quadratic term prevails on the cubic term and \eqref{ODE} is not $W^{-1, \infty}$. Thus, one must at least assume that $\hat{A}^1_K =0$.
\item If, instead of  $(\hat{H}_{\Quad}$), we assume that $\hat{A}^1_K =0$ and $\hat{A}^2_K \neq 0$, then \eqref{quad_IPP} with $n=2$ gives that the leading quadratic term is $\int_0^T u_2(t)^2 dt$. 
However, the cubic term can't absorb simultaneously such a quadratic term and the quartic term as by Cauchy-Schwarz inequality, one has
\begin{equation*}
\left|
\int_0^T u_1(t)^2 u_2(t) dt 
\right|^2
\ioe
\int_0^T u_2(t)^2 dt 
\int_0^T u_1(t)^4 dt.
\end{equation*}
To overcome this issue, one could try to prove a sharper error estimate than \eqref{bil_etape4}. Instead of doing this, we assume $(\hat{H}_{\Quad}$) so that the leading quadratic term is given by $\int_0^T u_3(t)^2dt$. This time, the cubic term can handle simultaneously such a quadratic term and the terms of order higher than four.
\end{itemize}
\end{rem}

To sum up, the goal of this section is to prove the following result. 
\begin{thm}
Let $H_0$ and $H_1$ satisfying 
($\hat{H}_{\lin}$),
($\hat{H}_{\Quad}$),
\eqref{hyp_cubic_en_plus}, 
and 
($\hat{H}_{\Cub}$). 
Then, the Schrödinger ODE \eqref{ODE} is $H^2_0$-STLC around the ground state: for all $T>0$, for all $\eta>0$, there exists $\delta> 0$ such that for every $X_f \in \C^p$ with $\|X_f - X_1(T)\| < \delta$,  there exists $u \in H^2_0((0,T),\R)$ with $\|u\|_{H^2_0(0,T)} < \eta$ such that the solution $X$ of \eqref{ODE} satisfies
\begin{equation*}
X(T; \ u, \ \varphi_1) = X_f. 
\end{equation*}
\end{thm}

\subsubsection{$i \varphi_K$ is a small-time $H^2_0$-continuously approximately reachable vector}
\label{dim_finie_TV1}
Working in two stages as for the polynomial toy-models \eqref{toy_model_2} and \eqref{toy_model_3}, we prove that $i \varphi_K$ is a small-time $H^2_0$-continuously approximately reachable vector associated with vector variations $\Xi(T)= i X_K(T)$. 
\begin{itemize}
\item First, the computations of \cref{sec:expansion_dim_finie} entail the existence of a family of oscillating controls $(u_b)_{b \in \R}$ arbitrary small in $H^2_0(0,T)$ such that 
\begin{equation}
\label{step1}
\langle 
X(T; \ u_b, \ \varphi_1), 
X_K(T)
\rangle
=
i
b
+
\O( |b|^{1+\frac{1}{41}}) 
\quad
\text{ when }
b \rightarrow 0. 
\end{equation}
\item Then, we make sure that
\begin{equation}
\label{step2}
\left\|
\P 
X(T; \ u_b, \ \varphi_1)
-
X_1(T)
\right\|
=
\O( |b|^{1+\frac{1}{41}}),
\end{equation}
where we recall that $\P$ denotes the orthogonal projection on $\H$. 
\end{itemize}

\
\paragraph{\textit{Step 1: Using oscillating controls along the lost direction.}}
Define, for $b \in \R^*$, 
\begin{equation}
\label{def_ub_dim_finie}
u_b(t)
=
\sign(b)
|b|^{\frac{7}{41}}
\phi^{(3)}
\left(
\frac{
t
}
{
|b|^{\frac{4}{41}}
}
\right)
\text { with }
\phi \in C_c^{\infty}(0,1)
\text{ s.\ t.\ }
\int_0^1 \phi''(\theta)^2 \phi'(\theta) d\theta
=
\frac{1}{\hat{C}_K},
\end{equation}
with
$
\hat{C}_K
:=
\hat{h}_1(0,0)
-
\hat{h}_2(0,0)
$ 
where $\hat{h}_1$ and $\hat{h}_2$ are defined in \eqref{df_kernel_cub_1} and \eqref{df_kernel_cub_2}. 
To ensure that $\phi$ exists, one must assume that

\smallskip \noindent
($\hat{H}_{\Cub}$) $\hat{C}_K \neq 0$. 

\noindent First, the same estimate as \eqref{size_ub_H2} gives that the controls are arbitrary small in $H^2_0$ and the map $b \mapsto u_b$ can be extended continuously from $\R$ to $H^2_0(0,T)$. Moreover, substituting these controls into \eqref{bil_etape1}, \eqref{bil_etape2}, \eqref{bil_etape3} and \eqref{bil_etape4}, the same computations as for \eqref{toy_model_3} lead to 
\begin{multline*}
\langle \tild{X}(T; \ u_b, \ \varphi_1), \varphi_K e^{-i \lambda_1T} \rangle
=
b
\int_0^1 
\phi''(\theta_1)^2
\int_0^{\theta_1}
\phi''(\theta_2)
\hat{h}_1
( 
|b|^{
\frac{4}
{41}
}
\theta_1, 
|b|^{
\frac{4}
{41}
}
\theta_2
)
d \theta_2
d\theta_1
\\
+
b
\int_0^1 
\phi''(\theta_1)
\int_0^{\theta_1}
\phi''(\theta_2)^2
\hat{h}_2
( 
|b|^{
\frac{4}
{41}
}
\theta_1, 
|b|^{
\frac{4}
{41}
}
\theta_2
)
d \theta_2
d\theta_1
+
\O( 
|b|^{
1
+
\frac{1}
{41}
}
).
\end{multline*}
Performing the expansion of the kernels $\hat{h}_1$ and $\hat{h}_2$ when $b$ goes to zero, one has
\begin{equation*}
\langle \tild{X}(T; \ u_b, \ \varphi_1), \varphi_K e^{-i \lambda_1T} \rangle
=
i 
\hat{C}_K
b 
\int_0^1 \phi''(\theta)^2 \phi'(\theta) d\theta
e^{i (\lambda_1-\lambda_K)T}
+
\O( 
|b|^{
1
+
\frac{1}
{41}
}
),
\end{equation*}
which gives \eqref{step1} by construction of $\phi$ as $\tild{X}(T)=X(T)$ when $u_1(T)=0$ (see \eqref{def_syst_aux_fin}). 

\
\paragraph{\textit{Step 2: Correcting the linear components.}} 
Unlike for the previous polynomial toy-models, it is not straightforward to make sure that \eqref{step2} holds. Thus, in a second time, the linear components of the solution are corrected using the STLC result on projection on $\H$ given in \cref{lem:linear_test}.
This gives the existence of a control $v_b \in H^2_0(T, 2T)$ such that the solution of \eqref{ODE} on $[T, 2T]$  associated to the control $v_b$ and the initial condition $X(T; \ u_b, \ \varphi_1)$ at time $T$ satisfies 
\begin{equation}
\label{size_vb}
\P X(2T) = X_1(2T)
\quad
\text{ with }
\quad
\| v_b \|_{H^2_0(T, 2T)}
\ioe 
C 
\| 
X(T; \ u_b, \ \varphi_1)
-
X_1(T)
\|.
\end{equation}
Then, \eqref{step2} is verified (for the final time $2T$). However, one needs to check that \eqref{step1} which holds at time $T$ thanks to Step 1, still holds at time $2T$ and has not been destroyed by the linear correction. Thus, one needs to check that 
\begin{equation}
\label{correction}
\left|
\langle X(2T), X_K(2T) \rangle
-
\langle X(T), X_K(T) \rangle
\right|
=
\O(|b|^{1+\frac{1}{41}})
.
\end{equation}

\noindent \emph{Evolution of the solution along the lost direction.}
First, we prove that under ($\hat{H}_{\lin}$), one has, 
\begin{equation}
\label{evolution_sol}
\left|
\langle X(2T), X_K(2T) \rangle
-
\langle X(T), X_K(T) \rangle
\right|
\ioe
C
\| v_b\|_{L^1(T,2T)}^2
.
\end{equation}
To that end, first notice that, using \eqref{dim_finie_lin}, under ($\hat{H}_{\lin}$), one has, 
\begin{equation*}
\forall t \in [0,T], 
\quad
\langle 
X(t), X_K(t)
\rangle
=
\langle 
(X-X_1-X_L)(t)
,
X_K(t)
\rangle. 
\end{equation*}
Besides, looking at \eqref{ODE} and \eqref{ODE_lin}, $X-X_1-X_L$ is the solution of 
\begin{equation*}
i (X-X_1-X_L)'= H_0 (X-X_1-X_L) -u(t) H_1 (X-X_1).
\end{equation*}
Thus, the Duhamel formula gives that the left-hand side of \eqref{evolution_sol} is estimated by 
\begin{multline}
\label{estim:non_add_df}
\left|
\int_T^{2T} 
v_b(t) 
\langle 
H_1 (X-X_1)(t), 
\varphi_K
\rangle
e^{-i\lambda_K (2T-t)}
dt
\right|
\\
\ioe 
C
\| v_b \|_{L^1(T, 2T)} \sup_{t \in [T, 2T]} \| (X-X_1)(t)\|. 
\end{multline}
Moreover, writing the Duhamel formula for the equation satisfied by $X-X_1$, one gets similarly that 
\begin{equation}
\label{df_estim_lin} 
\sup_{t \in [T, 2T]} \| (X-X_1)(t)\| \ioe C \| v_b\|_{L^1(T,2T)} \sup_{t \in [T, 2T]} \| X(t) \|.
\end{equation}
Besides, taking the scalar product of \eqref{ODE} with $X$ and then the imaginary part of the corresponding equality, one gets that the norm of $X$ is preserved. Thus, putting together all these estimates, one gets 
\eqref{evolution_sol}.

\smallskip \noindent \emph{Estimate on the end-point of the solution at time $T$.}
Thus, to prove \eqref{correction}, using \eqref{size_vb} and \eqref{evolution_sol}, it is enough to prove that 
\begin{equation}
\label{goal_estim}
\|
X(T; \ u_b, \ \varphi_1)
-
X_1(T)
\|^2
=
\O(|b|^{1+\frac{1}{41}}),
\end{equation}
that is, we need to estimate the error on the linear part of the solution when using the oscillating controls \eqref{def_ub_dim_finie}. Notice that by the Duhamel formula, as in \eqref{df_estim_lin}, a straightforward estimate  is given by 
\begin{equation*}
\| X(T; \ u_b, \ \varphi_1)
-
X_1(T)
\|
\ioe
C
\| u_b\|_{L^1(0,T)}
=
\O(
|b|^{\frac{11}{41}}
),
\end{equation*}
by definition \eqref{def_ub_dim_finie} of the family $(u_b)_{b \in \R}$. This is not enough to prove \eqref{goal_estim}. One can compute a sharper estimate by writing instead that
\begin{equation}
\label{goal_estim_0}
\| 
X(T; \ u_b, \ \varphi_1)
-
X_1(T)
\|
\ioe 
\| 
X_L(T; \ u_b, \ \varphi_1)
\|
+
\| (X-X_1-X_L)(T; \ u_b, \ \varphi_1)
\|.
\end{equation}
Moreover, using the estimate given in \cite[Prop.\ 2.6]{B21bis}, one has 
\begin{equation}
\label{goal_estim_1}
\| (X-X_1-X_L)(T; \ u_b, \ \varphi_1)
\|
\ioe
C
\| u_1\|_{L^2(0,T)}^2
=
\O
(
|b|^{
\frac{26}{41}
}
),
\end{equation}
looking at \eqref{def_ub_dim_finie}. Moreover, looking at the explicit computations given in \eqref{dim_finie_lin}, the linear part is estimated by
\begin{equation}
\label{goal_estim_2}
\| 
X_L(T; \ u_b, \ \varphi_1)
\|
\ioe 
C
\max_{j=1, \ldots, p}
\left|
\int_0^T u_b(t) e^{i(\lambda_j-\lambda_1)(t-T)} dt
\right| 
.
\end{equation}
Besides, for every $j=2, \ldots, p$, performing three integrations by parts as $u_1(T)=u_2(T)=u_3(T)=0$, one has
\begin{equation}
\label{goal_estim_3}
\left|
\int_0^T u_b(t) e^{i(\lambda_j-\lambda_1)(t-T)} dt
\right| 
=
\left|
(\lambda_j-\lambda_1)^3
\int_0^T u_3(t) e^{i(\lambda_j-\lambda_1)(t-T)} dt
\right|
=
\O( 
|b|^{
\frac{23}{41}
}
)
\end{equation}
looking at \eqref{def_ub_dim_finie}. Notice that this also holds for $j=1$ because in this case, the left hand-side is equal to $u_1(T)=0$. Therefore, \eqref{goal_estim_0}, \eqref{goal_estim_1}, \eqref{goal_estim_2}, \eqref{goal_estim_3} lead to 
\begin{equation*}
\| 
X(T; \ u_b, \ \varphi_1)
-
X_1(T)
\|
=
\O(
|b|^{\frac{23}{41}}
),
\end{equation*}
which gives \eqref{goal_estim} concluding Step 2. 

\subsubsection{$\varphi_K$ is a small-time $H^2_0$-continuously approximately reachable vector}
\label{dim_finie_TV2}
As for the polynomial toy-model \eqref{toy_model_3}, the second approximately reachable vector is built from the first one in a way inspired by the work \cite[Th.\ 6]{HK87}.  
To that end, denote by $(u_b)_{b \in \R}$ the control variations associated with $i \varphi_K$ constructed in \cref{dim_finie_TV1}. The goal is to prove that for every $b, c \in \R$ small enough, 
\begin{equation}
\label{goal_TV2}
X(3T; \ u_b \# 0_{[0, T]} \# u_c, \ \varphi_1) 
=
X_1(3T)
+
(
i 
c
e^{
2i( \lambda_K-\lambda_1)T
}
+
i
b
)
X_K(3T)
+
\O( | (b, c) |^{1+\frac{1}{41}})
.
\end{equation}
Thus, for all $T \in \left(0, \frac{\pi}{2(\lambda_K-\lambda_1)} \right)$, taking $c=-\frac{\alpha}{\sin(2(\lambda_K-\lambda_1)T)}$ and $b=-c \cos(2(\lambda_K-\lambda_1)T)$, this provides a family of controls $(v_{\alpha})_{\alpha \in \R}$ such that when $\alpha$ goes to zero, 
\begin{equation*}
X(3T; \ v_{\alpha}, \ \varphi_1)
=
X_1(3T)
+
\alpha 
X_K(3T)
+
\O
(
|\alpha|^{1+\frac{1}{41}}
)
\quad
\text{ with }
\| v_{\alpha}\|_{H^2_0(0,3T)} \ioe C |\alpha|^{\frac{1}{41}}.
\end{equation*}
This will give that $\varphi_K$ is a small-time $H^2_0$-continuously approximately reachable vector associated with vector variations $\Xi(T)= X_K(T)$.
So, it remains to prove \eqref{goal_TV2}. First, by construction of the family $(u_b)_{b \in \R}$, one has
\begin{equation*}
X(T; \ u_b, \ \varphi_1)
=
X_1(T)
+ 
i b X_K(T)
+
\O( |b|^{1+\frac{1}{41}})
\quad
\text{ with }
\| u_b\|_{H^2_0(0,T)}
\ioe 
C
|b|^{\frac{1}{41}}.
\end{equation*}
Then, on $[T, 2T]$, no control is activated, so, 
\begin{equation*}
X(2T; \ u_b \# 0_{[0,T]}, \ \varphi_1)
=
e^{-iH_0 T}X(T; \ u_b, \ \varphi_1)
=
X_1(2T)
+
ib X_K(2T)
+
\O( |b|^{1+\frac{1}{41}}).
\end{equation*}
Moreover, using the semi-group property of the equation, one has
\begin{equation}
\label{sur_3T_4T}
X(3T; \ u_b \# 0_{[0,T]} \# u_c, \ \varphi_1)
=
X(T; \ u_c, \ X(2T; \ u_b \# 0_{[0,T]}, \ \varphi_1))
.
\end{equation}
Besides, using Gronwall Lemma, one proves the existence of $C>0$ such that for all $\tau>0$, $p \in \R^5$ with $\| p\| < 1$ and $u \in H^2_0(0,T)$ with $\| u\|_{H^2_0} <1$, 
\begin{equation}
\label{Gronwall_df}
\| 
X(T; \ u, \ X_1(\tau)+p)
-
X(T; \ u, \ \varphi_1) e^{-i \lambda_1 \tau}
-
e^{-i H_0 T} p
\|
\ioe
C
\| 
u
\|_{H^2(0,T)}
\| 
p 
\|.
\end{equation}
The proof is left to the reader but one can refer to \cref{prop:dep_ci} for a similar proof for the Schrödinger PDE. 
Taking $u=u_c$, $\tau=2T$ and $p_b=ib X_K(2T) + \O( |b|^{1+\frac{1}{41}})$, one gets that $\| u_c\|_{H^2(0,T)} \| p_b\|=\O(|c|^{\frac{1}{41}} |b|)$. Moreover, by construction of $(u_c)$, 
$$X(T; \ u_c, \ \varphi_1)= X_1(T) + i cX_K(T) + \O( |c|^{1+\frac{1}{41}}).$$
Thus, using \eqref{Gronwall_df}, \eqref{sur_3T_4T} becomes 
\begin{equation*}
X(3T; \ u_b \# 0_{[0, T]} \# u_c, \ \varphi_1) 
=
X_1(3T)+ ic X_K(T) e^{-2 i \lambda_1 T} + ib X_K(3T) + \O( |b, c|^{1+\frac{1}{41}}),
\end{equation*}
which concludes the proof of \eqref{goal_TV2}. 

\subsubsection{Towards the Schrödinger PDE}
Let us state here the main difficulties we are going to face for the Schrödinger PDE \eqref{Schrodinger} compared to the ODE \eqref{ODE}.
\begin{itemize}
\item In \cref{sec:expansion}, the computations of the expansion of the solution will be quite similar. The only difference is that the kernels will be defined as function series. Thus, the regularity and boundness of such kernels needed to perform integrations by parts will not be straightforward but will stem from (H$_{\reg})$. 

\item The main difficulty for the PDE will be to prove that $i \varphi_K$ is a $H^2_0$-continuously approximately reachable vector (the second approximately reachable vector will be deduced from the first with the same proof). Using the same oscillating controls as for the ODE, we will have similarly that
\begin{equation}
\label{idee1}
\langle 
\psi(T; \ u_b, \ \varphi_1)
,
\psi_K(T)
\rangle 
=
i b
+
\O( 
|b|^{1+\frac{1}{41}}
).
\end{equation}
Then, contrary to the finite dimensional case, we will need to correct an infinite number of linear directions. This will also be done using the STLC in projection on $\H$ to get the existence of $v_b \in H^2_0(T, 2T)$ such that 
\begin{equation*}
\P \psi(2T) = \psi_1(2T).
\end{equation*}
The core of the paper is to prove that such a linear correction didn't destroy the work in \eqref{idee1}. This is done using two ingredients.
\begin{itemize}
\item The STLC result in projection provides an estimate, \eqref{size_vb} in finite dimension, on the linear control by the data to be reached. For the Schrödinger PDE, the classical estimate giving that the $L^2$-norm of the control is estimated by the data to be controlled in the $H^{3}_{(0)}$-norm is not sharp enough. The whole work of \cite{B21} has consisted in establishing sharper (and simultaneous) estimates on the control to make this step work.
\item Also, in finite dimension, the evolution of the solution along the lost dimension \eqref{evolution_sol} is estimated by the $L^1$-norm of the linear control. Once again, this will not be sharp enough for the Schrödinger PDE. That is why in \cref{sec:non_add}, we quantify more precisely the evolution of the solution along the lost direction.
\end{itemize}
\end{itemize}

\section{Well-posedness and STLC of the Schrödinger equation}
\label{sec:WP_Schro}

\subsection{Well-posedness of the Schrödinger equation}
In this section, we recall the result given in \cite[Theorem 2.1]{B21} about the existence and uniqueness of the solution of the following Cauchy problem, stressing the link between the regularity of the solution and the boundary conditions on the dipolar moment $\mu$,
\begin{equation}  
\label{Schro_source_term} 
\left\{
    \begin{array}{ll}
        i \partial_t \psi(t,x) = - \partial^2_x \psi(t,x) -u(t)\mu(x)\psi(t,x) -f(t,x),  \quad &(t,x) \in (0,T) \times (0,1),\\
        \psi(t,0) = \psi(t,1)=0, \quad &t \in (0,T), \\
        \psi(0,x) = \psi_0(x), \quad &x \in (0,1).
    \end{array}
\right.  
\end{equation}
\begin{thm}
\label{wp}
Let $T>0$, $(p,k) \in \N^2$, $\mu \in H^{2(p+k)+3}( (0,1), \R)$ with $\mu^{(2n+1)}(0)=\mu^{(2n+1)}(1)=0$ for all $n=0, \ldots, p-1$ , $u \in H^{k}_0( (0,T), \R)$, $\psi_0 \in H^{2(p+k)+3}_{(0)}(0,1)$ and $f \in H^{k}_0 ( (0,T), H^{2p+3} \cap H^{2p+1}_{(0)}(0,1))$.
There exists a unique solution of \eqref{Schro_source_term}, that is a function $\psi \in C^{k}( [0,T], H^{2p+3}_{(0)}(0,1))$ with $\psi(T)$ in $H^{2(p+k)+3}_{(0)}(0,1)$ such that the following equality holds in $H^{2p+3}_{(0)}$ for every $t \in [0,T]$:
\begin{equation*}
\psi(t) = e^{-iAt} \psi_0 + i \int_0^t e^{-iA (t- \tau)} \left( u(\tau) \mu \psi(\tau) + f(\tau) \right) d\tau.
\end{equation*}
Moreover, for every $R>0$, there exists $C=C(T, \mu, R)>0$ such that if $\| u \|_{H^k_0(0,T)} < R$, then this solution satisfies
\begin{equation*}
\| \psi(T)\|_{H^{2(p+k)+3}_{(0)}}, 
\
\| \psi \|_{C^k( [0,T], H^{2p+3}_{(0)})} 
\ioe 
C
\left(
\| \psi_0 \|_{H^{2(p+k)+3}_{(0)}} 
+
\| f \|_{H^k( (0,T), H^{2p+3} \cap H^{2p+1}_{(0)})}
\right).
\end{equation*}
\end{thm}
We will sometimes write $\psi( \cdot ; \ u, \ \psi_0)$ to denote the solution of \eqref{Schrodinger} associated with control $u$ and initial data $\psi_0$ when we will need to keep track of such a dependency. 

\begin{rem}
\label{rem:cont_mu}
Notice that when $\mu$ satisfies (H$_{\reg}$), the multiplication operators 
\begin{align}
\label{cont_mu}
\varphi 
\quad 
&\mapsto 
\quad 
\mu 
\varphi,
\\
\label{cont_exp}
\varphi 
\quad 
&\mapsto 
\quad 
e^{ i \alpha \mu}
\varphi,
\quad 
\alpha \in \R,
\end{align}
maps continuously $H^7_{(0)}$ and $H^7 \cap H^5_{(0)}$ into $H^7 \cap H^5_{(0)}$ but does not map continuously $H^7_{(0)}$ into $H^7_{(0)}$.
Moreover, the operator
\begin{equation}
\label{cont_der}
\varphi 
\quad 
\mapsto 
\quad 
2
\mu' 
\varphi'
+
\mu''
\varphi,
\end{equation}
maps continuously $H^7 \cap H^5_{(0)}$ into $H^6 \cap H^3_{(0)}$ and the operator 
\begin{equation}
\label{cont_mu_der}
\varphi 
\quad 
\mapsto 
\quad 
\mu'^2 
\varphi,
\end{equation}
maps continuously $H^7 \cap H^5_{(0)}$ into $H^7 \cap H^5_{(0)}$. 
Indeed, for \eqref{cont_mu}, Leibniz formula gives for $n=0,1,2$, 
\begin{equation*}
( \mu \varphi)^{(2n)}
=
\sum \limits_{k=0}^n
\binom{2n}{2k}
\mu^{(2k)}
\varphi^{(2n-2k)}
+
\sum \limits_{k=0}^{n-1}
\binom{2n}{2k+1}
\mu^{(2k+1)}
\varphi^{(2n-2k-1)}
.
\end{equation*}
Thus, if $\varphi \in H^7_{(0)}$, for all $n=0, 1, 2$, $( \mu \varphi)^{(2n)}$ vanishes at $x=0,1$ because for all $k \in \{0, \ldots, n\}$, $\varphi^{(2n-2k)}$ does and for all $k \in \{0, \ldots, n-1\}$, $\mu^{(2k+1)}$ does. This gives the continuity of \eqref{cont_mu}. Notice that one can't go higher as for $( \mu \varphi)^{(6)}$, in the sum, the term $\mu^{(5)} \varphi'$ doesn't vanish at $0$ and $1$. The other continuities are proved the same.  These continuities will be the key to prove the well-posedness of the equations considered in the following (see \cref{rem:wp}, Sections \ref{an_auxiliary_system} and \ref{expansion_aux}).
\end{rem}

\begin{rem}
\label{rem:wp}
Thanks to \cref{rem:cont_mu} on \eqref{cont_mu}, from \cref{wp} with $p=k=2$, one deduces that, when $\mu$ satisfies (H$_{\reg}$), for every $\psi_0 \in H^{11}_{(0)}$, $\phi \in C^2( [0,T], H^{7}_{(0)})$ and $u, v \in H^2_0(0,T)$, the Schrödinger equation 
\begin{equation*}  
\left\{
    \begin{array}{ll}
        i \partial_t \psi(t,x) = - \partial^2_x \psi(t,x) -u(t)\mu(x) \psi - v(t)\mu(x)\phi,  \quad &(t,x) \in (0,T) \times (0,1),\\
        \psi(t,0) = \psi(t,1)=0, \quad &t \in (0,T), \\
        \psi(0,x) = \psi_0(x), \quad &x \in (0,1).
    \end{array}
\right.  
\end{equation*}
admits a unique solution $\psi \in C^{2}( [0,T], H^{7}_{(0)}(0,1))$ with $\psi(T)$ in $H^{11}_{(0)}(0,1)$. Moreover, for every $R>0$, there exists $C=C(T, \mu, R)>0$ such that if $\| u \|_{H^2_0(0,T)} < R$,  this solution satisfies
\begin{equation}
\label{estim_sol_bis}
\| \psi(T)\|_{H^{11}_{(0)}}, 
\quad
\| \psi \|_{C^2( [0,T], H^{7}_{(0)})} 
\ioe 
C
\left(
\| \psi_0 \|_{H^{11}_{(0)}} 
+
\| v \|_{H^2_0(0,T)}
\| \phi \|_{C^2( [0,T], H^{7}_{(0)})}
\right).
\end{equation}
This will be the regularity on solutions used in all this paper. 
\end{rem}

\subsection{Dependency of the solution with respect to the initial condition}
From the well-posedness result given in \cref{wp}, one can deduce the following result about the dependency of the solution of \eqref{Schrodinger} with respect to the initial condition. 
\begin{prop}
\label{prop:dep_ci}
Let $T>0$, $\mu$ satisfying (H$_{\reg}$) and $\psi_0 \in H^{11}_{(0)}(0,1)$ and $\tau \in \R$. For all $R>0$, there exists $C=C(T, \mu, R)>0$ such that for all $u \in H^2_0(0,T)$ with $\| u \|_{H^2_0(0,T)}<R$, one has
\begin{equation*}
\| \psi(T; \ u, \ \psi_1(\tau) +\psi_0) -\psi(T; \ u, \ \varphi_1)e^{-i\lambda_1 \tau} - e^{-iAT}\psi_0 \|_{H^{11}_{(0)}} 
\ioe 
C \| u \|_{H^2_0(0,T)} \| \psi_0 \|_{H^{11}_{(0)}}.
\end{equation*}
\end{prop}

\begin{proof}
Define, for all $t \in [0,T]$,
$
\Lambda(t):=\psi(t; \ u, \ \psi_1(\tau) +\psi_0) -\psi(t; \ u, \ \varphi_1)e^{-i\lambda_1 \tau} - e^{-iAt}\psi_0. 
$
Notice that $\Lambda$ is the solution of,
\begin{equation*}
i \partial_t \Lambda = - \partial^2_x \Lambda-u(t)\mu(x)\Lambda -u(t) \mu(x) e^{-iAt}\psi_0,
\end{equation*}
with Dirichlet conditions and $\Lambda(0,\cdot)=0$.
%
%
Therefore, \cref{rem:wp} gives the existence of $C>0$ such that 
\begin{equation*}
\| \Lambda(T) \|_{H^{11}_{(0)}} 
\ioe C  \| u \|_{H^2_0(0,T)} \| e^{-iAt} \psi_0 \|_{C^2( [0,T], H^{7}_{(0)})}
=
C  \| u \|_{H^2_0(0,T)} \| \psi_0 \|_{ H^{11}_{(0)} }
.
\end{equation*}
\end{proof}

\subsection{Controllability in projection by the linear test with simultaneous estimates}
In this section, we recall the local controllability result in projection by the linear test given in \cite{B21} as it will be useful in this paper. To that end, we introduce the following notations: if $J$ is a subset of $\N^*$, we define the space
$
\H := \overline{\Span_{\C}} \left( \varphi_j , \ j \in J \right)
$
and the orthogonal projection on $\H$ given by 
$
\P_{\J}(\psi)
=
\psi - \sum \limits_{j \not\in J} \langle \psi, \varphi_j \rangle \varphi_j
$
for all $\psi \in L^2(0,1)$.

\begin{thm}
\label{linear_STLC}
Let $(p, k) \in \N^2$ with $p \soe k$, $J$ a subset of $\N^*$ and $\mu \in H^{2(p+k)+3}( (0,1), \R)$ such that $\mu^{(2n+1)}(0)=\mu^{(2n+1)}(1)=0$ for all $n=0, \ldots, p-1$ and 
\begin{equation}
\label{hyp_mu}
\text{there exists a constant } c>0 \text{ such that for all } j \in J, 
\quad
| \langle \mu \varphi_1, \varphi_j \rangle | \geq \frac{c}{j^{2p+3}}.
\end{equation}
Then, the Schrödinger equation \eqref{Schrodinger} is STLC in projection around the ground state with controls in $H^m_0(T_0,T)$ and targets $H^{2(p+m)+3}_{(0)}$ for every $m \in \{0, \ldots, k \}$ with the same control map. 

\smallskip \noindent 
More precisely, for all initial time $T_0 \soe 0$ and final time $T> T_0$, there exists $C$, $\delta >0$ and a $C^1$-map $\Gamma_{T_0,T} : \Omega_{T_0} \times \Omega_T \rightarrow H^k_0( (T_0,T), \R)$ where 
\begin{align}
\label{def_Omega_T0}
\Omega_{T_0} &:= \{ \psi_0 \in \S \cap H^{2(p+k)+3}_{(0)} ; \ \| \psi_0 - \psi_1(T_0) \|_{H^{2(p+k)+3}_{(0)}} < \delta \},
\\
\label{def_Omega_T}
\Omega_T &:= \{ \psi_f \in \H \cap H^{2(p+k)+3}_{(0)} ; \ \| \psi_f - \P_{\J} \left( \psi_1(T) \right) \|_{H^{2(p+k)+3}_{(0)}} < \delta \},
\end{align}
such that $\Gamma_{T_0,T}(\psi_1(T_0), \psi_1(T))=0$ and for every $(\psi_0, \psi_f) \in \Omega_{T_0} \times \Omega_T$, the solution of \eqref{Schrodinger} on $[T_0, T]$ with control $u:=\Gamma_{T_0,T}(\psi_0,\psi_f)$ and initial condition $\psi_0$ at $t=T_0$ satisfies 
\begin{equation}
\label{contr_proj}
 \P_{\J} \left( \psi(T) \right)= \psi_f,
\end{equation}
with the following boundary conditions on the control
\begin{equation}
\label{eq:weak_bc_nl}
u_2(T)= \ldots= u_{k+1}(T)=0, 
\end{equation}
where here $(u_n)_{n \in \N}$ denotes the iterated primitives of $u$ vanishing at $T_0$. Besides, for all $m$ in $\{-(k+1), \ldots, k\}$, the following estimates hold
\begin{equation}
\label{estim_contr_nl}
\| u \|_{H^m_0(T_0,T)}
\ioe
C
\left(
\| \psi_0 - \psi_1(T_0) \|_{H^{2(p+m)+3}_{(0)}} 
+
\| \psi_f - \P_{\J} \psi_1(T) \|_{H^{2(p+m)+3}_{(0)}} 
\right)
.
\end{equation}
\end{thm}

\section{Error estimates on the expansion of the solution}
\label{sec:expansion}
The goal of this section is to compute the power series expansion of the solution $\psi$ of the Schrödinger equation \eqref{Schrodinger} up to order 3 with a sharp error estimate, as it is the key to prove \cref{the_theorem}.  In all this section, if not mentioned, we will work with controls $u$ at least in $H^2_0(0,T)$ and with a dipolar moment $\mu$ satisfying (H$_{\reg})$.

\subsection{Formal expansion of the solution}
Formally, expanding the solution of \eqref{Schrodinger} around the trajectory $(\psi_1, u \equiv 0)$, 
\begin{itemize}
\item the first-order term $\Psi$ of the expansion of $\psi$ is solution of, 
\begin{equation}  
\left\{ 
    \begin{array}{ll}
        i \partial_t \Psi = - \partial^2_x \Psi -u(t)\mu(x) \psi_1(t,x) , \\
        \Psi(t,0) = \Psi(t,1)=0,\\
        \Psi(0,x)=0,
    \end{array}
\right. 
\label{order1} 
\end{equation}
which can be explicitly computed as, 
\begin{equation}
\label{order1explicit}
\Psi(t)=i 
\sum \limits_{j=1}^{+\infty}
\left(
 \langle \mu \varphi_1, \varphi_j\rangle  \int_0^t u(\tau) e^{i (\lambda_j-\lambda_1)\tau} d\tau
 \right)
  \psi_j(t), \quad t \in [0,T].
\end{equation}

\item The second-order term $\xi$ of the expansion of $\psi$ is solution of, 
\begin{equation}  
\left\{ 
    \begin{array}{ll}
        i \partial_t \xi = - \partial^2_x \xi -u(t)\mu(x) \Psi(t,x) , \\
        \xi(t,0) = \xi(t,1)=0,\\
        \xi(0,x)=0,
    \end{array}
\right. 
\label{order2} 
\end{equation}
\item and the third-order term $\zeta$ of the expansion of $\psi$ is solution of, 
\begin{equation}  
\left\{ 
    \begin{array}{ll}
        i \partial_t \zeta = - \partial^2_x \zeta -u(t)\mu(x) \xi(t,x) , \\
        \zeta(t,0) = \zeta(t,1)=0,\\
        \zeta(0,x)=0.
    \end{array}
\right. 
\label{order3} 
\end{equation}
\end{itemize}
The goal of this section is to quantify in which way the following expansion holds rigorously
\begin{equation}
\label{expansion}
\psi \approx \psi_1 + \Psi + \xi + \zeta. 
\end{equation}
Such expansion will be studied under the following asymptotic.
\begin{defi}
\label{def:O}
Given two scalar quantities $A(T,u)$ and $B(T,u)$, we will write $A(T,u)=\O \left( B(T,u) \right)$ if there exists $C, T^*>0$ such that for any $T \in (0, T^*)$, there exists $\eta > 0$ such that for all $u \in H^2_0(0,T)$ with  $\|u\|_{H^2_0(0,T)} <  \eta$, we have $|A(T,u)| \ioe C |B(T,u)|.$ 
\end{defi}
Thus, the notation $\O$ refers to the convergence $\| u\|_{H^2_0(0,T)} \rightarrow 0$ and holds uniformly with respect to the final time on a small time interval $[0, T^*]$. All estimates will be computed under this asymptotic as the goal stated in \cref{the_theorem} is prove $H^2_0$-STLC. 

But first, before computing any estimate, we state a well-posedness result about all the equations considered, which  directly stems from \cref{rem:wp}. 
\begin{prop}
\label{wp_syst_init}
Let $\mu$ satisfying (H$_{\reg}$) and $u$ in $H^2_0( (0,T), \R)$. Then, there exists a unique solution $\psi$ (resp.\ $\Psi$, $\xi$ and $\zeta$) of \eqref{Schrodinger} (resp.\ \eqref{order1}, \eqref{order2} and \eqref{order3}) belonging to $C^2( [0,T], H^7_{(0)}(0,1))$ with $\psi(T)$ (resp.\ $\Psi(T)$, $\xi(T)$ and $\zeta(T)$) in $H^{11}_{(0)}(0,1)$.  Moreover, the following estimate holds,
\begin{equation}
\label{psi_borne}
\| \psi \|_{C^2( [0,T], H^7_{(0)})}
=
\O
\left(
1
\right).
\end{equation}
\end{prop}


\subsection{An auxiliary system}
\label{an_auxiliary_system}
Our goal is to prove that when $\mu$ satisfies \eqref{lin_nul}, (H$_{\Quad}$) and (H$_{\Cub}$), the leading term of the solution $\psi$ of the Schrödinger equation \eqref{Schrodinger} along the lost direction is namely the cubic term $\int_0^T u_1(t)^2u_2(t)dt$ in the asymptotic given in \cref{def:O}. Thus, we seek to prove that such a cubic term can absorb both the quadratic term and the terms of order higher than four. Therefore, classical error estimates on the expansion \eqref{expansion} involving the $L^2$-norm of the control $u$ are not sharp enough because can't be absorbed by such a cubic term. As in \cref{sec:expansion_dim_finie}, one can compute sharper estimates, involving rather the $L^2$-norm of the time primitive $u_1$ of the control $u$ by introducing the new state
\begin{equation}
\label{link}
\wt{\psi}(t,x):=\psi(t,x) e^{-i u_1(t) \mu(x)}, \quad (t,x) \in [0,T] \times [0,1].
\end{equation}
This new state satisfies the following equation, called the auxiliary system
\begin{equation} 
\label{system_aux}
\left\{ 
    \begin{array}{ll}
        i \partial_t \wt{\psi} = - \partial^2_x \wt{\psi} -iu_1(t) ( 2 \mu'(x) \partial_x \wt{\psi} + \mu''(x) \wt{\psi} )+u_1(t)^2 \mu'(x)^2 \wt{\psi}, \\
        \wt{\psi}(t,0) = \wt{\psi}(t,1)=0,\\
        \wt{\psi}(0,x)=\varphi_1. 
    \end{array}
\right. \end{equation}
\begin{prop}
\label{wp_syst_aux}
Let $\mu$ satisfying (H$_{\reg}$) and $u_1$ in $H^3_0( (0,T), \R)$. There exists a unique solution $\tild{\psi}$ of \eqref{system_aux} in $C^2( [0,T], H^7 \cap H^5_{(0)})$, which satisfies 
\begin{equation}
\label{estim_aux}
\| \tild{\psi} \|_{C^2( [0,T], H^7 \cap H^5_{(0)})}
=
\O
\left(
1
\right).
\end{equation}
Moreover, the following equality holds in $H^5_{(0)}(0,1)$ for every $t \in [0,T]$, 
\begin{equation}
\label{weak_sol_aux}
\tild{\psi}(t)
=
\psi_1(t)
-
\int_0^t
e^{-iA(t-\tau)}
\left(
u_1(\tau)
\left(
2 \mu' \partial_x 
+
\mu''
\right)
\tild{\psi}(\tau)
+
i
u_1(\tau)^2 \mu'^2 \tild{\psi}(\tau)
\right)
d\tau.
\end{equation}
\end{prop}

To prove \eqref{weak_sol_aux}, we need to recall the following smoothing effect first proved in \cite{BL10} and then generalized in \cite{B21}.
\begin{prop}
\label{estim_G_Ck}
Let $(p, k) \in \N^2$. There exists a non-decreasing function $C : [0, +\infty) \rightarrow (0, +\infty)$ such that for all $T \soe 0$ and for all $f \in H^{k}_0((0,T), H^{2p+3} \cap H^{2p+1}_{(0)}(0,1))$, the function $G: t \mapsto \int_0^t e^{-iA(t-\tau)} f(\tau) d\tau$ belongs to $C^k( [0,T], H^{2p+3}_{(0)}(0,1))$ with the following estimate,  
\begin{equation}
\| G \|_{C^k( [0,T], H^{2p+3}_{(0)})} \ioe C \| f \|_{H^k((0,T), H^{2p+3} \cap H^{2p+1}_{(0)})}.
\end{equation}
\end{prop}

\begin{rem}
Because of the term $\partial_x \tild{\psi}$ in \eqref{system_aux}, up to now, the well-posedness of the auxiliary system is only understood through its link \eqref{link} with the Schrödinger equation and is not proved directly using for example a fixed-point argument on the formulation \eqref{system_aux}. However, one needs to be very careful: the multiplication by the exponential factor in \eqref{link} preserves the regularity but not the boundary conditions of $\psi$. More precisely, the continuity of the operators given in \cref{rem:cont_mu} is the key to know which boundary conditions can be deduced for the auxiliary system from the Schrödinger equation and which can't. 
\end{rem}

\begin{proof}[Proof of \cref{wp_syst_aux}]
\medskip \noindent  \emph{Regularity.}  By \cref{wp_syst_init}, under these hypotheses, the solution $\psi$ of the Schrödinger equation \eqref{Schrodinger} is $C^2([0,T], H^7_{(0)})$. Thus, from the continuity given in \cref{rem:cont_mu} on \eqref{cont_exp}, the function $\tild{\psi}$ defined by \eqref{link} is $C^2([0,T], H^7 \cap H^5_{(0)})$. Moreover, \eqref{estim_aux} follows from \eqref{psi_borne}. 

\medskip \noindent \emph{Equation.} For this regularity, the Schrödinger equation \eqref{Schrodinger} is satisfied for every $t$ in $H^5_{(0)}$. 
Thus, 
computations prove that \eqref{system_aux} holds for every $t$ in $H^5 \cap H^3_{(0)}$ using \cref{rem:cont_mu} with \eqref{cont_der} and \eqref{cont_mu_der}. 

\medskip \noindent  \emph{Uniqueness.} For every function $\tild{\psi}$ satisfying the first equation of \eqref{system_aux} for every $t \in [0, T]$ in $H^5 \cap H^3_{(0)}$, an energy estimate proves that its $L^2$-norm is preserved. This implies the uniqueness for solutions in $C^2([0,T], H^7 \cap H^5_{(0)})$ (see \cite[Prop. 4.6 and 4.7]{B21bis} for more details on such an energy estimate). 

\medskip \noindent  \emph{Weak formulation.} Denote by $\hat{\psi}$ the right-hand side of \eqref{weak_sol_aux}. 
The continuity of \eqref{cont_der} and \eqref{cont_mu_der} stated in \cref{rem:cont_mu} imply that the functions integrated in the definition of $\hat{\psi}$ belong to $H^2_0( (0,T), H^5 \cap H^3_{(0)})$. Thus, the smoothing effect stated in \cref{estim_G_Ck} with $(p,k)=(1,2)$ entails that $\hat{\psi}$ belongs to $C^2([0,T], H^5_{(0)}).$ Moreover, computations proves that $\hat{\psi}$ satisfies \eqref{system_aux} for every $t \in [0,T]$ in $H^5 \cap H^3_{(0)}$. Thus, the uniqueness allows to prove \eqref{weak_sol_aux}. 
\end{proof}

\subsection{Computation of the expansion of the auxiliary system}
\label{expansion_aux}
One can compute by hand the expansion of the solution $\tild{\psi}$ of the auxiliary system \eqref{system_aux} around the ground state, up to order 3.

\smallskip \noindent \emph{First-order term.}
 The linearized system of \eqref{system_aux} around the trajectory $(\psi_1, u=0)$ is given by 
\begin{equation}  
\left\{ 
    \begin{array}{ll}
        i \partial_t \wt{\Psi} = - \partial^2_x \wt{\Psi} -iu_1(t) \left( 2 \mu' \partial_x \psi_1 + \mu'' \psi_1 \right) , \\
        \wt{\Psi}(t,0) = \wt{\Psi}(t,1)=0,\\
        \wt{\Psi}(0,x)=0.
    \end{array}
\right. 
\label{order1aux} 
\end{equation}
Linearizing \eqref{link}, $\wt{\Psi}$ is also given by, 
\begin{equation}
\label{link_order1}
\tild{\Psi}(t) 
=
\Psi(t)
-
i 
u_1(t) 
\mu 
\psi_1(t),
\quad 
t 
\in
[0,T],
\end{equation}
where $\Psi$ is the solution of \eqref{order1}. Recall that by \cref{wp_syst_init}, $\Psi$ belongs to $C^2([0,T], H^7_{(0)})$. Thus, using the continuity of \eqref{cont_mu} given in \cref{rem:cont_mu} and \eqref{link_order1} entail that $\tild{\Psi}$ belongs to $C^2([0,T], H^7 \cap H^5_{(0)})$. As before, using  \cref{estim_G_Ck} with $(p,k)=(1, 0)$, the following equality holds in $H^5_{(0)}(0,1)$ for every $t \in [0,T]$, 
\begin{equation}
\label{order1aux_sp}
\tild{\Psi}(t)
=
-
\int_0^t 
e^{-iA(t -\tau)}
u_1(\tau)
\left( 2 \mu' \partial_x \psi_1(\tau) + \mu'' \psi_1(\tau) \right)
d\tau,
\end{equation}
and moreover, the following estimate holds
\begin{equation}
\label{estim_lin_aux}
\| \tild{\Psi} \|_{C^0( [0,T], H^{5}_{(0)}(0,1))}
=
\O
\left(
\| u_1\|_{L^2(0,T)}
\right)
.
\end{equation}
Moreover, the solution of \eqref{order1aux} can be computed explicitly as 
\begin{equation}
\label{expr:order1_aux}
\tild{\Psi}(t)
=
\sum \limits_{j=1}^{+\infty}
\left(
\left(
\lambda_j
-
\lambda_1
\right)
\langle
\mu
\varphi_1
,
\varphi_j 
\rangle
\int_0^t 
u_1(\tau)
e^{
i
(\lambda_j-\lambda_1)
\tau
}
d\tau
\right)
\psi_j(t), 
\quad 
t \in [0,T]. 
\end{equation}
\noindent \emph{Second-order term.}
The second-order term of the expansion of \eqref{system_aux} around the ground state is the solution of 
\begin{equation}  
\left\{ 
    \begin{array}{ll}
        i \partial_t \wt{\xi} = - \partial^2_x \wt{\xi} -iu_1(t) ( 2 \mu' \partial_x \wt{\Psi} + \mu'' \wt{\Psi} ) + u_1(t)^2 \mu'^2 \psi_1 , \\
        \wt{\xi}(t,0) = \wt{\xi}(t,1)=0,\\
        \wt{\xi}(0,x)=0. 
    \end{array}
\right.
\label{order2aux} 
\end{equation}
Identifying the second order terms in \eqref{link}, $\tild{\xi}$ can also be given by, 
\begin{equation}
\label{link_order2}
\tild{\xi}(t) 
= 
\xi(t) 
-
i 
u_1(t) 
\mu 
\tild{\Psi}(t) 
+ 
\frac{u_1(t)^2}{2}
\mu^2 
\psi_1(t).
\end{equation}
Thus, from the regularity of $\xi$ given in \cref{wp_syst_init}, on $\tild{\Psi}$ and the continuity of the operators given in \cref{rem:cont_mu}, $\xi$ is in $C^2([0,T], H^7 \cap H^5_{(0)})$. Moreover, the following equality holds in $H^5_{(0)}(0,1)$ for every $t \in [0,T]$, 
\begin{equation}
\label{order2aux_sp}
\tild{\xi}(t)
=
-
\int_0^t 
e^{-iA(t -\tau)}
\left[
u_1(\tau)
\left( 2 \mu' \partial_x \tild{\Psi}(\tau) + \mu'' \tild{\Psi}(\tau) \right)
+
i 
u_1(\tau)^2 \mu'^2 \psi_1(\tau)
\right]
d\tau.
\end{equation}
As all the integrated terms belongs for $\tau$ fixed to $H^3_{(0)}$, by \cref{rem:cont_mu}, the triangular inequality together with the fact that for every $s \in \R$, $e^{i s A}$ is an isometry from $H^3_{(0)}$ to $H^3_{(0)}$, gives that 
\begin{equation}
\label{estim_quad_aux}
\| \tild{\xi} \|_{C^0( [0,T], H^{3}_{(0)}(0,1))}
=
\O
\left(
\| u_1 \|_{L^1} \| \tild{\Psi}\|_{C^0( [0,T], H^4_{(0)})} + \| u_1\|^2_{L^2}
\right)
=
\O\left(
\|u_1\|^2_{L^2}
\right),
\end{equation}
using \eqref{estim_lin_aux}. Besides, substituting the explicit form of $\tild{\Psi}$ given in \eqref{expr:order1_aux} into \eqref{order2aux_sp}, the solution can be explicitly computed as 
\begin{multline}
\label{expr:order2_aux}
\tild{\xi}(t) 
 =
-i 
\sum \limits_{j=1}^{+\infty} 
\left(
\langle \mu'^2 \varphi_1, \varphi_j \rangle
\int_0^t u_1(\tau)^2
e^{ i \left( \lambda_j - \lambda_1 \right) \tau}
d\tau 
\right)
\psi_j(t)
\\
+
\sum \limits_{j=1}^{+\infty} 
\left(
\int_0^t u_1(\tau) 
\int_0^{\tau} u_1(s) 
\wt{k}_{\Quad, j}(\tau,s)
d\tau ds
\right)
\psi_j(t),
\end{multline}
where, for all $j \in \N^*$, the quadratic kernel $\wt{k}_{quad, j}$ is given by 
\begin{equation}
\label{eq:kernel_quad_aux}
\wt{k}_{\Quad, j}( \tau, s)
:= 
\sum \limits_{n=1}^{+\infty} 
(\lambda_1- \lambda_n) 
(\lambda_n - \lambda_j)
\langle \mu \varphi_1, \varphi_n \rangle
\langle \mu \varphi_n, \varphi_j \rangle
e^{i \left( (\lambda_j-\lambda_n)\tau + (\lambda_n - \lambda_1)s \right)}
.
\end{equation}
Thanks to \cref{decay_coeff}, all the quadratic kernels $\wt{k}_{\Quad, j}$ defined in \eqref{eq:kernel_quad_aux} are bounded in $C^{4}( \R^2, \C)$. This regularity is the key to perform integrations by parts and reveal a coercivity quantified by the $H^{-3}$-norm of the control, as stated in the following result. 
\begin{lem}
\label{coord_quad_IPP}
If the control $u \in L^2(0,T)$ is such that $u_2(T)=u_3(T)=0$, then, for all $j \in \N^*$, 
\begin{equation}
\label{eq:coord_quad_IPP}
\langle \tild{\xi}(T), \psi_j(T) \rangle
=
-i \sum \limits_{p=1}^3 A^p_j \int_0^T u_p(t)^2 e^{ i (\lambda_j - \lambda_1)t } dt
+
\int_0^T u_3(t) \int_0^t u_3(\tau) \partial^2_1 \partial^2_{2}  \wt{k}_{\Quad, j}(t, \tau)  d\tau dt. 
\end{equation}
\end{lem}
\begin{proof}
Let $j \in \N^*$. Thanks to \eqref{LB_A1K}, the computations given in \eqref{expr:order2_aux} give directly
\begin{equation}
\label{quad_non_IPP}
\langle 
\tild{\xi}(T), 
\psi_j(T) 
\rangle
=
-i 
A^1_j
\int_0^T u_1(t)^2
e^{ i \left( \lambda_j - \lambda_1 \right) t}
dt 
+
\int_0^T u_1(t) 
\int_0^{t} u_1(\tau) 
\wt{k}_{\Quad, j}(t,\tau)
d\tau 
dt
.
\end{equation}
Besides, for all $m \in \N$ and $H$ in $C^2(\R^2, \C)$, if $u_{m+1}(T)=0$, integrations by parts lead to
\begin{multline*}
\int_0^T u_m(t) \int_0^t u_m(\tau) H(t, \tau) d\tau dt 
= 
\int_0^T 
u_{m+1}(t)^2 
\left( 
\frac{1}{2} 
\frac{d}{dt}( H(t, t)) 
-  
\partial_1 H(t, t) 
\right) 
dt 
\\
+ 
\int_0^T 
u_{m+1}(t) 
\int_0^t 
u_{m+1}(\tau)  
\partial_1
\partial_{2} 
H(t, \tau) 
d\tau dt
.
\end{multline*}
Therefore, \eqref{eq:coord_quad_IPP} is deduced from \eqref{quad_non_IPP} applying this equality successively for $m=1$ and $H= \wt{k}_{\Quad, j}$ and for $m=2$ and $H= \partial_1 \partial_{2} \wt{k}_{\Quad, j}$ and also noticing that 
\begin{equation*}
\forall p=2, 3,
\quad
\frac{1}{2} 
\frac{d}{dt}
\left( 
\partial_1^{p-2} 
\partial_{2}^{p-2}
\wt{k}_{\Quad, j} (t, t) 
\right) 
-  
\partial_1^{p-1} 
\partial_{2}^{p-2}
\wt{k}_{\Quad, j} (t, t) 
=
-i 
A^p_j
e^{
i
(\lambda_j-\lambda_1)
t
}
.
\end{equation*}
\end{proof}
\noindent In particular, under (H$_{\Quad}$), the quadratic term, along the lost direction, is given by 
\begin{equation}
\label{quad_u3}
\langle \tild{\xi}(T), \psi_K(T) \rangle
=
-i A^3_K \int_0^T u_3(t)^2 e^{ i (\lambda_K - \lambda_1)t } dt
+
\int_0^T u_3(t) \int_0^t u_3(\tau) \partial^2_1 \partial^2_{2}  \wt{k}_{\Quad, K}(t, \tau)  d\tau dt.
\end{equation}
Thus, namely, the leading quadratic term along the lost direction of the expansion is given by $\int_0^T u_3(t)^2 dt$.

\medskip \noindent \emph{Third-order term.}
The third-order term of the expansion of \eqref{system_aux} around the ground state is the solution of 
\begin{equation}  
\left\{ 
    \begin{array}{ll}
        i \partial_t \wt{\zeta} = - \partial^2_x \wt{\zeta} -iu_1(t) \left( 2 \mu' \partial_x \wt{\xi} + \mu'' \wt{\xi} \right) + u_1(t)^2 \mu'^2 \Psi , \\
        \wt{\zeta}(t,0) = \wt{\zeta}(t,1)=0,\\
        \wt{\zeta}(0,x)=0.  
    \end{array}
\right.
\label{order3aux} 
\end{equation}
As before, thanks to \eqref{link}, the cubic term can also be given by 
\begin{equation}
\label{link_order3}
\tild{\zeta}(t) 
= 
\zeta(t)
- 
i 
u_1(t) 
\mu 
\tild{\xi}(t) 
+
\frac{u_1(t)^2}{2} 
\mu^2 
\tild{\Psi}(t)
+
i 
\frac{u_1(t)^3}{6}
\mu^3 
\psi_1(t),
\quad 
t \in [0,T].
\end{equation}
Thus, the cubic term $\tild{\zeta}$ belongs to  $C^2([0,T], H^7 \cap H^5_{(0)})$. Moreover, the following equality holds in $H^5_{(0)}(0,1)$ for every $t \in [0,T]$, 
\begin{equation}
\label{order3aux_sp}
\tild{\xi}(t)
=
-
\int_0^t 
e^{-iA(t -\tau)}
\left[
u_1(\tau)
\left( 2 \mu' \partial_x \tild{\xi}(\tau) + \mu'' \tild{\xi}(\tau) \right)
+
i 
u_1(\tau)^2 \mu'^2 \tild{\Psi}(\tau)
\right]
d\tau,
\end{equation}
and the following estimate holds, using the triangular inequality
\begin{equation}
\label{estim_cub_aux}
\| 
\tild{\zeta} 
\|_{C^0( [0,T], H^{1}_{(0)}(0,1))}
=
\O
\left(
\| u_1\|_{L^2(0,T)}^3
\right)
.
\end{equation}
Using the explicit computations of $\tild{\Psi}$ and $\tild{\xi}$ given in \eqref{expr:order1_aux} and \eqref{expr:order2_aux}, one gets that the third-order term is given by 
\begin{multline}
\label{expr:order3aux}
\tilde{\zeta}(T)  
=
\sum \limits_{j=1}^{+\infty}
\Big(
 i\int_0^T u_1(t)^2 
 \int_0^t u_1( \tau)
 \widetilde{k}_{\Cub,j}^1( t, \tau) 
d\tau dt
 +
 i 
\int_0^T u_1(t)
\int_0^t u_1(\tau)^2 
\widetilde{k}_{\Cub,j}^2( t, \tau) 
d\tau dt
\\
- 
\int_0^T u_1(t) 
\int_0^t u_1(\tau) 
\int_0^{\tau} u_1(s)
\widetilde{k}_{\Cub,j}^3(t, \tau, s ) 
ds d\tau dt
\Big)
\psi_j(T),
\end{multline}
where the cubic kernels are given by
\begin{align}
\label{kernel_cubic_1}
 \widetilde{k}_{\Cub,j}^1( t, \tau) 
 &:=
  \sum \limits_{n=1}^{+\infty}
 (\lambda_1- \lambda_n) 
 \langle \mu \varphi_1, \varphi_n \rangle
  \langle \mu'^2 \varphi_n, \varphi_j \rangle
 e^{i \left[ (\lambda_j- \lambda_n)t + (\lambda_n- \lambda_1) \tau \right] },
 \\
 \label{kernel_cubic_2}
\widetilde{k}_{\Cub,j}^2(t, \tau) 
&:=
\sum \limits_{n=1}^{+\infty} 
(\lambda_n-\lambda_j) 
\langle \mu'^2 \varphi_1, \varphi_n \rangle 
\langle \mu \varphi_n, \varphi_j \rangle 
e^{i \left[  (\lambda_j- \lambda_n)t + (\lambda_n- \lambda_1) \tau \right]},
\end{align}
and
\begin{multline}
\label{kernel_cubic_3}
\widetilde{k}_{\Cub,j}^3(t, \tau, s ) 
:=
\sum \limits_{p=1}^{+\infty} 
\sum \limits_{n=1}^{+\infty} 
(\lambda_1- \lambda_n) 
(\lambda_n - \lambda_p)
(\lambda_p - \lambda_j)
\\
\times 
\langle \mu \varphi_1, \varphi_n \rangle 
\langle \mu \varphi_n, \varphi_p \rangle 
\langle \mu \varphi_p, \varphi_j \rangle 
e^{i \left[ (\lambda_j - \lambda_p)t + (\lambda_p - \lambda_n) \tau + (\lambda_n- \lambda_1)s     \right]}.
\end{multline}
Thanks to \cref{decay_coeff}, the kernels $ \widetilde{k}_{\Cub,j}^1$ and $ \widetilde{k}_{\Cub,j}^2$ are bounded in $C^2(\R^2, \C)$ and the kernel $\widetilde{k}_{\Cub,j}^3$ is bounded in $C^1(\R^2, \C)$.
Formally, in an asymptotic of small time, \eqref{expr:order3aux} entails that the cubic term behaves as $\int_0^T u_1^2(t) u_2(t)dt$ as $ \widetilde{k}_{\Cub,j}^i( t, \tau)  \approx \widetilde{k}_{\Cub,j}^i(0,0)$ for $i=1,2$ at first order and as the third term of \eqref{expr:order3aux} is a higher order cubic term.

\subsection{Sharp error estimates for the auxiliary system}
The goal of this section is to compute sharp error estimates on the expansion of the auxiliary system.

\begin{prop}
\label{prop:error_estim_aux}
If $\mu$ satisfies (H$_{\reg}$), then the following error estimates on the expansion of the auxiliary system hold, 
\begin{align}
\label{eq:estim_lin2_aux}
\| 
\tild{\psi}
-
\psi_1
-
\tild{\Psi}
\|_{L^{\infty}((0,T), H^2_{(0)}(0,1))}
=
\O
\left( 
\| u_1\|_{L^2(0,T)}^2
\right),
\\
\label{eq:estim_lin4_aux}
\| 
\tild{\psi}
-
\psi_1
-
\tild{\Psi}
-
\tild{\xi}
-
\tild{\zeta}
\|_{L^{\infty}((0,T),L^2(0,1))}
=
\O
\left( 
\| u_1\|_{L^2(0,T)}^4
\right).
\end{align}
\end{prop}

\begin{proof} \emph{Proof of the linear remainder.} We have seen in \cref{wp_syst_aux}, that the following equality holds in $H^5_{(0)}$ for all $t \in [0,T]$, 
\begin{equation*}
\big( \tild{\psi}-\psi_1 \big)(t) 
= 
-\int_0^t 
e^{-iA(t - \tau)} 
\Big( u_1(\tau) 
(2\mu'\partial_x + \mu'')
\tild{\psi}(\tau) 
+ 
i u_1(\tau)^2 \mu'^2 \tild{\psi}(\tau) \Big) 
d\tau,
\end{equation*}
where for $\tau$ fixed, every term under this integral belongs to $H^5 \cap H^3_{(0)}$. Thus, the triangular inequality and the isometry of $e^{i As}$ from $H^3_{(0)}$ to $H^3_{(0)}$ for every $s$ give, 
\begin{multline}
\label{eq:estim_lin1_aux}
\| 
\tild{\psi}-\psi_1 
\|_{L^{\infty}( (0,T), H^3_{(0)})} 
=
\O
\left( 
\|u_1\|_{L^1} \| \tild{\psi} \|_{L^{\infty}( (0,T), H^4_{(0)})} +  \| u_1\|^2_{L^2}  \|\tild{\psi} \|_{L^{\infty}( (0,T), H^3_{(0)})}
\right) 
\\
= \O
\left(  
\| u_1\|_{L^2}
\right),
\end{multline}
using estimate \eqref{weak_sol_aux} on $\tild{\psi}$. Notice that, by Cauchy-Schwarz inequality, we indeed have $\| u_1\|_{L^1}=\O( \| u_1\|_{L^2})$ as the definition of $\O$ given in \cref{def:O} also means we work in small time. 

\smallskip \noindent \emph{Proof of \eqref{eq:estim_lin2_aux}.}
Using \eqref{weak_sol_aux} and \eqref{order1aux_sp}, the following equality holds in $H^5_{(0)}$ for all $t \in [0,T]$, \begin{equation*}
(\widetilde{\psi} - \psi_1 - \widetilde{\Psi})(t) 
= 
- \int_0^t 
e^{-iA(t-\tau)} 
\left(
u_1(\tau) \left(2 \mu' \partial_x  + \mu'' \right)(\widetilde{\psi} -\psi_1)(\tau) 
+ i u_1(\tau)^2 {\mu'}^2 \widetilde{\psi}(\tau)
\right) 
d\tau.
\end{equation*}
Once again, for $\tau$ fixed, every term belongs to $H^5 \cap H^3_{(0)}$ thanks to \cref{rem:cont_mu}, so using once again the triangular inequality, 
\begin{equation*}
\| 
\tild{\psi}
-\psi_1
-\tild{\Psi}
\|_{L^{\infty} H^2_{(0)}} 
=
\O
\left( 
\| u_1 \|_{L^1} 
\| \tild{\psi} - \psi_1 \|_{L^{\infty} H^3_{(0)}} 
+
\| u_1 \|^2_{L^2} 
\| \tild{\psi} \|_{L^{\infty}H^2_{(0)}} 
\right) 
=
\O( \|u_1\|^2_{L^2}),
\end{equation*}
using the estimate \eqref{eq:estim_lin1_aux} on the linear remainder and estimate \eqref{weak_sol_aux} on  $\tild{\psi}$.

\smallskip \noindent \emph{Proof of cubic remainder.} As before, in $H^5_{(0)}$,
\begin{multline*}
(
\widetilde{\psi} 
- \psi_1 
- \widetilde{\Psi}
-\tild{\xi}
)(t) 
\\
= 
- 
\int_0^t 
e^{-iA(t-\tau)} 
\left(
u_1(\tau) 
\left(
2 \mu' \partial_x  + \mu'' 
\right)
(\widetilde{\psi} -\psi_1-\tild{\Psi})(\tau) 
+ 
i 
u_1(\tau)^2
{\mu'}^2 
(\widetilde{\psi}-\psi_1)(\tau)
\right) 
d\tau.
\end{multline*}
And thus, using the triangular inequality, \eqref{eq:estim_lin1_aux} and \eqref{eq:estim_lin2_aux}, one gets
\begin{align}
\notag
\| 
\tild{\psi}
-\psi_1
-\tild{\Psi}
-\tild{\xi}
\|_{L^{\infty} H^1_{(0)}} 
&=
\O
\left( 
\| u_1 \|_{L^1} 
\| 
\tild{\psi} 
- \psi_1 
-\tild{\Psi}
\|_{L^{\infty} H^2_{(0)}} 
+
\| u_1 \|^2_{L^2} 
\| \tild{\psi}- \psi_1 \|_{L^{\infty}H^1_{(0)}} 
\right) 
\\
\label{eq:estim_lin3_aux}
&=
\O( \|u_1\|^3_{L^2}).
\end{align}

\smallskip \noindent \emph{Proof of \eqref{eq:estim_lin4_aux}.} Finally, 
in $H^5_{(0)}$, 
\begin{multline*}
(\widetilde{\psi} - \psi_1 - \widetilde{\Psi}-\tild{\xi} - \tild{\zeta})(t) 
=
- \int_0^t 
e^{-iA(t-\tau)} 
\Big(
u_1(\tau) 
\left(2 \mu' \partial_x  + \mu'' \right)(\widetilde{\psi} -\psi_1-\tild{\Psi}-\tild{\xi})(\tau) 
\\
+ 
i u_1(\tau)^2 {\mu'}^2 (\widetilde{\psi}-\psi_1-\tild{\Psi})(\tau)
\Big) 
d\tau.
\end{multline*}
And thus, using \eqref{eq:estim_lin2_aux} and \eqref{eq:estim_lin3_aux}, one gets \eqref{eq:estim_lin4_aux}.
\end{proof}

In a nutshell, when (H$_{\reg}$), (H$_{\lin}$), (H$_{\Quad}$) and (H$_{\Cub}$) are satisfied, for the expansion of the auxiliary system along the lost direction,
\begin{itemize}
\item by \eqref{quad_u3}, the leading quadratic term is given by $\int_0^T u_3(t)^2dt$,
\item by \eqref{expr:order3aux}, the leading cubic term is given by $\int_0^T u_1(t)^2 u_2(t)dt$,
\item and by \eqref{eq:estim_lin4_aux} among every term of order higher than four, the leading term is given by $(\int_0^T u_1(t)^2 dt)^2$.
\end{itemize}

Therefore, in the asymptotic of controls small in $H^2$, along the lost direction, the cubic term prevails on the linear term (because it vanishes) but also on the quadratic term and on the terms of order higher or equal than four. 
This is why, in the next proposition, we state that along the lost direction, we only keep the dominant cubic term and all other terms are seen as (small) pollution, the bigger pollution being given by the quadratic term. 

Let us stress that, in another asymptotic on controls, for example, in the asymptotic of controls small in $H^3$, this does not hold any more: Gagliardo-Nirenberg inequalities prove that the quadratic term prevails on the cubic term (and on the higher-order terms), and thus one can deny $H^3$-STLC using such a quadratic term as done in \cite{B21bis}. 
\begin{cor}
\label{prop:exp_aux}
Let $\mu$ satisfying (H$_{\reg}$), (H$_{\lin}$), (H$_{\Quad}$). Let $u \in H^2_0(0,T)$ be a control such that $u_2(T)=u_3(T)=0$. Then, the solution $\tild{\psi}$ of the auxiliary system \eqref{system_aux} associated to the initial condition $\varphi_1$ satisfies 
\begin{multline}
\label{dev_aux_varphiK}
\langle
\tild{\psi}
(T)
,
\psi_K(T)
\rangle 
-
 i\int_0^T u_1(t)^2 
 \int_0^t u_1( \tau)
 \widetilde{k}_{\Cub,K}^1( t, \tau) 
d\tau dt
\\
-
 i 
\int_0^T u_1(t)
\int_0^t u_1(\tau)^2 
\widetilde{k}_{\Cub,K}^2( t, \tau) 
d\tau dt
=
\O
\left(
\| u_3 \|^2_{L^2(0,T)}
+
\| u_1\|^3_{L^1(0,T)}
\right)
,
\end{multline}
where we recall that $ \widetilde{k}_{\Cub,K}^1$ and $ \widetilde{k}_{\Cub,K}^2$ are respectively defined by \eqref{kernel_cubic_1} and \eqref{kernel_cubic_2}. 
\end{cor}

\begin{proof}
The computations \eqref{expr:order1_aux}, \eqref{quad_u3}, \eqref{expr:order3aux} and the error estimate \eqref{eq:estim_lin4_aux} give that the right-hand side of \eqref{dev_aux_varphiK} is estimated by 
\begin{equation*}
\O
\left(
\| u_3 \|^2_{L^2(0,T)}
+
\| u_1\|^3_{L^1(0,T)}
+
\| u_1\|^4_{L^2(0,T)}
\right).
\end{equation*}
However, for every control $u$ such that $u_2(T)=u_3(T)=0$, integrations by parts and then Cauchy-Schwarz inequality prove that 
\begin{equation*}
\| u_1\|_{L^2}^4 
= 
\left(
\int_0^T u'(t) u_3(t)dt 
\right)^2 
\ioe 
C 
\| u'\|^2_{L^2(0,T)} 
\| u_3\|^2_{L^2(0,T)}
= \O( \| u_3\|^2_{L^2(0,T)}),
\end{equation*}
as we recall we work in the asymptotic of controls small in $H^2_0$ (see \cref{def:O}).
\end{proof}

\subsection{Sharp error estimates for the Schrödinger equation}
From the expansion of the auxiliary system, one can deduce sharp error estimates on the expansion of the solution of the Schrödinger equation \eqref{Schrodinger}. 
\begin{prop}
\label{thm:error_order4}
Let $\mu$ satisfying (H$_{\reg}$). Then,
\begin{align}
\label{estim_r2}
\left\|
\psi
- \psi_1
- \Psi
\right\|_{L^{\infty}( (0,T), L^2(0,1))}
&=
\O 
\left(
\| u_1\|_{L^2(0,T)}^2
+
| u_1(T)|^2
\right),
\\
\label{estim_r4}
\left\|
\psi
- \psi_1
- \Psi
- \xi
-\zeta
\right\|_{L^{\infty}( (0,T), L^2(0,1))}
&=
\O 
\left(
\| u_1\|_{L^2(0,T)}^4
+
| u_1(T)|^4
\right).
\end{align}
\end{prop}
\begin{proof}
The proof of \eqref{estim_r2} and \eqref{estim_r4} are very similar. Thus, we only prove \eqref{estim_r4}. 
Using all the links \eqref{link}, \eqref{link_order1}, \eqref{link_order2} and \eqref{link_order3} between the expansions of the Schrödinger equation and of the auxiliary system, one gets
\begin{multline*}
\left(
\psi
-\psi_1
-\Psi
-\xi
-\zeta
\right)
(T)
=
e^{i u_1(T) \mu}
(
\tild{\psi}
-\psi_1
-\tild{\Psi}
-\tild{\xi}
-\tild{\zeta}
)
(T)
+
(
e^{i u_1(T) \mu}
- 1
)
\tild{\zeta}(T) 
\\
+ 
(
e^{i u_1(T) \mu}
- 1
-iu_1(T) \mu
)
\tild{\xi}(T) 
+
(
e^{i u_1(T) \mu}
- 1
-iu_1(T) \mu
- \frac{u_1(T)^2}{2} \mu^2
)
\tild{\Psi}(T) 
\\
+ 
(
e^{i u_1(T) \mu}
- 1
-iu_1(T) \mu
- \frac{u_1(T)^2}{2} \mu^2
-i \frac{u_1(T)^3}{6} \mu^3
)
\psi_1(T) .
\end{multline*}
The first term is estimated by $\|u_1\|^4_{L^2}$ thanks to the estimate \eqref{eq:estim_lin4_aux} on the auxiliary system. Doing an expansion of $e^{i u_1(T) \mu}$, the second term (resp.\ the third, fourth and fifth term) is estimated by $|u_1(T)| \| \tild{\zeta} (T) \|_{L^2} $ (resp.\ $|u_1(T)|^2 \| \tild{\xi} (T) \|_{L^2} $, $ |u_1(T)|^3 \| \tild{\Psi} (T) \|_{L^2}$ and $|u_1(T)|^2$). Then, estimates \eqref{estim_lin_aux}, \eqref{estim_quad_aux} and \eqref{estim_cub_aux} on $\tild{\Psi}$, $\tild{\xi}$ and $\tild{\zeta}$ together with Young inequalities lead to \eqref{estim_r4}. 
\end{proof}

To conclude on the error estimate of the expansion of the Schrödinger equation, one needs to estimate the boundary term $u_1(T)$. This can be done for specific motions of the solution. 
\begin{lem}
For every $u$ in $H^2_0(0,T)$ such that the solution of \eqref{Schrodinger} satisfies
\begin{equation}
\label{hyp:motion_spec}
\langle \psi(T; \ u, \ \varphi_1), \varphi_1 \rangle = \langle \psi_1(T), \varphi_1 \rangle,
\end{equation}
then, the following estimate holds
\begin{equation}
\label{estim_u1(T)}
|u_1(T) | =
\O
\left(
\| u_1 \|^2_{L^2(0,T)}
\right). 
\end{equation}
\end{lem}

\begin{proof}
Thanks to the explicit computation of $\Psi$ given in \eqref{order1explicit}, one gets 
\begin{equation*}
\langle \psi(T), \varphi_1 \rangle 
=
\langle \psi_1(T), \varphi_1 \rangle
+ 
i e^{-i\lambda_1 T} \langle \mu \varphi_1, \varphi_1 \rangle u_1(T)
+
\O
\left(
\| 
\left(\psi- \psi_1-\Psi\right)(T)
\|_{L^2(0,1)}
\right)
\end{equation*}
As $\langle \mu \varphi_1, \varphi_1 \rangle \neq 0$ by \eqref{H_lin_2}, assumption \eqref{hyp:motion_spec} together with the estimate \eqref{estim_r2} of the quadratic remainder lead to
\begin{equation*}
u_1(T) 
=
\O 
\left(
\| u_1\|_{L^2(0,T)}^2
+
| u_1(T)|^2
\right).
\end{equation*}
By definition of $\O$ (see \cref{def:O}), we work with controls arbitrary small in $H^2_0$, thus, such estimate entails \eqref{estim_u1(T)}. 
\end{proof}

\begin{cor}
Let $\mu$ satisfying (H$_{\reg}$). Then, for every control $u \in H^2_0(0,T)$ such that the solution of \eqref{Schrodinger} satisfies \eqref{hyp:motion_spec}, the following estimate holds, 
\begin{equation}
\label{remainder_order4}
\left\|
\psi
- \psi_1
- \Psi
- \xi
-\zeta
\right\|_{L^{\infty}( (0,T), L^2(0,1))}
=
\O 
\left(
\| u_1\|_{L^2(0,T)}^4
\right).
\end{equation}
\end{cor}

\subsection{The non overlapping principle}
\label{sec:non_add}
In this paper, it is quite useful to use non overlapping controls as already seen in \cref{sec:toy_models}. Moreover, if $v$ (resp.\ $w$) is a control defined on $(0,T_1)$ (resp.\ on $(0, T_2)$), it would be every convenient to have
\begin{equation*}
\psi( T_1+ T_2; \ v \# w, \ \varphi_1)
=
\psi(T_1; \ v, \ \varphi_1)
+
\psi(T_2; \ w, \ \varphi_1),
\end{equation*}
where recall that the concatenation of two controls is defined in \eqref{concatenation}. However, it is not the case. That is why, in the following section, we estimate precisely the evolution of the solution along the lost direction, to then use it in \cref{sec:STLC_result} for non overlapping controls. 

%

\subsubsection{For the quadratic term}

\begin{prop}
Let $0< T_1 < T_2$. If $\mu$ satisfies \eqref{quad_nul_1} and \eqref{quad_nul_2}, then for all control $u$ in $H^2_0(0, T_2)$ such that $u_1(T_1)=u_2(T_1)=u_3(T_1)=u_2(T_2)=u_3(T_2)=0$, the solution of \eqref{order2} satisfies
\begin{multline}
\label{non_lin_quad}
\left|
\langle 
\xi(T_2; \ u, \ \varphi_1), 
\psi_K(T_2)
\rangle
-
\langle
\xi(T_1; \ u, \ \varphi_1), 
\psi_K(T_1)
\rangle 
\right|
\\
=
\O
\left(
\| u_3\|^2_{L^2(0,T_2)}
+
|u_1(T_2)|
\| u_2\|_{L^1(0,T_2)}
+
| u_1(T_2)|^2
\right).
\end{multline}
\end{prop}

\begin{proof}
Using the link with the auxiliary system \eqref{link_order2} and the explicit form of $\tild{\Psi}$ given in \eqref{expr:order1_aux}, one has, for every $T \in [0, T_2]$, 
\begin{equation*}
\langle 
\xi(T), 
\psi_K(T)
\rangle
=
\langle 
\tild{\xi}(T), 
\psi_K(T)
\rangle
-u_1(T) 
\int_0^T 
u_1(t) 
k_{\Quad, T}(t)
dt
-
\frac{u_1(T)^2}{2}
\langle \mu^2 \varphi_1, \varphi_K \rangle
e^{i (\lambda_K - \lambda_1)T},
\end{equation*}
where the quadratic kernel $k_{\Quad, T}$ is given by,
\begin{equation}
\label{eq:kernel_quad}
k_{\Quad, T}(t)
=
\sum 
\limits_{n=1}^{+\infty}
(\lambda_n- \lambda_1)
\langle \mu \varphi_1, \varphi_n \rangle
\langle \mu \varphi_p, \varphi_n \rangle
e^{i [ (\lambda_n - \lambda_1)t + (\lambda_K- \lambda_j)T]}.
\end{equation}
Thus, if the control satisfies $u_1(T_1)=0$, then, 
\begin{multline*}
\langle 
\xi(T_2), 
\psi_K(T_2)
\rangle
-
\langle 
\xi(T_1), 
\psi_K(T_1)
\rangle
=
\langle 
\tild{\xi}(T_2), 
\psi_K(T_2)
\rangle
-
\langle 
\tild{\xi}(T_1), 
\psi_K(T_1)
\rangle
\\
+u_1(T_2)
\int_0^{T_2} 
u_1(t)
k_{\Quad, T_2}(t)
dt
- \frac{u_1(T_2)^2}{2}
\langle \mu^2 \varphi_1, \varphi_K \rangle
e^{i (\lambda_K - \lambda_1)T_2}
.
\end{multline*}
The first term of the right-hand side is estimated by $\O( \| u_3\|^2_{L^2(0,T_2)})$ using the explicit computation of $\tild{\xi}$ given in  \eqref{quad_u3}. The second term of the right-hand side is naturally estimated by 
$\O(|u_1(T_2)|
\| u_1\|_{L^1(0,T_2)})$
as the kernel is bounded. However, such estimate will not be sharp enough to use in the sequel of this paper (and more precisely in the proof of \cref{prop:TV1}. Thus, one can compute a shaper estimate by performing one integration by parts in the integral. This gives that the second term is estimated by 
$\O(|u_1(T_2)|
\| u_2\|_{L^1(0,T_2)})$ 
noticing that $k_{\Quad, T_2}'$ is still bounded thanks to \cref{decay_coeff}.
\end{proof}

\subsubsection{For the cubic term}
\begin{prop}
Let $0 < T_1 < T_2$.  For every control $u$ in $H^2_0(0, T_2)$ such that $u_2(T_1)=u_2(T_2)=0$, the solution of \eqref{order3aux} satisfies
\begin{multline}
\label{non_lin_cub_aux}
\left|
\langle 
\tild{\zeta}(T_2), 
\psi_K(T_2)
\rangle
-
\langle
\tild{\zeta}(T_1), 
\psi_K(T_1)
\rangle 
\right|
\\
=
\O
\left(
\| u_1\|^3_{L^1(0,T_2)}
+
\| u_1\|^2_{L^2(T_1, T_2)}
\| u_1\|_{L^1(0, T_2)}
+
\| u_2\|_{L^{\infty}(T_1, T_2)}
\| u_1\|^2_{L^2(0, T_2)}
\right).
\end{multline}
\end{prop}

\begin{proof}
Using the explicit form of $\tild{\zeta}$ given in \eqref{expr:order3aux}, one has, 
\begin{multline*}
\langle 
\tild{\zeta}(T_2), 
\psi_K(T_2)
\rangle
-
\langle
\tild{\zeta}(T_1), 
\psi_K(T_1)
\rangle 
=
 i\int_{T_1}^{T_2} u_1(t)^2 
 \int_0^t u_1( \tau)
 \widetilde{k}_{\Cub,K}^1( t, \tau) 
 d\tau dt
\\
+
i 
\int_{T_1}^{T_2} u_1(t)
\int_0^t u_1(\tau)^2 
\widetilde{k}_{\Cub,K}^2( t, \tau) 
d\tau dt
+
\O
\left(
\| u_1\|^3_{L^1(0, T_2)}
\right),
\end{multline*}
as the kernel $\tild{k}_{\Cub, K}^3$ defined in \eqref{kernel_cubic_3} is bounded in $C^0( \R^3, \C)$. The first term of the right-hand side is bounded by $\| u_1\|^2_{L^2(T_1, T_2)} \| u_1\|_{L^1(0, T_2)}$ as $\tild{k}_{\Cub, K}^1$ is also bounded. Moreover, it would seem natural to estimate the second term by $\| u_1\|_{L^1(T_1, T_2)}\| u_1\|_{L^2(0, T_2)}^2.$ However, as before, it would not provide an estimate sharp enough to use in the sequel of the work. One can compute a sharper estimate by performing one integration by parts to get that the second term of the right-hand side is bounded by,  
\begin{multline*}
\left|
\int_{T_1}^{T_2}
u_2(t) u_1(t)^2 
\widetilde{k}_{\Cub,K}^2( t, t)
dt
+
\int_{T_1}^{T_2} u_2(t)
\int_0^t u_1(\tau)^2 
\partial_1 \widetilde{k}_{\Cub,K}^2( t, \tau) 
dt d\tau
\right|
\\
=\O
\left(
\| u_2\|_{L^{\infty}(T_1, T_2)}
\| u_1\|^2_{L^2(0, T_2)}
\right), 
\end{multline*}
as the kernel  $\widetilde{k}_{\Cub,K}^2$ defined in \eqref{kernel_cubic_2} is bounded in $C^1( \R^2, \C)$. 
\end{proof}

\begin{prop}
Let $0 < T_1 < T_2$.  For every control $u$ in $H^2_0(0, T_2)$ such that $u_1(T_1)=u_2(T_1)=u_2(T_2)=0$, the solution of \eqref{order3} satisfies
\begin{multline}
\label{non_lin_cub}
\left|
\langle 
\zeta(T_2), 
\psi_K(T_2)
\rangle
-
\langle
\zeta(T_1), 
\psi_K(T_1)
\rangle 
\right|
=
\O
\Big(
\| u_1\|^3_{L^1(0,T_2)}
+
\| u_1\|^2_{L^2(T_1, T_2)}
\| u_1\|_{L^1(0, T_2)}
\\+
\| u_2\|_{L^{\infty}(T_1, T_2)}
\| u_1\|^2_{L^2(0, T_2)}
+
|u_1(T_2)| \| u_1\|_{L^2(0,T_2)}^2
+
|u_1(T_2)|^3
\Big).
\end{multline}
\end{prop}

\begin{proof}
Using the link with the auxiliary system \eqref{link_order3} and the explicit computations of $\tild{\Psi}$ and $\tild{\xi}$ given in \eqref{expr:order2_aux} and \eqref{expr:order3aux}, one gets, for all $T \in [0, T_2]$, 
\begin{multline*}
\langle 
\zeta(T), 
\psi_K(T)
\rangle
=
\langle 
\tild{\zeta}(T), 
\psi_K(T)
\rangle
+
iu_1(T)
\int_0^T 
u_1(t)^2 
k_{\Cub,T}^1(t) 
dt 
\\
+
i
u_1(T)
\int_0^T 
u_1(t)
\int_0^t 
u_1(\tau)
k_{\Cub,T}^2(t, \tau)
d\tau
dt 
-
\frac{u_1(T)^2}{2}
\int_0^T 
u_1(t) 
k_{\Cub, T}^3(t) 
dt 
\\-
i
\frac{u_1(T)^3}{6}
\langle \mu^3 \varphi_1, 
\varphi_K
\rangle
e^{i (\lambda_K- \lambda_1)T},
\end{multline*}
where the cubic kernels are given by, 
\begin{align*}
k_{\Cub, T}^1(t)
&
:=
-i 
\sum\limits_{n=1}^{+\infty}
\langle 
\mu'^2 
\varphi_1
, 
\varphi_n
\rangle
\langle 
\mu 
\varphi_n 
, 
\varphi_K
\rangle
\int_0^T 
u_1(t)^2
e^{
i
[
( \lambda_n - \lambda_1)t
+
(\lambda_K-\lambda_n)T
]
}
dt
, 
\\
k_{\Cub,T}^2(t, \tau)
&:=
\sum\limits_{p=1}^{+\infty}
\sum\limits_{n=1}^{+\infty}
(\lambda_1- \lambda_n)
(\lambda_n- \lambda_p)
\langle \mu \varphi_1, \varphi_n \rangle
\langle \mu \varphi_n, \varphi_p \rangle
\langle \mu \varphi_p, \varphi_K \rangle
\\
&\times e^{
i 
[
(\lambda_p- \lambda_n)t
+
(\lambda_n-\lambda_1) \tau
+
(\lambda_K-\lambda_p)T
]
}
, 
\\
k_{\Cub, T}^3(t)
&
:=
-
\sum\limits_{n=1}^{+\infty}
(\lambda_n-\lambda_1)
\langle \mu \varphi_1, \varphi_n \rangle
\langle \mu^2 \varphi_n, \varphi_K \rangle 
e^{i 
[
(\lambda_n- \lambda_1)t
+
(\lambda_K-\lambda_n)T
]
}
.
\end{align*}
So, if the control satisfies $u_1(T_1)=0$, one gets
\begin{multline*}
\langle 
\zeta(T_2), 
\psi_K(T_2)
\rangle
-
\langle 
\zeta(T_1), 
\psi_K(T_1)
\rangle
=
\langle 
\tild{\zeta}(T_2), 
\psi_K(T_2)
\rangle
-
\langle 
\tild{\zeta}(T_1), 
\psi_K(T_1)
\rangle
\\
+
iu_1(T_2)
\int_0^{T_2} 
u_1(t)^2 
k_{\Cub, T_2}^1(t) 
dt 
+
i
u_1(T_2)
\int_0^{T_2} 
u_1(t)
\int_0^t 
u_1(\tau)
k_{\Cub, T_2}^2(t, \tau)
d\tau
dt 
\\
-
\frac{u_1(T_2)^2}{2}
\int_0^{T_2} 
u_1(t) 
k_{\Cub, T_2}^3(t) 
dt 
-
i
\frac{u_1(T_2)^3}{6}
\langle \mu^3 \varphi_1, 
\varphi_K
\rangle
e^{
i (\lambda_K- \lambda_1)T_2
}.
\end{multline*}
From the estimate on the auxiliary system \eqref{non_lin_cub_aux} and the boundness of the kernels, one deduces \eqref{non_lin_cub}. 
\end{proof}

From the behaviors of the quadratic and cubic terms given in \eqref{non_lin_quad} and \eqref{non_lin_cub}, from the error estimate \eqref{remainder_order4} and from the estimate \eqref{estim_u1(T)} on the boundary term $u_1(T_1)$, one can deduce the following estimate. 
\begin{thm}
\label{non_linearity}
Let $0 < T_1 < T_2$, $\mu$ satisfying \eqref{lin_nul}, \eqref{quad_nul_1} and \eqref{quad_nul_2}. For every control $u$ in $H^2(0, T_2)$ such that $u_1(T_1)=u_2(T_1)=u_3(T_1)=u_2(T_1)=u_2(T_2)=u_3(T_2)=0$, if the solution of \eqref{Schrodinger} satisfies the specific motion \eqref{hyp:motion_spec}, then, one has
\begin{multline*}
\left|
\langle 
\psi(T_2), 
\psi_K(T_2)
\rangle
-
\langle
\psi(T_1), 
\psi_K(T_1)
\rangle 
\right|
=
\O
\Big(
\| u_3\|^2_{L^2(0, T_2)}
+
\| u_1\|^2_{L^2(0, T_2)}
\|u_2\|_{L^1(0, T_2)}
\\
+
\| u_1\|^3_{L^1(0,T_2)}
+
\| u_1\|^2_{L^2(T_1, T_2)}
\| u_1\|_{L^1(0, T_2)}
+
\| u_2\|_{L^{\infty}(T_1, T_2)}
\| u_1\|^2_{L^2(0, T_2)}
\Big).
\end{multline*}
\end{thm}

\section{Motions in the lost directions}
\label{sec:STLC_result}

\subsection{Motions in the lost directions  $\pm i \varphi_K$}
As for the bilinear toy-model \eqref{ODE}, the motions in the $\pm i \varphi_K$ directions are done in two steps.
\begin{itemize}
\item First, we initiate the motion along $\pm i \varphi_K$ by noticing that when $\mu$ satisfies (H$_{\lin}$), (H$_{\Quad}$) and (H$_{\Cub}$), along $i \varphi_K$, the solution is mainly driven by the cubic term $\int_0^T u_1(t)^2 u_2(t)dt$ which enables us to move in both $+$ and $- i \varphi_K$ directions. This work is done in \cref{partial_tv_lambda}. At this end of this first step, along the other directions $(\varphi_j)_{j \in \N^* - \{K\}}$, the error is possibly `big' (in a sense to precise).
\item Then, this error is corrected using the exact local controllability in projection result given in \cref{linear_STLC}. However, one needs to make sure that such `linear' motions don't induce a too large error along $i \varphi_K$ and preserve the work done in the first step. This is done \cref{prop:TV1}. 
\end{itemize}

\begin{prop}
\label{partial_tv_lambda}
For all $T>0$, there exists $C, \rho>0$ and a continuous map $b \mapsto u_b$ from $\R$ to $H^2_0(0,T)$ such that,
\begin{equation}
\label{eq:dev_K}
\forall b \in (-\rho, \rho), 
\quad
\left|
\langle \psi(T; \ u_b, \ \varphi_1), \psi_K(T) \rangle 
-
i 
b
\right|
\ioe 
C
|b|^{1+\frac{1}{41}}
.
\end{equation}
Moreover, for all $p \in [1, +\infty]$ and $k \in \Z$, $k \geq -3$, there exists $C>0$ such that, 
\begin{equation}
\label{eq:size_controls}
\forall b \in (-\rho, \rho), 
\quad
\| u_b \|_{W^{k, p}(0,T)} 
\ioe
C
|b|^{
\frac{1}{41}
(
7
-
4k
+
\frac{4}{p}
)
}
,
\end{equation}
and for all $\eps \in (0, \frac{3}{4})$, there exists $C>0$ such that for all $b \in (-\rho, \rho)$, 
\begin{equation}
\label{estim_valeur_finale}
\left\| 
\psi(T; \ u_b, \ \varphi_1) 
- 
\psi_1(T)
\right\|_{H^{2m+7}_{(0)}(0,1)}
\ioe 
C 
|b|^{ 
\frac{1}{41}
(10 -4m - 4\eps)
}
,
\quad 
\forall 
m=-3, \ldots, 2
. 
\end{equation}
\end{prop}

\begin{proof}
Let $T>0$ and $\rho \in (0, T^{\frac{41}{4}})$. For all $b \in \R^*$, we define the control $u_b$ by,
\begin{multline}
\label{eq:contr_oscill}
\forall t \in [0,T], 
\quad
u_b(t) 
:=
\sign(b)
|b|^{\frac{7}{41}} 
\phi^{(3)} 
\left( 
\frac{t}{
|b|^{
\frac{4}{41}
}
}
\right)
\\
\text{ where }
\quad
\phi \in C_c^{\infty}(0,1)
\quad
\text{ such that} 
\
\int_0^1 \phi''( \theta)^2 \phi'(\theta) d\theta=\frac{1}{C_K},
\end{multline}
where $C_K$ is defined in \eqref{cub_non_nul}. 
Notice that for all $b \in (-\rho, \rho)$, $u_b$ is supported on $(0, |b|^{\frac{4}{41}}) \subset (0,T)$. 

\smallskip \noindent \emph{Size of the control.} Let $p \in [1, +\infty]$ and $k \in \Z$, $k \geq -3$. Recall that, if $k$ is negative, we still write $u^{(k)}$ to denote $u_{|k|}$ the $|k|$-th primitive of $u$. For $k >0$, by Poincaré's inequality, there exists $C>0$ such that $\| u_b \|_{W^{k,p}} \ioe C \| u_b^{(k)} \|_{L^p}$. For $k<0$, it holds by definition \eqref{def_norm_faible} of the negative norms. Thus, by definition \eqref{eq:contr_oscill} and performing the change of variables 
$t = \llambda \theta$, 
one has, for all $b \in (-\rho, \rho)$, $b \neq 0$, 
\begin{align*}
\| u_b \|_{W^{k,p}(0,T)}^p
&\ioe 
%
C
\int_0^{\llambda}
\left|
|b|^{
\frac{1}{41}
(
7-4k
)
}
\phi^{(3+k)} \left( \frac{t}{\llambda} \right)
\right|^p 
dt 
\\
&\ioe
C
\|
\phi^{(3+k)}
\|_{L^p(0,1)}^p
|b|^{
\frac{1}{41}
(
p(
7-4k
)
+
4
)
}
.
\end{align*}

\smallskip \noindent \emph{Continuity of the map $b \mapsto u_b$.} The continuity from $\R^*$ to $H^2_0(0,T)$ directly stems from the dominated convergence theorem. Moreover, the size estimate \eqref{eq:size_controls} with $k=p=2$ gives the existence of $C>0$ such that, 
\begin{equation*}
\forall b \in (-\rho, \rho), \ b \neq 0, 
\quad
\| u_b \|_{H^2_0(0,T)}
\ioe 
C
|b|^{
\frac{1}{41}
}.
\end{equation*}
Thus, the map $b \mapsto u_b$ can be continuously extended at zero with $u_0=0$. 

\smallskip \noindent \emph{Expansion of the solution.}
As $u_1(T)=0$, by \eqref{link}, the end-point of $\psi$ and $\tild{\psi}$ are the same. Thus, it suffices to prove \eqref{eq:dev_K}  with $\tild{\psi}$ instead of $\psi$. Moreover, \cref{prop:exp_aux} gives the following expansion of order 3 of the auxiliary system, when $b$ goes to zero, 
\begin{multline}
\label{calcul_cub}
\Big|
\langle
\tild{\psi}
(T; \ u_b, \ \varphi_1)
,
\psi_K(T)
\rangle 
-
i\int_0^T u_1(t)^2 
\int_0^t u_1( \tau)
\widetilde{k}_{\Cub,K}^1( t, \tau) 
d\tau dt 
\\
-
i 
\int_0^T u_1(t)
\int_0^t u_1(\tau)^2 
\widetilde{k}_{\Cub,K}^2( t, \tau) 
d\tau dt
\Big|
\ioe
C
|b|^{\frac{42}{41}}
,
\end{multline}
as by \eqref{eq:size_controls}, one has the following estimates:
$
\| u_3\|^2_{L^2} 
\ioe C |b|^{\frac{42}{41}}
$
and 
$
\| u_1\|_{L^1}^3
\ioe
C |b|^{\frac{45}{41}}.
$
Then, substituting the explicit form of the control \eqref{eq:contr_oscill} and performing the change of variables $(t, \tau)=(\llambda \theta_1, \llambda \theta_2)$, the  two integral terms of \eqref{calcul_cub} are given by 
\begin{multline*}
i
\sign(b)
|b|
\Big(
\int_0^1
\phi''(\theta_1)^2
\int_0^{\theta_1}
\phi''(\theta_2)
\tild{k}_{\Cub,K}^1(\llambda \theta_1, \llambda \theta_2)
d\theta_2 \theta_1
\\
+
\int_0^1
\phi''(\theta_1)
\int_0^{\theta_1}
\phi''(\theta_2)^2
\tild{k}_{\Cub,K}^2(\llambda \theta_1, \llambda \theta_2)
d\theta_2 \theta_1
\Big).
\end{multline*}
Expanding the kernels when $b$ goes to zero, as they are both bounded in $C^1(\R^2, \C)$, one gets that \eqref{calcul_cub} can be written as 
\begin{equation*}
\left|
\langle \widetilde{\psi}(T; \ u_b, \ \varphi_1), \psi_K(T) \rangle 
-
i 
b
(
 \widetilde{k}_{\Cub,K}^1(0,0)
 -
\widetilde{k}_{\Cub,K}^2(0,0)
)
\int_0^1 \phi''(\theta)^2 \phi'(\theta) d\theta 
\right|
\ioe 
C
|b|^{\frac{42}{41}}
.
\end{equation*}
Moreover, looking at the definition of $C_K$,  $\tild{k}_{\Cub,K}^1$ and $\tild{k}_{\Cub,K}^2$ given respectively in \eqref{cub_non_nul}, \eqref{kernel_cubic_1} and \eqref{kernel_cubic_2}, one can notice that 
$
 \widetilde{k}_{\Cub,K}^1(0,0)
 -
\widetilde{k}_{\Cub,K}^2(0,0)
=C_K$.
Thus, the previous equality leads to \eqref{eq:dev_K} by choice of $\phi$ given in \eqref{eq:contr_oscill}. 

\smallskip \noindent \emph{Size of the end-point.} Using the explicit form of $\Psi$ given in \eqref{order1explicit}, one gets,
\begin{multline}
\label{estim_end_point}
\left\| 
\psi(T
) 
- \psi_1(T)
\right\|_{H^{2m+7}_{(0)}}
\ioe
\left\|
\left(
\langle \mu \varphi_1, \varphi_j \rangle
\int_0^T 
u_b(t) 
e^{
i 
(\lambda_j- \lambda_1)
(t-T)
} 
dt
\right)
\right\|_{h^{2m+7}(\N^*)}
\\
+
\| 
(
\psi 
- \psi_1
- 
\Psi
)(T) 
\|_{H^{2m+7}_{(0)}}, 
\quad 
m \in \{-3, \ldots, 2 \}
\end{multline}
Yet, for all $j \in \N^*$, $j \soe 2$ and $k \in \Z$ with $k \soe -3$, by integrations by parts (integrating $u$ when $k< 0$ or differentiating $u$ when $k \soe 0$), one gets, 
\begin{equation}
\label{estim_coeff_lin_ponc}
\left|
\int_0^T 
u_b(t) 
e^{
i 
(\lambda_j- \lambda_1)
t
}  
dt
\right|
= 
\left|
(\lambda_j- \lambda_1)^{-k} 
\int_0^T 
u^{(k)}_b(t) 
e^{
i 
(\lambda_j- \lambda_1)
t
} 
dt 
\right|
%
\ioe
C
| \lambda_j - \lambda_1 |^{-k}  
|b|^{\frac{1}{41}(11-4k)}
,
\end{equation}
using estimates \eqref{eq:size_controls} with $p=1$. This also holds for $j=1$ as $u_1(T)=0$. By interpolation, such estimates hold for all $j \in \N^*$ and $k \in [-3, 3]$ with a uniform constant with respect to $k$.  Let $\eps \in (0, 3/4)$ and $m \in \{-3, \ldots, 2\}$. Taking $k= \eps +m +\frac{1}{4} \in [-3, 3]$ in \eqref{estim_coeff_lin_ponc}, summing over $j \in \N^*$ and using \cref{decay_coeff} to estimate the coefficients $(\langle \mu \varphi_1, \varphi_j \rangle)_{j \in \N^*}$, one gets
\begin{equation}
\label{estim_end_point_2}
\left\|
\left(
\langle \mu \varphi_1, \varphi_j \rangle
\int_0^T 
u_b(t) 
e^{
i 
(\lambda_j- \lambda_1)
(t-T)
} 
dt
\right)
\right\|_{h^{2m+7}(\N^*)}
%
\ioe 
C
\left(
\sum \limits_{j=1}^{+\infty} \frac{1}{j^{1+4\eps}} 
\right)^{1/2}
|b|^{
\frac{1}{41}(10-4m-4\eps)
}
.
\end{equation}
Moreover, \cite[Proposition 4.5]{B21} gives the existence of $C>0$ such that for all $m=-3, \ldots, 2$, 
\begin{equation}
\label{estim_end_point_3}
\| 
(
\psi 
- \psi_1
- 
\Psi
)(T) 
\|_{H^{2m+7}_{(0)}}
\ioe 
C
\| u_b\|_{H^m(0,T)} \| u_b \|_{H^2(0,T)}
\ioe 
C 
|b|^{
\frac{1}{41}(10-4m)
}
,
\end{equation}
using \eqref{eq:size_controls} to estimate the size of the controls. Then, \eqref{estim_end_point}, \eqref{estim_end_point_2} and \eqref{estim_end_point_3} lead to \eqref{estim_valeur_finale}.
\end{proof}

\begin{prop}
\label{prop:TV1}
The vector $i \varphi_K$ is a small-time $H^2_0$-continuously approximately reachable vector associated with vector variations $\Xi(T)= i \psi_K(T)$. More precisely, for all $T>0$, there exists $C, \rho>0$ and a continuous map $b \mapsto w_b$ from $\R$ to $H^2_0(0,T)$ such that,
\begin{equation}
\label{tv_1}
\forall b \in (-\rho, \rho), 
\quad
\left\|
\psi(T ; \ w_b, \ \varphi_1) - \psi_1(T) - i b \psi_K(T) 
\right\|_{H^{11}_{(0)}(0,1)}
\ioe 
C
|b|^{
1
+
\frac{1}{82}
}, 
\end{equation}
with the following size estimate on the family of controls,
\begin{equation}
\label{tv_size_control}
\| w_b \|_{ H^2_0(0,T)} 
\ioe 
C
|b|^{\frac{1}{41}}.
\end{equation}
\end{prop}

\begin{proof} 
\emph{Definition of the control.} 
Let $0< T_1<T$. To move along the $\pm i\varphi_K$ directions, we use non overlapping controls. More precisely, we define, for all $b \in \R$,
\begin{equation}
\label{def_w_lambda}
w_b
:=
u_b
\1_{[0, T_1]} 
+
v_b
\1_{[T_1, T]},
\end{equation}
 where $(u_b)_{b \in \R}$ is the family of controls defined on $[0,T_1]$ constructed in \cref{partial_tv_lambda}
and 
$$v_b:=\Gamma_{T_1,T} \left( \psi(T_1; \ u_b, \ \varphi_1), \psi_1(T) \right),$$ where $\Gamma_{T_1,T}$ is the control operator constructed in \cref{linear_STLC} with $J=\N^* - \{K\}$ and $(p,k)=(2,2)$. 

\smallskip \noindent \emph{Size of the controls.} Because we use non overlapping controls, for all $b\in \R$, 
\begin{equation}
\label{size_wb_partial}
\| w_b\|_{H^2_0(0,T)}
=
\| u_b\|_{H^2_0(0,T_1)}
+
\| v_b\|_{H^2_0(T_1,T)}
\ioe 
C
|b|^{\frac{1}{41}}
+
\| v_b\|_{H^2_0(T_1,T)}
,
\end{equation}
using the size estimate \eqref{eq:size_controls} on the family $(u_b)_{b \in \R}$ with $p=k=2$. 
On $[T_1, T]$, using the linear estimates \eqref{estim_contr_nl} on $\Gamma_{T_1,T}$ and the estimates \eqref{estim_valeur_finale} on the end-point of the solution at time $T_1$, for all $\eps \in (0, \frac{3}{4})$, one gets the existence of $C>0$, such that for all $b$ small enough and for all $m \in \{-3, \ldots, 2\}$,
\begin{equation}
\label{size_v_lambda}
\| 
v_b
\|_{H^m_0(T_1, T)}
\ioe 
C
\left\| 
\psi(T_1; \ u_b, \ \varphi_1) - \psi_1(T_1)
\right\|_{H^{2m+7}_{(0)}(0,1)}
\ioe 
C 
|b|^{
\frac{1}{41}(10-4m-4\eps)
}
. 
\end{equation}
Taking $\eps < \frac{1}{4}$, estimate \eqref{size_v_lambda} with $m=2$  and \eqref{size_wb_partial} imply \eqref{tv_size_control}.

\smallskip \noindent \emph{Motion along $i \varphi_K$.} By construction of $\Gamma_{T_1,T}$ (see \eqref{contr_proj}), we already know that 
\begin{equation*}
\P \psi(T; \ w_b, \ \varphi_1)
=
\psi_1(T)
=
\P
\left(
\psi_1(T)
+
i b \psi_K(T)
\right),
\end{equation*}
where $\P$ denotes the orthogonal projection on $\overline{\Span_{\C}} \left( \varphi_j , \ j \in \N^*-\{K\} \right)$.
Thus, to prove \eqref{tv_1}, it only remains to prove that 
\begin{equation}
\label{goal}
\left|
\langle \psi(T; \ w_b, \ \varphi_1), \psi_K(T) \rangle 
-
i
b
\right|
\ioe 
C
|b|^{
1
+
\frac{1}{82}
}
.
\end{equation}
By definition of $(u_b)_{b \in \R}$, one already has that  
\begin{equation*}
\left|
\langle \psi(T_1; \ u_b, \ \varphi_1), \psi_K(T_1) \rangle 
-
i
b
\right|
\ioe 
C
|b|^{
1
+
\frac{1}{41}
}
.
\end{equation*}
Thus, it remains to prove that the linear correction used on the time interval $[T_1, T]$ didn't destroy such an estimate, and more precisely, to prove that 
\begin{equation*}
\left|
\langle \psi(T; \ w_b, \ \varphi_1), \psi_K(T) \rangle 
-
\langle \psi(T_1; \ w_b, \ \varphi_1), \psi_K(T_1) \rangle 
\right|
\ioe 
C
|b|^{
1
+
\frac{1}{82}
}
.
\end{equation*}
By \cref{non_linearity}, the left-hand side is estimated by
\begin{multline}
\label{estim_non_additivity}
\O
\Big(
\| w_3\|^2_{L^2(0, T)}
+
\| w_1\|^2_{L^2(0, T)}
\|w_2\|_{L^1(0, T)}
\\
+
\| w_1\|^3_{L^1(0,T)}
+
\| v_1\|^2_{L^2(T_1, T)}
\| w_1\|_{L^1(0, T)}
+
\| v_2\|_{L^{\infty}(T_1, T)}
\| w_1\|^2_{L^2(0, T)}
\Big).
\end{multline}
Thus, it remains to prove that the estimates \eqref{eq:size_controls} on $(u_b)_{b \in \R}$ and \eqref{size_v_lambda} on $(v_b)_{b \in \R}$ are sharp enough so that the previous quantity can be neglected in front of $|b|^{
1
+
\frac{1}{82}
}.
$
For example, using these estimates in $H^{-3}$, one has 
\begin{equation*}
\| w_3\|^2_{L^2(0, T)}
=
\| u_3\|^2_{L^2(0, T_1)}
+
\| v_3\|^2_{L^2(T_1, T)}
\ioe 
C
\left(
|b|^{
\frac{42}{41}
}
+
|b|^{
\frac{44}{41}
-
\frac{8 \eps}{41}
}
\right)
\ioe 
C
|b|^{
\frac{42}{41}
}
,
\end{equation*}
for $b$ small enough, choosing $\eps< \frac{1}{4}$. Similarly, one gets, 
\begin{align*}
\| w_1\|^2_{L^2(0, T)}
\|w_2\|_{L^{2}(0, T)}
&\ioe
C
\left(
|b|^{
\frac{26}{41}
}
+
|b|^{
\frac{28}{41}
-
\frac{8\eps}{41}
}
\right)
\left(
|b|^{
\frac{17}{41}
}
+
|b|^{
\frac{18}{41}
-
\frac{4\eps}{41}
}
\right)
,
\\
\| w_1\|^3_{L^1(0,T)}
&\ioe
C
\left(
|b|^{
\frac{45}{41}
}
+
|b|^{
\frac{42}{41}
-
\frac{12 \eps}{41}
}
\right)
,
\\
\| v_1\|^2_{L^2(T_1, T)}
\| w_1\|_{L^1(0, T)}
&\ioe
C
|b|^{
\frac{28}{41}
-
\frac{8\eps}{41}
}
\left(
|b|^{
\frac{15}{41}
}
+
|b|^{
\frac{14}{41}
-
\frac{4\eps}{41}
}
\right)
.
%
%
\end{align*}
Choosing $\eps< \frac{1}{24}$, for $b$ small enough, every term can be neglected in front of 
$
|b|^{
1+
\frac{1}{82}
}
$. 
In \eqref{estim_non_additivity}, it only remains to estimate 
$\| v_2\|_{L^{\infty}(T_1, T)}
\| w_1\|^2_{L^2(0, T)}$.
As \eqref{size_v_lambda} provides only estimates of $(v_b)_{b \in \R}$ in $L^2$-spaces, one needs to use a Gagliardo-Nirenberg inequality (see \cite[Theorem p.125]{N59}) to estimate $\| v_2\|_{L^{\infty}(T_1, T)}$. More precisely, there exists $C>0$ such that 
\begin{equation*}
\| v_2 \|_{L^{\infty}(T_1,T)} 
\ioe 
C 
\| v_1 \|_{L^2(T_1,T)}^{1/2} 
\| v_2 \|_{L^2(T_1,T)}^{1/2} 
+ C \| v_2 \|_{L^2(T_1,T)}
.
\end{equation*}
Thus, thanks to \eqref{size_v_lambda},  $v_b$ in $W^{-2, \infty}$ is estimated by 
\begin{equation*}
\| v_2 \|_{L^{\infty}(T_1,T)} 
\ioe
C
|b|^{
\frac{16}{41}
-\frac{4\eps}{41}
}.
\end{equation*}
And, finally, one has
\begin{equation*}
\| v_2\|_{L^{\infty}(T_1, T)}
\| w_1\|^2_{L^2(0, T)}
\ioe 
C
|b|^{
\frac{16}{41}
-\frac{4\eps}{41}
}
\left(
|b|^{
\frac{26}{41}
}
+
|b|^{
\frac{21}{41}
-\frac{4\eps}{41}
}
\right)
\end{equation*}
which is also neglected in front of 
$
|b|^{
1+
\frac{1}{82}
}
$ 
.
This concludes the proof of \eqref{tv_1}.

\smallskip \noindent \emph{Continuity of $b \mapsto w_{b}$.} 
The map $b \mapsto u_b$  of \eqref{partial_tv_lambda} is continuous from $\R$ to $H^2(0,T_1)$. 
Besides, the continuity of $b \mapsto v_b$ from $\R$ to $H^2_0(T_1,T)$ stems from the regularity of $\Gamma_{T_1,T}$ (see \cref{linear_STLC}) and of the solution of the Schrödinger equation with respect to the control 
(see \eqref{estim_sol_bis}).  This gives the continuity of the map $b \mapsto w_b$ constructed by \eqref{def_w_lambda}.
\end{proof}

\begin{rem}
The sharp estimates \eqref{size_v_lambda} on the control operator $\Gamma_{T_1, T}$ of \cite{B21bis} together with the sharp estimate \eqref{estim_non_additivity} on the evolution of the solution along the lost direction are the key to prove the motions along the first lost direction $i \varphi_K$. 
\end{rem}

\subsection{Motions in the lost directions $\pm \varphi_K$}

As for the toy-models \eqref{toy_model_3} and \eqref{ODE}, the second approximately reachable vector can be deduced from the first one using a proof inspired by \cite[Th.\ 6]{HK87}. The following statement and its proof are very similar to the one done in finite dimension in \cref{dim_finie_TV2}. One only needs to be careful about the functional setting. 
\begin{prop}
\label{prop:TV2}
The vector $\varphi_K$ is a small-time $H^2_0$-continuously approximately reachable vector associated with vector variations $\Xi(T)= \psi_K(T)$. More precisely, there exists $T^*>0$ such that for all $T \in (0, T^*)$, there exists $C, \rho>0$ and a continuous map $b \mapsto v_b$ from $\R$ to $H^2_0(0,T)$ such that,
\begin{equation}
\label{tv_2}
\forall b \in (- \rho, \rho), 
\quad
\left\|
\psi(T ; \ v_b, \ \varphi_1) - \psi_1(T) -  b \psi_K(T) 
\right\|_{H^{11}_{(0)}(0,1)}
\ioe 
C
|b|^{
1
+
\frac{1}{82}
}, 
\end{equation}
with the following size estimate on the family of controls,
\begin{equation}
\label{tv_size_control_2}
\| v_b \|_{ H^2_0(0,T)} 
\ioe 
C
|b|^{\frac{1}{41}}.
\end{equation}
\end{prop}

\begin{proof}
Denote by $(u_b)_{b \in \R}$ the control variations associated with $i \varphi_K$ constructed in \cref{prop:TV1}. The goal is to prove that the existence of $C>0$ such that for all $(\alpha, \beta) \in \R^2$ small enough, 
\begin{multline}
\label{goal_TV2_infinite}
\left\|
\psi(3T; \ u_{\alpha} \# 0_{[0, T]} \# u_{\beta}, \ \varphi_1) 
-
\psi_1(3T)
-
(
i 
\beta
e^{
2i( \lambda_K-\lambda_1)T
}
+
i
\alpha
)
\psi_K(3T)
\right\|_{H^{11}_{(0)}}
\\
\ioe 
C
| (\alpha, \beta) |^{1+\frac{1}{82}}.
\end{multline}
Thus, for all $T \in \left(0, \frac{\pi}{2(\lambda_K-\lambda_1)} \right)$ and $b \in \R$, taking $\beta=-\frac{b}{\sin(2(\lambda_K-\lambda_1)T)}$ and $\alpha=-\beta \cos(2(\lambda_K-\lambda_1)T)$, this proves the existence of a family $(v_{b})_{b \in \R}$ satisfying \eqref{tv_2} and \eqref{tv_size_control_2}. 
So, it remains to prove \eqref{goal_TV2_infinite}. 
First, by definition of $(u_b)_{b \in \R}$ in \cref{prop:TV1}, there exists $C>0$ and $\rho>0$ such that for all $\alpha \in (-\rho, \rho)$, 
\begin{equation}
\label{sur_0_T_infinite}
\left\|
\psi(T; \ u_{\alpha}, \ \varphi_1)
-
\psi_1(T)
-
i \alpha \psi_K(T)
\right\|_{H^{11}_{(0)}}
\ioe
C
|\alpha|^{1+\frac{1}{82}}
\quad
\text{ with }
\| u_{\alpha}\|_{H^2_0(0,T)}
\ioe 
C
|\alpha|^{\frac{1}{41}}.
\end{equation}
Then, on $[T, 2T]$, no control is activated, so $\psi(2T)= e^{-iAT} \psi(T)$ and \eqref{sur_0_T_infinite} becomes
\begin{equation}
\label{sur_T_2T_inf}
\left\|
\psi(2T; \ u_{\alpha} \# 0_{[0, T]}, \ \varphi_1)
-
\psi_1(2T)
-
i\alpha \psi_K(2T)
\right\|_{H^{11}_{(0)}}
\ioe
C
|\alpha|^{1+\frac{1}{82}}.
\end{equation}
Then, using the semi-group property of the Schrödinger equation, one has, 
\begin{equation*}
\psi(3T; \ u_{\alpha} \# 0_{[0, T]} \# u_{\beta}, \ \varphi_1)
=
\psi(T; \ u_{\beta}, \ \psi(2T; \ u_{\alpha} \# 0_{[0, T]}, \ \varphi_1)).
\end{equation*}
Together with the estimate \cref{prop:dep_ci} about the dependency of the solutions of the Schrödinger equation with respect to initial condition, one has
\begin{multline*}
\left\|
\psi(3T; \ u_{\alpha} \# 0 \# u_{\beta}, \ \varphi_1)
-
\psi(T; \ u_{\beta}, \ \varphi_1)e^{-i\lambda_1 2T}
-
e^{-iA T}
\left(
\psi(2T; \ u_{\alpha} \# 0, \ \varphi_1)
-
\psi_1(2T)
\right)
\right\|_{H^{11}_{(0)}}
\\
\ioe
C
\| u_{\beta} \|_{H^2_0(0,T)}
\left\|
\psi(2T; \ u_{\alpha} \# 0, \ \varphi_1)
-
\psi_1(2T)
\right\|_{H^{11}_{(0)}}.
\end{multline*}
Using estimate \eqref{tv_size_control} on $(u_{\beta})$ and  \eqref{sur_T_2T_inf} on 
$
\psi(2T; \ u_{\alpha} \# 0, \ \varphi_1)
-
\psi_1(2T)
$,
 the right-hand side of the previous inequality is estimated by 
$
C
| \beta |^{\frac{1}{41}}
| \alpha|
.
$
Then, using once again the estimate \eqref{sur_T_2T_inf} on 
$
\psi(2T; \ u_{\alpha} \# 0, \ \varphi_1)
-
\psi_1(2T)
$
and the definition of $(u_{\beta})_{\beta}$ given by \eqref{tv_1}, one has, 
\begin{multline*}
\left\|
\psi(3T; \ u_{\alpha} \# 0_{[0, T]} \# u_{\beta}, \ \varphi_1)
-
\psi_1(3T)
-
i \beta \psi_K(T) e^{-i \lambda_1 2T}
-
i \alpha \psi_K(3T)
\right\|_{H^{11}_{(0)}}
\\
\ioe
C
| \beta |^{\frac{1}{41}}
| \alpha|
+ 
C
| \alpha|^{1+ \frac{1}{82}},
\end{multline*}
which gives \eqref{goal_TV2_infinite} and this concludes the proof. 

\end{proof}

\subsection{Proof of \cref{the_theorem}: The $H^2_0-$STLC of the Schrödinger equation}
The goal of this section is to prove \cref{the_theorem} using the systematic approach developed in \cref{sec:black_box}. More precisely, we apply \cref{black_box} with $E_T:=H^2_0(0,T)$ for all $T> 0$, $X:=H^{11}_{(0)}(0,1)$ and 
\begin{equation*}
\F_T : (\psi_0, u) \mapsto \psi(T; \ u, \ \psi_0),
\end{equation*}
where $\psi$ is the solution of the Schrödinger equation \eqref{Schrodinger} with $\psi(0)=\psi_0$. Let us now check that the assumptions of \cref{black_box} hold in this setting (with the adaptation discussed in \cref{adaptation}). 

\begin{itemize}
\item [$(A_1)$] By \cite[Prop.\ 4.2]{B21}, it is known that when $\mu$ satisfies ($H_{\reg}$), the end-point map is $C^1$ around $(\varphi_1, 0)$. The $C^2$-regularity is proved similarly, and thus, the proof is left to the reader.

\item [$(A_2)$]  By \cite[Prop.\ 4.2]{B21}, the differential at $(\varphi_1, 0)$ is given by 
$
d \F_T (\varphi_1, 0).(\Psi_0,v) = \Psi(T),
$
where $\Psi$ is the solution of the linearized system
\begin{equation*}  
\left\{ 
    \begin{array}{ll}
        i \partial_t \Psi = - \partial^2_x \Psi -v(t)\mu(x) \psi_1(t,x) , \\
        \Psi(t,0) = \Psi(t,1)=0,\\
        \Psi(0,x)=\Psi_0.
    \end{array}
\right. 
\end{equation*}
Thus, for all $\Psi_0 \in H^{11}_{(0)}$, $T \mapsto d\F_T(\varphi_1,0).(\Psi_0,0)=e^{-iAT} \Psi_0$ is continuous on $\R$ and $d\F_0(\varphi_1, 0).(\Psi_0, 0)=\Psi_0$. 

\item [$(A_3)$] By the uniqueness result stated in \cref{wp}, one can check that, for all $T_1, T_2 >0$, $\psi_0 \in H^{11}_{(0)}$,  $u \in H^2_0(0,T_1)$ and $v \in H^2_0(0,T_2)$, 
\begin{equation*}
\psi(T_1+T_2; \ u \# v, \ \psi_0)
=
\psi(T_2; \ v, \ \psi(T_1; \ u, \ \psi_0)).
\end{equation*}

\item [$(A_4)$] By  \cite[Prop.\ 4.3]{B21}, when $\mu$ satisfies ($H_{\lin}$), the reachable set of the linearized system around the ground state is given by 
\begin{equation*}
\H
=
\overline{
\Span_{\C}
}
\left(
\psi_j(T);
\
\text{for all }
j 
\in \N^*-\{K\}
\right). 
\end{equation*}
This space doesn't depend on $T$, is closed, and is of codimension 2 in $L^2(0,1)$. 

\item [$(A_5)$] By \cref{prop:TV1} and \cref{prop:TV2}, when $\mu$ satisfies ($H_{\lin}$), ($H_{\Quad}$) and ($H_{\Cub}$), both $i \varphi_K$ and $\varphi_K$ are small-time $H^2_0$-continuously approximately reachable vectors. 
\end{itemize}

By \cref{black_box}, when $\mu$ satisfies (H$_{\reg}$), (H$_{\lin}$), (H$_{\Quad}$) and (H$_{\Cub}$), the Schrödinger equation \eqref{Schrodinger} is $H^2_0$-STLC around the ground state with targets in $H^{11}_{(0)}(0,1)$.

\appendix
\section{Existence of a function $\mu$ satisfying all the hypotheses}
\label{existence_mu}

\begin{rem}
In this appendix, the coefficients $(A^p_K)_{p=1,2,3}$ and $C_K$ respectively defined in \eqref{quad_nul_1}, \eqref{quad_nul_2}, \eqref{quad_non_nul} and \eqref{cub_non_nul} are seen as quadratic or cubic forms with respect to $\mu$. Moreover, the definition given in terms of series can be tricky to use. Thus, we use instead the expressions in terms of Lie brackets given in \cref{rem:lie_brackets}. Computing the Lie brackets, one gets that for all $\mu$ satisfying (H$_{\reg}$), the quadratic (resp.\ cubic) forms $A^1_K$ and $C_K$ are given by
\begin{equation}
\label{LB_A1K}
A^1_K( \mu) 
=
\langle \mu'^2 \varphi_1, \varphi_K \rangle
\quad
\text{ and }
\quad
C_K(\mu)
=
-4
\langle
\mu'^2 \mu'' \varphi_1, \varphi_K
\rangle.
\end{equation}
The similar expression of $A^2_K$ is quite heavy. Computing the associated Lie bracket and then `symmetrizing' the associated quadratic form (see \cite[Proposition A.3]{B21bis} for more details), one gets the existence of a constant $C>0$ such that for all $\mu$ satisfying (H$_{\reg}$), one has
\begin{equation*}
\label{LB_A2K}
| 
A^2_K(\mu)
-
\langle {\mu^{(3)}}^2 \varphi_1, \varphi_K \rangle
|
\ioe 
C \| \mu \|_{H^2(0,1)}^2.
\end{equation*}
\cite[Proposition A.3]{B21bis} also provides a similar approximate expression of $A^3_K$, but it will not be useful.
\end{rem}

\begin{thm}
\label{thm:existence_mu}
Let $K \in \N^*$, $K \soe 2$. There exists $\mu$ satisfying (H$_{\reg}$), (H$_{\lin}$), (H$_{\Quad}$) and (H$_{\Cub}$). 
\end{thm}

\begin{rem}
To prove \cref{thm:existence_mu}, it is enough to prove the existence of a function  $\mu \in H^{11}( (0,1), \R) \cap H^4_0(0,1)$ satisfying
\eqref{lin_nul},
\eqref{quad_nul_1},
\eqref{quad_nul_2},
\eqref{quad_non_nul},
\eqref{cub_non_nul}
and
\begin{align}
\label{supp_mu}
&\supp \mu \subset [0,1), 
\\
\label{cond_bord}
&\mu^{(5)}(0)\neq0, 
\\
\label{lin_non_nul}
&\forall j \in \N^*-\{K\}, 
\quad
\langle \mu \varphi_1, \varphi_j \rangle \neq 0. 
\end{align}
Indeed, when the boundary conditions \eqref{mu_bc} hold, thanks to \eqref{coeff_IPP} of \cref{decay_coeff}, assumption \eqref{H_lin_2} is equivalent to 
\begin{equation*}
\mu^{(5)}(0)
\pm 
\mu^{(5)}(1)
\neq 
0 
\quad 
\text{ and } 
\quad 
\ 
\forall j \in \N^*-\{K\}, 
\ 
\langle \mu \varphi_1, \varphi_j \rangle 
\neq 0.
\end{equation*}
\end{rem}
The proof of \cref{thm:existence_mu} is in four steps.
\begin{itemize}
\item  First, using Baire theorem to deal with the infinite number of non-vanishing conditions, one can find a function $\mu_{\Rref}$ satisfying \eqref{lin_nul}, \eqref{cub_non_nul}, \eqref{cond_bord} and \eqref{lin_non_nul}. Notice that only the non-vanishing condition \eqref{quad_non_nul} is not treated at this stage as the strategy of the two following steps, relying on oscillating functions, would destroy this condition.  
\item Then, using some analyticity and the isolated zeros theorem, one constructs $\hat{\mu}_{\Rref}$ a perturbation of $\mu_{\Rref}$ satisfying \eqref{quad_nul_1} while conserving all the previous properties already satisfied by $\mu_{\Rref}$. 
\item Similarly, one constructs then $\tild{\mu}_{\Rref}$ a perturbation of $\hat{\mu}_{\Rref}$ satisfying \eqref{quad_nul_2} while conserving all the previous properties satisfied by $\hat{\mu}_{\Rref}$. 
\item Finally, using the construction of a quadratic basis, from $\tild{\mu}_{\Rref}$, one constructs a new function satisfying \eqref{quad_non_nul} in addition to the previous conditions. 
\end{itemize}

\begin{proof}[Proof of \cref{thm:existence_mu}]
Let $K \in \N^*$, $K \soe 2$ and $\overline{x} \in (0,1)$ such that $\varphi_K(\overline{x})=0$. As $\varphi_1 >0$ on $(0,1)$ and $\varphi_K'(\overline{x})>0$, one may assume the existence of $\delta>0$ such that $\varphi_1 \varphi_K>0$ on $( \overline{x}, \overline{x}+ \delta)$ and  $\varphi_1 \varphi_K<0$ on $(\overline{x}- \delta, \overline{x})$. Let $\eta \in (0, \overline{x} - \delta)$ such that $\varphi_1 \varphi_K \neq 0$ on $(0, \eta)$. 

\smallskip \noindent \emph{Step 1: Existence of $\mu$ in $H^{11} \cap H^4_0(0,1)$ supported on $ [0, \eta)$ and satisfying  \eqref{lin_nul}, \eqref{cub_non_nul}, \eqref{cond_bord} and \eqref{lin_non_nul}.} In this step, we work with the $H^{11}(0,1)$-topology. 
Denote by 
\begin{align*}
\E 
&:=
\left\{
\mu \in H^{11}( (0,1), \R);
\
\mu \equiv 0 \text{ on } \left[ \frac{\eta}{2}, 1 \right]
\text{ and }
\mu \text{ satisfies } \eqref{lin_nul}
\right\}
\cap
H^4_0(0,1)
, 
\\
\U
&:=
\left\{
\mu \in \E; 
\
\mu \text{ satisfies }
\eqref{cub_non_nul},
\eqref{cond_bord}
\text{ and }
\eqref{lin_non_nul}
\right\}.
\end{align*}
The goal of Step 1 is to prove that $\U$ is not empty. As $\E$ is not empty, it suffices to prove that $\U$ is dense in $\E$. 
Moreover, denoting by 
\begin{equation*}
\Cspace
:=
\left\{
\mu \in \E;
\
C_K(\mu) \neq 0
\right\}
,
\
\V
:=
\left\{
\mu \in \E;
\
\mu^{(5)}(0) \neq 0
\right\}
\text{ and }
\
\U_j
:=
\left\{
\mu \in \E;
\
\langle \mu \varphi_1, \varphi_j \rangle \neq 0
\right\}
,
\end{equation*}
$\U$ is the intersection of all the open subsets $\V$, $\Cspace$ and $\U_j$ for $j \in \N^*-\{K\}$. Thus, as $\E$ is a complete space (because closed in $H^{11}$), by Baire theorem, to prove that $\U$ is dense in $\E$, it suffices to prove that $\V$, $\Cspace$ and $\U_j$ for $j \in \N^*-\{K\}$ are dense in $\E$. The density of $\V$ is clear. Let $j \in \N^*-\{K\}$. 

\smallskip \noindent \emph{$\U_j$ is dense in $\E$.} Let $\mu^*$ in $\E$ such that $\langle \mu^* \varphi_1, \varphi_j \rangle =0$ and let $\eps>0$. As the linear forms $\mu \mapsto \langle \mu \varphi_1, \varphi_K \rangle$ and $\mu \mapsto \langle \mu \varphi_1, \varphi_j \rangle$ are linearly independent (for $j \neq K$), one can find $\nu \in C_c^{\infty}(0, \frac{\eta}{2})$ such that 
$
\langle \nu \varphi_1, \varphi_K \rangle = 0
$
and 
$
\langle \nu \varphi_1, \varphi_j \rangle \neq 0.
$
Then $\mu_{\eps}:= \mu^* + \frac{\eps}{\| \nu \|} \nu$ is in $\U_j$ with $\| \mu_{\eps} - \mu^* \|_{H^{11}} < \eps$.

\smallskip \noindent \emph{$\Cspace$ is dense in $\E$.} Let $\mu^*$ in $\E$ such that $C_K(\mu^*) =0$ and let $\eps>0$. By a similar construction than the one given in \cite[Theorem A.4]{B21bis}, one can find $\nu \in C_c^{\infty}(0, \frac{\eta}{2})$ such that 
$
\langle \nu \varphi_1, \varphi_K \rangle = 0
$
and
$
C_K(\nu) \neq 0.
$
Then, by \eqref{LB_A1K}, $\eps \mapsto C_K ( \mu^* + \frac{\eps}{\| \nu \|} \nu )$ is a polynomial of degree 3 vanishing at zero. Thus, there exists $\eps^* > 0$ such that this polynomial doesn't vanish on $(0, \eps^*)$. Hence, for all $\eps \in (0, \eps^*)$, $\mu_{\eps}:= \mu^* + \frac{\eps}{\| \nu \|} \nu$ is in $\Cspace$ with $\| \mu_{\eps} - \mu^* \|_{H^{11}} < \eps$.

\smallskip \noindent \emph{Step 2: Existence of $\mu$ in $H^{11} \cap H^4_0$ satisfying  \eqref{lin_nul}, \eqref{quad_nul_1}, \eqref{cub_non_nul}, \eqref{supp_mu}, \eqref{cond_bord} and \eqref{lin_non_nul}.}
Let $\mu_{\Rref}$  in $H^{11} \cap H^4_0$ constructed at Step 1, supported on $ [0, \eta)$ and satisfying \eqref{lin_nul}, \eqref{cub_non_nul},  \eqref{cond_bord} and \eqref{lin_non_nul}. The goal of this step is to prove that if $\mu_{\Rref}$ doesn't already satisfy \eqref{quad_nul_1} (we assume that  $A^1_K( \mu_{\Rref}) <0$, the case $A^1_K( \mu_{\Rref}) >0$ is similar), then one can construct a perturbation of $\mu_{\Rref}$ satisfying \eqref{quad_nul_1} while conserving the properties already satisfied by $\mu_{\Rref}$.  To that end, one can consider the following `basis' functions.
\begin{itemize}
\item Let $\mu_0$ in $C_c^{\infty} \left( \overline{x}+ \delta, \frac{1+\overline{x}+\delta}{2} \right)$ such that $\langle \mu_0 \varphi_1, \varphi_K \rangle =1$. 
\item
Let $J^-$ and $J^+$ two open intervals of respectively $(\overline{x} - \delta, \overline{x})$ and $(\overline{x}, \overline{x}+\delta)$. For all $\eps >0$ and $\lambda \neq 0$, we define
\begin{equation}
\label{exp_mu_oscill}
\mu_{\eps, \lambda}(x)
:= 
\sqrt{
\frac{
 \eps | \lambda|
}
{
|\varphi_1(x(\lambda)) \varphi_K(x(\lambda))|
}
} 
\
g 
\left(
\frac{x-x(\lambda)}{\eps} 
\right),
\quad x \in [0,1],
\end{equation}
where $g \in C^{\infty}_c(0,1)$ such that $\int_0^1 g'(y)^2 dy=1$ and $x(\lambda):=x^+ \1_{\lambda>0}+x^- \1_{\lambda<0}$ where $x^{\pm}$ are in $J^{\pm}$ (thus, $\sign( \varphi_K(x^{\pm}))= \pm 1$). Notice that $\mu_{\eps, \lambda}$ is supported on $(x^- , x^- + \eps) \cup (x^+ , x^+ + \eps)$ and thus on  $J^- \cup J^+$ for $\eps$ small enough. Formally, $\mu_{\eps, \lambda}$ is constructed so that $A^1_K( \mu_{\eps, \lambda}) \approx \lambda$.
\end{itemize}
We consider perturbations of $\mu_{\Rref}$ of the following form, 
\begin{equation}
\label{def_nu_eps_lambda}
\nu_{\eps, \lambda}
:=
\mu_{\Rref}
+
\mu_{\eps, \lambda}
-
\langle \mu_{\eps, \lambda} \varphi_1, \varphi_K \rangle \mu_0,
\quad
\eps >0, 
\quad
\lambda \neq 0. 
\end{equation}
Notice that all the functions have disjoint supports (see Figure \ref{supports_Step2}) so that the quadratic and cubic forms can be seen as additive. Moreover, by construction, for all $\eps>0$ and $\lambda \neq 0$, $\nu_{\eps, \lambda}$ already satisfies \eqref{lin_nul}, \eqref{supp_mu} and \eqref{cond_bord}.

\begin{figure}[!h]
\centering
\begin{tikzpicture}
\draw (0,0) -- (1.5,0) ;
\draw (0,-0.1) -- (0, 0.1) ;
\draw (0,-0.01) node[below]{$\scriptstyle{0}$} ;
\draw [line width=0.7mm,color=Magenta] (0,0) -- (1.5, 0) ;
\draw [Magenta] (0.7,-0.01) node[above]{$\scriptstyle{\mu_{\Rref}}$} ;
\draw (1.5,-0.1) -- (1.5, 0.1) ;
\draw (1.5,-0.01) node[below]{$\scriptstyle{\eta}$} ;
\draw (3,-0.1) -- (3, 0.1) ;
\draw (3,-0.01) node[below]{$\scriptstyle{\overline{x}-\delta}$} ;
\draw (1,0) -- (1.8,0) ;
\draw  (1.8,0) -- (2.5,0) ;
\draw[line width=0.7mm,color=Cyan]  (5.2,0) -- (5.7,0) ;
\draw [Cyan] (5.5, 0.01) node[below]{$\scriptstyle J^-$} ;
\draw [Cyan] (5.5,-0.01) node[above]{$\scriptstyle{\mu_{\eps, \lambda}}$} ;
\draw (5.5, -0.1) -- (5.5, 0.1) ;
%
%
\draw  (1.5,0) -- (6,0) ;
%
\draw (6,0) -- (6.5,0) ;
\draw (6.5,-0.1) -- (6.5, 0.1) ;
\draw (6.5,-0.01) node[below]{$\scriptstyle \overline{x}$} ;
\draw (6.5, 0) -- (7,0) ;
%
\draw  (7,0) -- (12,0) ;
%
\draw[line width=0.7mm,color=Cyan]  (8,0) -- (9,0) ;
\draw [Cyan] (8.5, 0.01) node[below]{$\scriptstyle J^+$} ;
\draw [Cyan] (8.5,-0.01) node[above]{$\scriptstyle{\mu_{\eps, \lambda}}$} ;
\draw (8.5, -0.1) -- (8.5, 0.1) ;
%
\draw (12, 0) -- (12.5,0) ;
\draw (10.3,-0.1) -- (10.3, 0.1) ;
\draw (10.3,-0.01) node[below]{$\scriptstyle \overline{x}+\delta$} ;
\draw (12.5, 0) -- (13.5,0) ;
\draw (13.5,-0.1) -- (13.5, 0.1) ;
\draw (13.5,-0.01) node[below]{$\scriptstyle 1$} ;
\draw (11.9,-0.1) -- (11.9, 0.1) ;
\draw (11.9,-0.01) node[below]{$\scriptstyle \frac{1+\overline{x}+\delta}{2}$};
\draw[line width=0.7mm,color=Gray]  (10.5,0) -- (11.5,0) ;
\draw [Gray] (11,-0.01) node[above]{$\scriptstyle{\mu_0}$} ;
%
\end{tikzpicture}
\caption{The supports of the functions used in Step 2 are depicted.
}
\label{supports_Step2}
\end{figure}

\smallskip \noindent \emph{Step 2.1: For all $\eps$ small enough, there exists $\lambda(\eps)>0$ such that $\nu_{\eps, \lambda(\eps)}$ satisfies \eqref{quad_nul_1}.}
The goal is to construct a one-parameter family of functions such that the following quantity vanishes, 
\begin{equation}
\label{def_Q}
Q( \eps, \lambda) 
:=
A^1_K 
(
\nu_{\eps, \lambda}
)
=
A^1_K
(
\mu_{\Rref}
)
+
A^1_K
(
\mu_{\eps, \lambda}
)
+
\langle \mu_{\eps, \lambda} \varphi_1, \varphi_K \rangle ^2 
A^1_K
(
\mu_0
).
\end{equation}

\smallskip \noindent \emph{Regularity of $Q$.}
Looking at \eqref{exp_mu_oscill}, one could fear some lack of regularity for $Q$ with respect to $\lambda$. However, as $A^1_K(\mu_{\Rref})<0$, one only needs to study $Q$ on $(\R^*_+)^2$.  
Moreover, substituting the expression \eqref{exp_mu_oscill} and performing the change of variables $x=\eps y + x^+$, one has, for all $\eps >0$ and $\lambda >0$, 
\begin{multline}
\label{exp_fK}
\langle \mu_{\eps, \lambda} \varphi_1, \varphi_K \rangle 
=
\sqrt{
\frac{
\eps \lambda
}
{
\varphi_1( x^+) \varphi_K(x^+)
}
}
\int_{x^+}^{x^+ + \eps} 
g 
\left(
\frac{x-x^+}{\eps} 
\right)
\varphi_1( x)
\varphi_K(x)
dx
\\
=
\eps^{3/2}
\sqrt{
\frac{
 \lambda
}
{
\varphi_1( x^+) \varphi_K(x^+)
}
}
\int_0^1 
g(y)
\varphi_1( \eps y + x^+)
\varphi_K( \eps y + x^+)
dy.
\end{multline}
 Thus, the map 
$
(\eps, \lambda) 
\mapsto 
\langle \mu_{\eps, \lambda} \varphi_1, \varphi_K \rangle 
$
is analytic on $(\R^*_+)^2$. Similarly, using the computation of $A^1_K$ given in \eqref{LB_A1K}, one gets, for all $\eps>0$ and $\lambda>0$,
\begin{equation}
\label{exp_A1K}
A^1_K
(
\mu_{\eps, \lambda} 
)
=
\frac{
\lambda
}
{
\varphi_1( x^+) \varphi_K(x^+)
}
\int_0^1 
g'(y)^2 
\varphi_1( \eps y + x^+)
\varphi_K( \eps y + x^+)
dy.
\end{equation}
Hence, the map $(\eps, \lambda) \mapsto A^1_K( \mu_{\eps, \lambda})$ is also analytic on $(\R^*_+)^2$. Thus, $Q$ is analytic on $(\R^*_+)^2$. 

\smallskip \noindent \emph{For $\eps$ small enough, $Q(\eps, \cdot)$ can take both signs.}
Doing a Taylor expansion with respect to $\eps$ of \eqref{exp_A1K}, one gets the existence of $C>0$ such that for all $\eps>0$ and $\lambda>0$, 
\begin{equation}
\label{estim_A1K}
\left|
A^1_K
(
\mu_{\eps, \lambda} 
)
- 
\lambda
\right| 
\ioe 
C
\lambda \eps. 
\end{equation}
Thus, \eqref{def_Q}, \eqref{exp_fK} and \eqref{estim_A1K}  lead to the existence of $C>0$ such that for all $\eps \in (0,1)$ and $\lambda >0$, 
\begin{equation}
\label{estim_Q}
| Q( \eps, \lambda) - A^1_K( \mu_{\Rref}) - \lambda| \ioe C \lambda \eps.
\end{equation}
Thus, there exists $\eps^*>0$ such that for all $\eps \in (0, \eps^*)$, 
\begin{equation*}
Q
\left(
\eps, 
- 
\frac{
A^1_K( \mu_{\Rref})
}
{
2
}
\right)
<0
\quad 
\text{ and }
\quad 
Q
\left(
\eps, 
- 
\frac{
3 A^1_K( \mu_{\Rref})
}
{
2
}
\right)
>0.
\end{equation*}

\smallskip \noindent \emph{For $\eps$ small enough, $Q(\eps, \cdot)$ is increasing.}
Differentiating \eqref{exp_fK} and \eqref{exp_A1K} with respect to $\lambda$ and performing again an expansion with respect to $\eps$ in the spirit of \eqref{estim_A1K}, one gets the existence of $C>0$ such that for all $\eps>0$ and $\lambda>0$, 
$
\left|
\partial_{\lambda} Q( \eps, \lambda) -1 
\right|
\ioe 
C \eps.
$
Thus, for $\eps>0$ small enough, $\partial_{\lambda} Q(\eps, \cdot)$ is positive. 

\smallskip \noindent \emph{Conclusion.}
Applying the intermediate value theorem, one gets for all $\eps \in (0, \eps^*)$ the existence of 
$\lambda=\lambda(\eps) 
\in 
\left(
-
\frac{
A^1_K( \mu_{\Rref})
}
{
2
},
-
\frac{
3 A^1_K( \mu_{\Rref})
}
{
2
}
\right)
$ 
such that $Q(\eps, \lambda(\eps))=0$, meaning that $\nu_{\eps, \lambda(\eps)}$ defined in \eqref{def_nu_eps_lambda} satisfies \eqref{quad_nul_1} by definition \eqref{def_Q} of $Q$. 

\smallskip \noindent \emph{Step 2.2: The map $\eps \mapsto \lambda(\eps)$ is continuous.}
As for all $\eps \in (0, \eps^*)$, $Q(\eps, \lambda(\eps))=0$, recalling the definition \eqref{def_Q} of $Q$, one has
\begin{equation*}
\lambda(\eps)
= 
\lambda(\eps) 
- 
A^1_K(\mu_{\eps, \lambda(\eps)}) 
- 
A^1_K( \mu_{\Rref}) 
- 
\langle \mu_{\eps, \lambda(\eps)} \varphi_1, \varphi_K \rangle ^2 
A^1_K(\mu_0)
=:
\lambda(\eps)H(\eps) - A^1_K( \mu_{\Rref}).
\end{equation*}
Using \eqref{exp_fK} and \eqref{exp_A1K}, $H$ is analytic on $(0, \eps^*)$ with $|H(\eps)| \ioe C \eps$  for all $\eps \in (0, \eps^*)$. Thus, for all $\eps$ and $\eps_0$ in $(0, \eps^*)$, 
\begin{equation*}
\left|
\lambda(\eps)
- \lambda(\eps_0)
\right|
%
\ioe  
C \eps
| \lambda( \eps) - \lambda(\eps_0) | 
+ 
| \lambda(\eps_0) | 
|H(\eps) - H (\eps_0) |.
\end{equation*}
Reducing $\eps^*$ if needed, 
the continuity of $\eps \mapsto \lambda(\eps)$ on $(0, \eps^*)$ stems from the one of $H$. 

\smallskip \noindent \emph{Step 2.3: The map $\eps \mapsto \lambda(\eps)$ is analytic.}
Let $\eps_0 \in (0, \eps^*)$. By construction, $Q(\eps_0, \lambda(\eps_0))=0.$  Besides, in Step 2.1, we proved that $\partial_{\lambda} Q( \eps_0, \lambda(\eps_0)) > 0$ and  that $Q$ is analytic on $(0, \eps^*) \times \R^*_+$. Hence, by the implicit function theorem, there exists an open neighborhood $\mathcal{V}$ of $\eps_0$, an open neighborhood $\mathcal{W}$ of $\lambda(\eps_0)$ and an analytic function $\Lambda :  \mathcal{V} \rightarrow \mathcal{W}$
such that 
\begin{equation*}
\left( 
\eps \in \mathcal{V}, \ \lambda \in \mathcal{W} \text{ and } Q( \eps, \lambda) =0 
\right)
\quad
\Leftrightarrow 
\quad
\left(
\eps \in \mathcal{V} \text{ and } \lambda=\Lambda( \eps)
\right)
\end{equation*}
As $\eps \mapsto \lambda(\eps)$ is continuous, locally $\lambda= \Lambda$ and thus, $\eps \mapsto \lambda(\eps)$ is analytic on $(0, \eps^*)$. 

\smallskip \noindent \emph{Step 2.4: There exists $\eps$ such that $\nu_{\eps, \lambda(\eps)}$ satisfies \eqref{lin_non_nul} and \eqref{cub_non_nul}.} By a similar computation than the one in \eqref{exp_fK}, one gets, for all $j \in \N^*$, $\eps >0$ and $\lambda>0$, 
\begin{multline*}
\langle \nu_{\eps, \lambda} \varphi_1, \varphi_j \rangle 
=
\langle \mu_{\Rref} \varphi_1, \varphi_j \rangle 
+
\eps^{3/2}
\sqrt{
\frac{
 \lambda
}
{
\varphi_1( x^+) \varphi_K(x^+)
}
}
\int_0^1 
g(y)
\varphi_1( \eps y + x^+)
\varphi_j( \eps y + x^+)
dy
\\
-
\langle \mu_{\eps, \lambda} \varphi_1, \varphi_K \rangle 
\langle \mu_0 \varphi_1, \varphi_j \rangle.
\end{multline*}
As 
$(\eps, \lambda) \mapsto  \langle \mu_{\eps, \lambda} \varphi_1, \varphi_K \rangle$
is analytic on $(\R^*_+)^2$
(see Step 2.1)
and 
$\eps \mapsto \lambda(\eps)$ 
is analytic on $(0, \eps^*)$
(see Step 2.3)
,
for all $j \in \N^*-\{K\}$, 
the map 
$
\eps 
\mapsto 
\langle \nu_{\eps, \lambda(\eps)} \varphi_1, \varphi_j \rangle 
$ 
is analytic on $(0, \eps^*)$. It can also be extended by continuity at zero with the value $\langle \mu_{\Rref} \varphi_1, \varphi_j \rangle \neq 0$ by construction of $\mu_{\Rref}$. Similarly, using \eqref{LB_A1K}, $\eps \mapsto C_K(\nu_{\eps, \lambda(\eps)})$ is analytic on $(0, \eps^*)$ and can be extended continuously at zero with the value $C_K(\mu_{\Rref}) \neq 0$. Thus, the functions 
$
( 
\eps 
\mapsto 
\langle \nu_{\eps, \lambda(\eps)} \varphi_1, \varphi_j \rangle
)_{j \in \N^*- \{K\}}
$ 
and 
$
\eps \mapsto C_K(\nu_{\eps, \lambda(\eps)})
$
are analytic and non-zero on $(0, \eps^*)$. Hence, by the isolated zeros theorem, there exists $\eps \in (0, \eps^*)$, such that for all $j \in \N^*- \{K \}$, $\langle \nu_{\eps, \lambda(\eps)} \varphi_1, \varphi_j \rangle \neq 0$ and $C_K(\nu_{\eps, \lambda(\eps)})\neq0$, meaning that  $\nu_{\eps, \lambda(\eps)}$ satisfies \eqref{cub_non_nul} and \eqref{lin_non_nul}.

\smallskip \noindent \emph{Step 3: Existence of $\mu$ in $H^{11}-H^4_0$ satisfying \eqref{lin_nul}, \eqref{quad_nul_1}, \eqref{quad_nul_2}, \eqref{cub_non_nul}, \eqref{supp_mu}, \eqref{cond_bord} and \eqref{lin_non_nul}.} The proof of Step 3 is quite similar to the one of Step 2. Let $\hat{\mu}_{\Rref}$ constructed at Step 2 satisfying \eqref{lin_nul}, \eqref{quad_nul_1}, \eqref{cub_non_nul}, \eqref{supp_mu}, \eqref{cond_bord} and \eqref{lin_non_nul}. As in Step 2, the goal is to prove that if $\hat{\mu}_{\Rref}$ doesn't already satisfy \eqref{quad_nul_2} (we assume that  $A^2_K( \hat{\mu}_{\Rref}) <0$), then one can construct a perturbation of $\hat{\mu}_{\Rref}$ satisfying \eqref{quad_nul_2} while conserving the properties already satisfied by $\hat{\mu}_{\Rref}$. Let $\hat{J}^{+}$ and $I^{+}$ (resp.\ $\hat{J}^{-}$ and $I^{-}$) open disjoint intervals of $(\overline{x}- \delta, \overline{x}) - J^-$ (resp.\ $(\overline{x}, \overline{x}+ \delta) - J^+$). In this step, we consider the following new `basis' functions. 
\begin{itemize}
%
\item Let $\hat{\mu}_0$ in $C_c^{\infty}\left(\frac{1+\overline{x}+\delta}{2}, 1 \right)$ such that $\langle \hat{\mu}_0 \varphi_1, \varphi_K \rangle =1$ and $A^1_K( \hat{\mu}_0)=0$. 
\item By \cite[Theorem A.4]{B21bis}, there exists $\mu_1^{\pm}$ in $C_c^{\infty}(I^{\pm})$ such that
$
\langle 
\mu_1^{\pm}
\varphi_1,
\varphi_K
\rangle=0
$
and
$ 
A^1_K( \mu_1^{\pm}) = \pm 1. 
$
\item For all $\eps >0$ and $\lambda \neq 0$, we define
\begin{equation*}
\hat{\mu}_{\eps, \lambda}(x)
:= 
\eps^{5/2}
\sqrt{
\frac{
| \lambda|
}
{
|\varphi_1(x(\lambda)) \varphi_K(x(\lambda))|
}
} 
\
g 
\left(
\frac{x-x(\lambda)}{\eps} 
\right),
\end{equation*}
where $g \in C^{\infty}_c(0,1)$ such that $\int_0^1 g^{(3)}(y)^2 dy=1$ and $x(\lambda):=\hat{x}^+ \1_{\lambda>0}+\hat{x}^- \1_{\lambda<0}$, where $\hat{x}^{\pm}$ are in $\hat{J}^{\pm}$. Notice that for $\eps$ small enough, the support of $\hat{\mu}_{\eps, \lambda}$ is in $\hat{J}^- \cup \hat{J}^+$. Formally, this time, $\hat{\mu}_{\eps, \lambda}$ is constructed so that $A^2_K( \hat{\mu}_{\eps, \lambda}) \approx \lambda$. 
\end{itemize}
In this step, we consider perturbations of $\hat{\mu}_{\Rref}$ of the form, 
\begin{equation*}
\hat{\nu}_{\eps, \lambda}
:=
\hat{\mu}_{\Rref}
+
\hat{\mu}_{\eps, \lambda}
-
\langle
\hat{\mu}_{\eps, \lambda} 
\varphi_1,
\varphi_K
\rangle
\hat{\mu}_0
+
\sqrt{
\left| 
A^1_K( \hat{\mu}_{\eps, \lambda}) 
\right|
}
\mu_1^{
-\sign(
A^1_K( \hat{\mu}_{\eps, \lambda})
)
},
\quad 
\eps >0, 
\quad 
\lambda \neq 0.
\end{equation*}
Once again, we made sure that all the functions considered have disjoint supports (see Figure \ref{supports_Step3}) so that the quadratic and cubic forms can be seen as additive. 

\begin{figure}[!h]
\centering
\begin{tikzpicture}
\draw (0,-0.1) -- (0, 0.1) ;
\draw (0,-0.01) node[below]{$\scriptstyle 0$} ;
\draw (0,0) -- (1,0) ;
\draw [line width=0.7mm,color=Brown] (0,0) -- (1.5, 0) ;
\draw [Brown] (0.7,-0.01) node[above]{$\scriptstyle{\hat{\mu}_{\Rref}}$} ;
\draw (1.5,-0.1) -- (1.5, 0.1) ;
\draw(1.5,-0.01) node[below]{$\scriptstyle \eta$} ;
\draw (3,-0.1) -- (3, 0.1) ;
\draw (3,-0.01) node[below]{$\scriptstyle \overline{x}-\delta$} ;
\draw (1,0) -- (1.5,0) ;
\draw  (1.8,0) -- (2.5,0) ;
\draw[line width=0.7mm,color=Green]  (3.8,0) -- (4.5,0) ;
\draw [Green] (4.2, 0.01) node[below]{$\scriptstyle I^-$} ;
\draw [Green] (4.2,-0.01) node[above]{$\scriptstyle{\mu_1^{\pm}}$} ;
\draw[line width=0.7mm,color=Brown]  (5.2,0) -- (5.7,0) ;
\draw [Brown] (5.4, 0.01) node[below]{$\scriptstyle J^-$} ;
\draw [Brown] (5.4,-0.01) node[above]{$\scriptstyle{\hat{\mu}_{\Rref}}$} ;
\draw (5.5, -0.1) -- (5.5, 0.1) ;
%
\draw[line width=0.7mm,color=Red]  (5.9,0) -- (6.2,0) ;
\draw (6.1, -0.1) -- (6.1, 0.1) ;
\draw [Red] (6.1, 0.01) node[below]{$\scriptstyle \hat{J}^-$} ;
\draw [Red] (6.1,-0.01) node[above]{$\scriptstyle{\hat{\mu}_{\eps, \lambda}}$} ;
%
\draw  (1.5,0) -- (6,0) ;
%
\draw (6,0) -- (6.5,0) ;
\draw (6.5,-0.1) -- (6.5, 0.1) ;
\draw (6.5,-0.01) node[below]{$\scriptstyle \overline{x}$} ;
\draw (6.5, 0) -- (7,0) ;
%
\draw  (7,0) -- (12,0) ;
%
\draw[line width=0.7mm,color=Red]  (6.7,0) -- (7.2,0) ;
\draw [Red] (7, 0.01) node[below]{$\scriptstyle \hat{J}^+$} ;
\draw [Red] (7,-0.01) node[above]{$\scriptstyle{\hat{\mu}_{\eps, \lambda}}$} ;
\draw (7, -0.1) -- (7, 0.1) ;
\draw[line width=0.7mm,color=Green]  (9.5,0) -- (10,0) ;
\draw [Green] (9.7, 0.01) node[below]{$\scriptstyle I^+$} ;
\draw [Green] (9.8,-0.01) node[above]{$\scriptstyle{\mu_1^{\pm}}$} ;
\draw[line width=0.7mm,color=Brown]  (8,0) -- (9,0) ;
\draw [Brown] (8.5, 0.01) node[below]{$\scriptstyle J^+$} ;
\draw [Brown] (8.5,-0.01) node[above]{$\scriptstyle{\hat{\mu}_{\Rref}}$} ;
\draw (8.5, -0.1) -- (8.5, 0.1) ;
%
\draw (12, 0) -- (12.5,0) ;
\draw (10.3,-0.1) -- (10.3, 0.1) ;
\draw (10.3,-0.01) node[below]{$\scriptstyle \overline{x}+\delta$} ;
\draw (12.5, 0) -- (13.5,0) ;
\draw (13.5,-0.1) -- (13.5, 0.1) ;
\draw (13.5,-0.01) node[below]{$\scriptstyle 1$} ;
\draw (11.9,-0.1) -- (11.9, 0.1) ;
\draw (11.9,-0.01) node[below]{$\scriptstyle \frac{1+\overline{x}+\delta}{2}$};
\draw[line width=0.7mm,color=Brown]  (10.5,0) -- (11.5,0) ;
\draw [Brown] (11,-0.01) node[above]{$\scriptstyle{\hat{\mu}_{\Rref}}$} ;
\draw[line width=0.7mm,color=Blue]  (12.3,0) -- (13.2,0) ;
\draw [Blue] (12.8,-0.01) node[above]{$\scriptstyle{\hat{\mu}_0}$} ;
\end{tikzpicture}
\caption{The supports of the functions used in Step 3 are depicted.
}
\label{supports_Step3}
\end{figure}

\noindent Moreover, by construction, for all $(\eps, \lambda) \in \R^*_+ \times \R^*$, $\hat{\nu}_{\eps, \lambda}$ already satisfies \eqref{lin_nul}, \eqref{quad_nul_1}, \eqref{supp_mu} and \eqref{cond_bord}. Then, define 
\begin{multline}
\label{def_Q_hat}
\hat{Q}( \eps, \lambda) 
:=
A^2_K
(
\hat{\nu}_{\eps, \lambda}
)
=
A^2_K
(
\hat{\mu}_{\Rref}
)
+
A^2_K
(
\hat{\mu}_{\eps, \lambda}
)
+
\langle
\hat{\mu}_{\eps, \lambda} 
\varphi_1,
\varphi_K
\rangle^2
A^2_K(
\hat{\mu}_0
)
\\
+
| A^1_K( \hat{\mu}_{\eps, \lambda}) |
A^2_K
(
\mu_1^{
-\sign(
A^1_K( \hat{\mu}_{\eps, \lambda})
)
}
).
\end{multline}
 The end of Step 3 is exactly the same as the one of Step 2. 
%
%
%
%
%
%
%
%
%
%
\begin{itemize}
\item Applying the intermediate value theorem to $\hat{Q}(\eps, \cdot)$, one proves the existence of a continuous map $\eps \mapsto \lambda(\eps)$ such that for all $\eps$ small enough, $\hat{Q}(\eps, \lambda(\eps))=0$ and thus, $\hat{\nu}_{\eps, \lambda(\eps)}$ satisfies \eqref{quad_nul_2}. Notice that because of the last term in \eqref{def_Q_hat},
one could fear some lack of regularity of $\hat{Q}$. However, as $\lambda \mapsto A^1_K( \hat{\mu}_{\eps, \lambda})$ is continuous on $\R^*_+$ with a computation similar to \eqref{exp_A1K}, 
$\sign(
A^1_K( \hat{\mu}_{\eps, \lambda})
)$
is locally constant around $\lambda=- A^2_K( \hat{\mu}_{\Rref})$. 
\item Then, one uses the implicit function theorem to get that $\eps \mapsto \lambda(\eps)$ is analytic and thus, with the isolated zeros theorem, to get the existence of an $\eps$ such that $\hat{\nu}_{\eps, \lambda(\eps)}$ satisfies \eqref{cub_non_nul} and \eqref{lin_non_nul}. 
\end{itemize}

\smallskip \noindent \emph{Step 4: Existence of $\mu$ in $H^{11} \cap H^4_0$ satisfying 
\eqref{lin_nul},
\eqref{quad_nul_1},
\eqref{quad_nul_2},
\eqref{quad_non_nul},
\eqref{cub_non_nul},
\eqref{supp_mu},
\eqref{cond_bord} 
and
\eqref{lin_non_nul}
.} 
Let $\tild{\mu}_{\Rref}$ constructed at Step 3 satisfying \eqref{lin_nul}, \eqref{quad_nul_1}, \eqref{quad_nul_2}, \eqref{cub_non_nul}, \eqref{supp_mu}, \eqref{cond_bord} and \eqref{lin_non_nul}. Assume that $A^3_K( \tild{\mu}_{\Rref}) = 0$, otherwise $\tild{\mu}_{\Rref}$  already satisfies \eqref{quad_non_nul}. Let $\tild{J}^-$ and $\tild{J}^+$ two open disjoint intervals of respectively $(\overline{x}- \delta, \overline{x}) - (J^- \cup \hat{J}^- \cup I^-)$ and $(\overline{x}, \overline{x}+ \delta) - (J^+ \cup \hat{J}^+ \cup I^+)$. By 
\cite[Theorem A.4]{B21bis}, there exists $\nu$ in $C_c^{\infty}(\tild{J}^{\pm})$ such that 
\begin{equation*}
\langle
\nu
\varphi_1, 
\varphi_K
\rangle
=
A^1_K( \nu) 
= 
A^2_K(\nu)
=0 
\quad
\text{ and }
\quad 
A^3_K(\nu)=1.
\end{equation*}
Define for all $\eps \in \R$, 
$
\nu_{\eps} := \tild{\mu}_{\Rref} + \eps \nu.
$
By construction, for all $\eps \in \R^*$, $\nu_{\eps}$ satisfies \eqref{lin_nul}, \eqref{quad_nul_1}, \eqref{quad_nul_2}, \eqref{quad_non_nul}, \eqref{supp_mu} and \eqref{cond_bord} because the functions have disjoint supports. Moreover, the maps $\eps \mapsto C_K( \nu_{\eps})$ and $\eps \mapsto \langle \nu_{\eps} \varphi_1, \varphi_j \rangle$ for all $j \in \N^* - \{K \}$ are polynomial, so analytic and non-vanishing at zero by construction of $\tild{\mu}_{\Rref}$. So, by the isolated zeros theorem, there exists $\eps \in \R^*$ such that $\nu_{\eps} $ satisfies \eqref{cub_non_nul} and \eqref{lin_non_nul}.
\end{proof}

\textbf{Acknowledgments.} The author would like to thank Karine Beauchard and Frédéric Marbach (École Normale Supérieure de Rennes) for having interested her in this problem, for many fruitful discussions and helpful advices. 

\bibliography{Master_Biblio}
\bibliographystyle{plain}

\end{document}